\documentclass{amsart}

\usepackage{graphicx}
\usepackage{bbm}

\usepackage[margin=1in]{geometry}

\usepackage[all]{xy}

\usepackage{multirow}
\usepackage{tikz-cd}

\newtheorem{theorem}{Theorem}[section]
\newtheorem{corollary}[theorem]{Corollary}
\newtheorem{lemma}[theorem]{Lemma}

\theoremstyle{definition}

\newtheorem{remark}[theorem]{Remark}

\DeclareMathOperator{\Khr}{Khr}

\newcommand{\RP}{\mathbbm{RP}}

\newcommand{\calF}{{\mathcal F}}

\newcommand{\partialbar}{\overline{\partial}}

\newcommand{\id}{\mathbbm{1}}

\newcommand{\Tbar}{{\overline{T}}}

\newcommand{\Cbar}{{\overline{C}}}

\newcommand{\sigmabar}{{\overline{\sigma}}}
\newcommand{\Ttilde}{{\tilde{T}}}
\newcommand{\alphabar}{{\overline{\alpha}}}
\newcommand{\betabar}{{\overline{\beta}}}

\newcommand{\ptb}{L}
\newcommand{\ptc}{R}
\newcommand{\pto}{C}

\DeclareMathOperator{\etadot}{\dot{\eta}}
\DeclareMathOperator{\epsilondot}{\dot{\epsilon}}

\DeclareMathOperator{\F}{{\mathbbm{F}}}

\begin{document}

\title{Khovanov homology via 1-tangle diagrams in the annulus}

\begin{abstract}
We show that the reduced Khovanov homology of an oriented link $L$ in
$S^3$ can be expressed as the homology of a chain complex
constructed from a description of $L$ as the closure of a 1-tangle
diagram $T$ in the annulus.
Our chain complex is constructed using a cube of resolutions of
$T$ in a manner similar to ordinary Khovanov homology, but it is
typically smaller than the ordinary Khovanov chain complex and has
several unusual features, such as
\emph{long differentials} corresponding to pairs of successive saddles
in the cube of resolutions.
Our chain complex carries a natural filtration, which we use to
construct a spectral sequence that converges to reduced Khovanov
homology.
Our results are part of a larger program to construct an analog
of Khovanov homology for links in lens spaces by generalizing a
symplectic interpretation of Khovanov homology due to Hedden, Herald,
Hogancamp, and Kirk, and our chain complex was predicted by this
program for the case when the lens space is $S^3$.
\end{abstract}

\author{David Boozer} 

\date{\today}

\maketitle

\section{Introduction}

Khovanov homology is a powerful invariant defined for oriented links
in $S^3$ \cite{Khovanov}.
The Khovanov homology of an oriented link is a finitely-generated
bigraded abelian group that categorifies its Jones
polynomial \cite{Jones}; roughly speaking, the relationship between
the Khovanov homology of a link and its Jones polynomial is analogous
to the relationship between the singular homology of a topological
space and its Euler characteristic.
The Jones polynomial of a link can be recovered from its Khovanov
homology, but the Khovanov homology generally contains more
information: it can sometimes distinguish between links with the same
Jones polynomial, and Khovanov homology detects the unknot
\cite{Kronheimer-2}, but it is not known whether the same is true of
the Jones polynomial.

Here we construct a new chain complex for the reduced Khovanov
homology of a link via a description of the link as the closure of an
oriented 1-tangle diagram $T$ in the annulus $S^1 \times [0,1]$.
We choose a basepoint $b_0 \in S^1$ and take the boundary of $T$ to be
the points $(b_0,0)$ and $(b_0,1)$ on the inner and outer bounding
circles of the annulus.
We close $T$ with an unknotted overpass or underpass arc to obtain
link diagrams $T^+$ and $T^-$.
These link diagrams describe oriented links in $S^3$ that are unique
up to isotopy, which for simplicity we also denote by $T^+$ and $T^-$.
We construct a chain complex $(C_{T^\pm},\partial_{T^\pm})$ for the
link $T^\pm$ from a cube of resolutions of $T$.
An example tangle diagram $T$ and its cube of resolutions are shown in
Figure \ref{fig:example-T}.
For this example the link diagrams $T^+$ and $T^-$ describe the unknot
and right trefoil, respectively.

\begin{figure}
  \centering
  \includegraphics[scale=0.35]{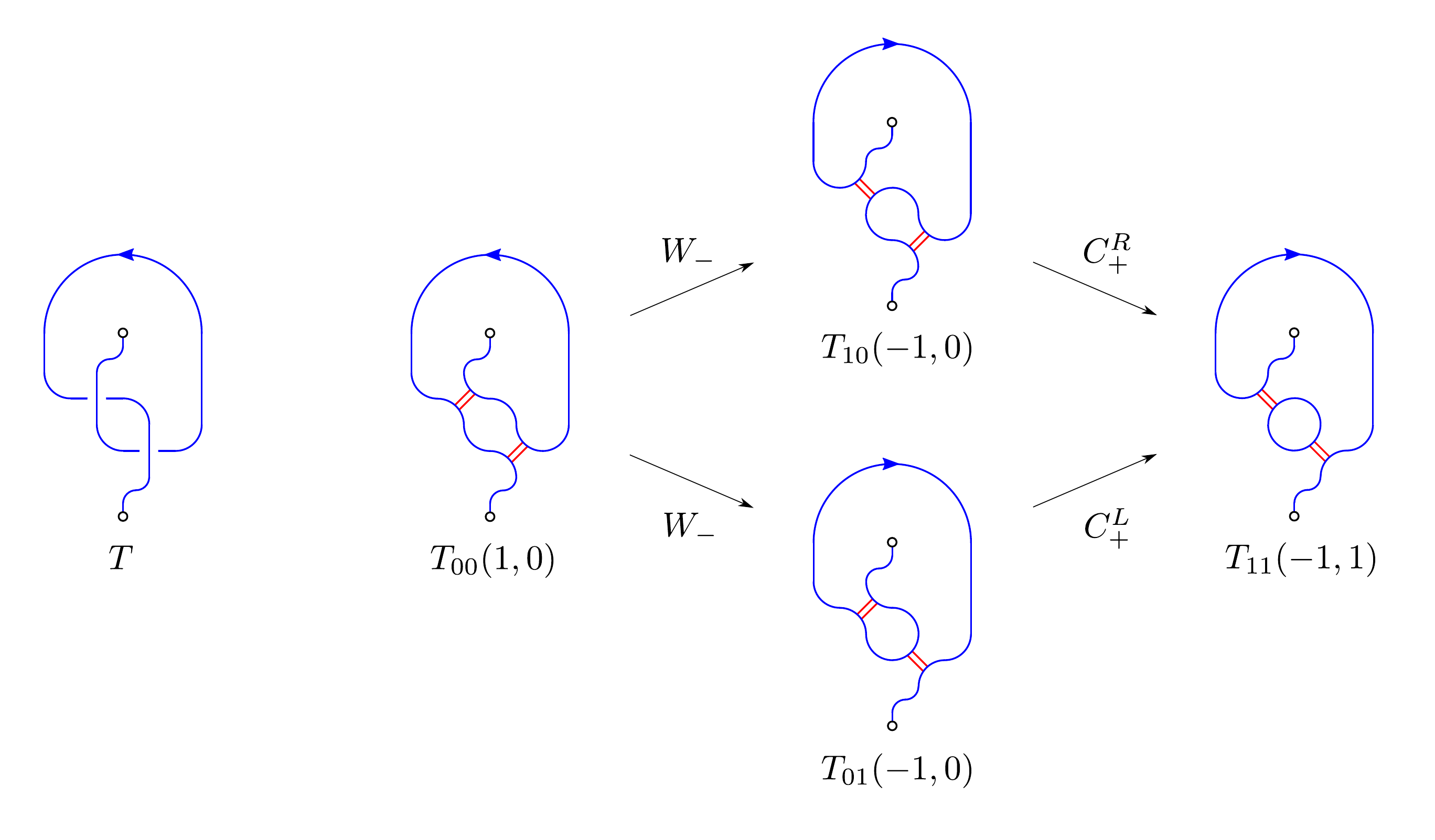}
\caption{
\label{fig:example-T}
Example tangle diagram $T$ and its cube of resolutions.
For each planar tangle $T_i$ in the cube of resolutions, we indicate
its winding number $n_i$ and circle number $r_i$ as $T_i(n_i,r_i)$.
The saddles in the cube of resolutions are labeled by their types, as
described in Section \ref{sec:chain-complex}.
}
\end{figure}

The vertices of the cube of resolutions of $T$ are planar tangles.
Each planar tangle consists of an arc component, which we orient in
the \emph{outward} direction from $(b_0,0)$ to $(b_0,1)$, and some
number of circle components.
We assign a vector space to each planar tangle based on its number of
circle components in a manner analogous to the usual Khovanov chain
complex, and we take the direct sum of these vector spaces to
construct the vector space $C_{T^\pm}$.

The edges of the cube of resolutions of $T$ are saddles between pairs
of planar tangles.
The image of the arc component of a planar tangle under the
projection $S^1 \times [0,1] \rightarrow S^1$ is an oriented loop
based at $b_0$, and we define the \emph{winding number} of the planar
tangle to be the number of times this loop winds counterclockwise
around $S^1$.
A saddle either preserves the winding number or changes it by two.

To each saddle that preserves the winding number, we assign the same
linear map as in ordinary reduced Khovanov homology, where the arc
component plays the role of the marked circle component.
We sum the linear maps for all saddles that preserve the winding
number to obtain a differential $\partial_T^0$ that squares to zero.
In Figure \ref{fig:example-T}, the saddles
$T_{10} \rightarrow T_{11}$ and $T_{01} \rightarrow T_{11}$ preserve
the winding number and hence contribute to $\partial_T^0$.

We next assign a linear map that we call a \emph{long differential} to
each pair of successive saddles for which one saddle lowers the
winding number by two for $T^+$ or raises it by two for $T^-$, and one
saddle connects a circle to the \emph{left} side of the arc component,
given its outward orientation.
We sum the linear maps for all such pairs of saddles to obtain a
differential $\partial_T^\pm$ for $T^\pm$.
In Figure  \ref{fig:example-T},
the saddle $T_{00} \rightarrow T_{01}$ lowers the winding number by
two and the saddle $T_{01} \rightarrow T_{11}$ splits a circle from
the left side of the arc component, so this pair of saddles
contributes to $\partial_T^+$.
There are no saddles that raise the winding number by two, so
$\partial_T^- = 0$.

The total differential $\partial_{T^\pm}$ is then given by
\begin{align*}
  \partial_{T^\pm} = \partial_T^0 + \partial_T^\pm.
\end{align*}
We prove:

\begin{theorem}
\label{theorem:main-intro}
The chain complex $(C_{T^\pm},\partial_{T^\pm})$ is homotopy
equivalent to the chain complex for the reduced Khovanov homology of
the link diagram $T^\pm$.
\end{theorem}

\begin{corollary}
The homology of the chain complex $(C_{T^\pm},\partial_{T^\pm})$ is
the reduced Khovanov homology of the link $T^\pm$.
\end{corollary}

If the planar tangles in the cube of resolutions of $T$ all have
winding number zero, then it is possible to close $T$ with an arc that
does not cross $T$, and for such an arc
$(C_{T^\pm}, \partial_{T^\pm}) = (C_{T^\pm}, \partial_T^0)$ is the usual
Khovanov chain complex for the link diagram $T^\pm$.
But in general our chain complex is smaller than the usual chain
complex, since crossings between $T$ and the closing arc are not
included in the cube of resolutions.
For example, a right trefoil has crossing number three, but we can
compute its reduced Khovanov homology using the tangle diagram shown
in Figure \ref{fig:example-T}, which has only two crossings.

We also prove:

\begin{theorem}
\label{theorem:spectral-intro}
There is a spectral sequence with $E_2$ page given by the homology of
$(C_{T^\pm},\partial_T^0)$ and differential $d_2$ induced by
$\partial_T^\pm$ that converges to the reduced Khovanov
homology of the link $T^\pm$.
\end{theorem}

The vector spaces $C_{T^+}$ and $C_{T^-}$ differ only in their
bigrading structure, so if we ignore the bigradings the $E_2$ pages of
the spectral sequences for $T^+$ and $T^-$ are the same.

Khovanov homology was originally defined only for links in $S^3$, and
an important open problem is to generalize Khovanov homology to links
in arbitrary 3-manifolds.
Khovanov homology has been extended to links in $I$-bundles over
arbitrary surfaces by Asaeda, Przytycki, and Sikora \cite{Asaeda-1},
to links in $S^2 \times S^1$ by Rozansky \cite{Rozansky}, to links in
$\RP^3$ by Gabrov\v{s}ek \cite{Gabrovsek}, and to links in all
connected sums of $S^2 \times S^1$ by Willis \cite{Willis}.
The results presented here are part of a larger program described in
\cite{Boozer-Fukaya}
to construct Khovanov homology for links in arbitrary lens spaces by
generalizing a symplectic interpretation of Khovanov homology due to
Hedden, Herald, Hogancamp, and Kirk \cite{Hedden-3}.
This interpretation originated as an offshoot of their project to
construct \emph{pillowcase homology} \cite{Hedden-1,Hedden-2}, a
symplectic counterpart to Kronheimer and Mrowka's singular instanton
link homology \cite{Kronheimer-2,Kronheimer-1,Kronheimer-3}.

The setup for pillowcase homology is as follows.
Given an oriented link $L$ in $S^3$, one considers a Heegaard
splitting
\begin{align}
  \label{eqn:split-hhhk}
  (S^3,L) = (B^3,T_0) \cup_{(S^2,4)} (B^3,T_1),
\end{align}
where the Heegaard surface $(S^2,4)$ is a 2-sphere that
transversely intersects $L$ in four points, the handlebodies
$(B^3,T_0)$ and $(B^3,T_1)$ are closed 3-balls containing 2-tangles
$T_0$ and $T_1$, and $T_0$ is trivial.
To the Heegaard surface $(S^2,4)$ one associates the irreducible locus
$R^*(S^2,4)$ of the traceless $SU(2)$ character variety for the
2-sphere with four punctures, a symplectic manifold known as the
\emph{pillowcase}.
To the handlebodies $(B^3,T_0)$ and $(B^3,T_1)$ one associates
traceless $SU(2)$ character varieties
$R_\pi^\natural(B^3,T_0)$ and $R^*(B^3,T_1)$.
By pulling back $SU(2)$ representations along the inclusions
$(S^2,4) \hookrightarrow (B^3,T_0)$ and
$(S^2,4) \hookrightarrow (B^3,T_1)$, one obtains maps
$R_\pi^\natural(B^3,T_0) \rightarrow R^*(S^2,4)$ and
$R^*(B^3,T_1) \rightarrow R^*(S^2,4)$ of the corresponding
character varieties.
The images of these maps define Lagrangians $L_0$ and $W_1$ in
$R^*(S^2,4)$.
Roughly speaking, the pillowcase homology of $(S^3,L)$ is defined to
be the Lagrangian Floer homology of the pair of Lagrangians
$(L_0,W_1)$.

In \cite{Hedden-3}, Hedden, Herald, Hogancamp, and Kirk obtain a
symplectic interpretation of reduced Khovanov homology by modifying
the construction used to define pillowcase homology.
Instead of working directly with the tangle $T_1$, they consider a
cube of resolutions of a 2-tangle diagram in the disk obtained by
projecting $T_1$ onto the plane.
This cube of resolutions is used to construct an object
$(X_1,\delta_1)$
in the twisted Fukaya category of $R^*(S^2,4)$, an $A_\infty$ category
that can be thought of as an analog for Fukaya categories of the
notion of a cochain complex for vector spaces.
The object $X_1$ consists of shifted copies of
Lagrangians corresponding to planar tangles at the vertices of the
cube, and the morphism $\delta_1:X \rightarrow X$ consists of maps of
Lagrangians corresponding to saddles at the edges of the cube.
The trivial tangle $T_0$ corresponds to an object $(W_0,0)$ of the
twisted Fukaya category.

The morphism spaces of the twisted Fukaya category have the structure
of cochain complexes, so in particular the space of morphisms from
$(W_0,0)$ to $(X_1,\delta_1)$ is a cochain complex.
Hedden, Herald, Hogancamp, and Kirk show this cochain complex is
identical to the usual cochain complex for the reduced Khovanov
homology $\Khr(L)$ of $L$, thus proving:

\begin{theorem}(Hedden--Herald--Hogancamp--Kirk
  \cite[Corollary 1.2]{Hedden-3})
\label{theorem:hhhk-khr}
We have an isomorphism of bigraded vector spaces
\begin{align*}
  \Khr(L) \rightarrow H^*(\hom((W_0,0), (X_1,\delta_1))),
\end{align*}
where the $\hom$ space is taken within the twisted Fukaya category of
$R^*(S^2,4)$.
\end{theorem}

Theorem \ref{theorem:hhhk-khr} shows that the Fukaya
category of $R^*(S^2,4)$ knows about Khovanov homology.
Our strategy for generalizing Khovanov homology is based on
generalizing this observation.
Theorem \ref{theorem:hhhk-khr} is formulated in terms of $R^*(S^2,4)$
because this is the character variety of the Heegaard surface
$(S^2,4)$ for the Heegaard splitting (\ref{eqn:split-hhhk}) of
$(S^3,L)$.
In general, one can split a 3-manifold $Y$ containing a
link $L$ along a Heegaard surface $(\Sigma,n)$ that intersects $L$ in
$n$ points, and in light of Theorem \ref{theorem:hhhk-khr} one might
hope that the Fukaya category of the corresponding character variety
$R^*(\Sigma,n)$ could provide clues as to how to generalize Khovanov
homology to links in $Y$.

As a first step towards this goal, in \cite{Boozer-Fukaya} we consider
the case of links in lens spaces.
Given an oriented link $L$ in a lens space $Y$, we consider a Heegaard
splitting
\begin{align*}
  (Y,L) = (U_0,T_0) \cup_{(T^2,2)} (U_1,T_1)
\end{align*}
such that the Heegaard surface $(T^2,2)$ is a 2-torus that
transversely intersects $L$ in two points, the handlebodies
$(U_0,T_0)$ and $(U_1,T_1)$ are solid tori containing 1-tangles $T_0$
and $T_1$, and $T_0$ is trivial.
To the Heegaard surface $(T^2,2)$ we associate the irreducible locus
$R^*(T^2,2)$ of the traceless $SU(2)$ character variety for the torus
with two punctures.
We project $(U_1,T_1)$ onto the plane to obtain a 1-tangle diagram $T$
in the annulus, and we use a cube of resolutions of $T$ to construct
an object $(X_1,\delta_1)$ of the twisted Fukaya category of
$R^*(T^2,2)$.
The handlebody $(U_0,T_0)$ corresponds to an object $(W_0,0)$ of the
twisted Fukaya category.

The space of morphisms from $(W_0,0)$ to
$(X_1,\delta_1)$ has the structure of a cochain complex, and
in \cite{Boozer-Fukaya} we use a partly conjectural description of
the Fukaya category for $R^*(T^2,2)$ to explicitly construct this
cochain complex for the case of links in $S^3$.
The result is the chain complex $(C_{T^\pm}, \partial_{T^\pm})$ that
we consider here.
The methods described in \cite{Boozer-Fukaya} can be used to construct
chain complexes for links in lens spaces, but they do not guarantee
that these chain complexes yield link invariants, and the purpose of
the current paper is to prove this is in fact the case for the chain
complex $(C_{T^\pm}, \partial_{T^\pm})$.
Our Theorem \ref{theorem:main-intro} can thus be viewed as the analog
for $R^*(T^2,2)$ of Theorem \ref{theorem:hhhk-khr} for $R^*(S^2,4)$,
but whereas $R^*(S^2,4)$ yields the usual Khovanov chain complex,
$R^*(T^2,2)$ yields a chain complex that is only homotopy equivalent
to the usual Khovanov chain complex.
As we describe in Remark \ref{remark:pillowcase}, our chain complex is
also relevant to the Fukaya category of $R^*(S^2,4)$.

One could perhaps view this successful prediction of a new chain
complex for links in $S^3$ as evidence that our approach might yield
invariants analogous to Khovanov homology for links in other lens
spaces.
In \cite{Boozer-Fukaya} we use this approach to explicitly
construct chain complexes for some links in $S^2 \times S^1$ and we
present results that suggest the cohomology is indeed a link
invariant.

The paper is organized as follows.
In Section \ref{sec:linear-maps},
we define vector spaces and linear maps that are used to construct the
chain complex $(C_{T^\pm}, \partial_{T^\pm})$.
In Section \ref{sec:chain-complex},
we explain how the chain complex $(C_{T^\pm}, \partial_{T^\pm})$ is
constructed from the cube of resolutions of the tangle diagram $T$,
and we use this chain complex to construct the spectral sequence
described in Theorem \ref{theorem:spectral-intro}.
In Section \ref{sec:example},
we illustrate our results using the example tangle diagram $T$ shown
in Figure \ref{fig:example-T}.
In Section \ref{sec:notation},
we introduce some notation that is useful for describing planar
tangles and saddles.
In Section \ref{sec:squares},
we express the differential $\partial_T^+$ in terms of commuting
squares of saddles in the cube of resolutions of $T$.
In Section \ref{sec:proof},
we outline the proof of Theorem \ref{theorem:main-intro}.
In Sections \ref{sec:induced-chain-complex} -- \ref{sec:saddle-2}, as
well as Appendices \ref{appendix:bigradings} and
\ref{appendix:induced-saddles},
we fill in technical details needed to complete the proof.

\section{Vector spaces and linear maps}
\label{sec:linear-maps}

Here we define vector spaces and linear maps that we will use to
construct the chain complex $(C_{T^\pm}, \partial_{T^\pm})$.
The vector spaces we consider are typically bigraded.
Given a bigraded vector space $V$, we use the notation $v^{(h,q)}$ to
indicate that a homogeneous vector $v \in V$ has bigrading $(h,q)$.
We refer to $h$ as the \emph{homological grading} and $q$ as the
\emph{quantum grading}.
We define the vector space $V[h_s,q_s]$ to be $V$ with gradings
shifted \emph{upward} by $(h_s,q_s)$, so if $v \in V$ is homogeneous
with bigrading $(h, q)$ then the corresponding vector
$v \in V[h_s,q_s]$ is homogeneous with bigrading $(h + h_s, q + q_s)$.
We define $\F$ to be the field of two elements, where $1 \in \F$ is
assigned bigrading $(0,0)$.
We define a two-dimensional bigraded $\F$-vector space
\begin{align*}
  A = \langle e^{(0,1)},\, x^{(0,-1)} \rangle.
\end{align*}

We define the following $\F$-linear maps.
We define \emph{unit maps} $\eta^{(0,1)}$ and $\etadot^{(0,-1)}$:
\begin{align*}
  &\eta:\F \rightarrow A, &
  &\eta(1) = e, \\
  &\etadot:\F \rightarrow A, &
  &\etadot(1) = x.
\end{align*}
We define \emph{counit maps} $\epsilon^{(0,1)}$ and
$\epsilondot^{(0,-1)}$:
\begin{align*}
  &\epsilon:A \rightarrow \F, &
  &\epsilon(e) = 0, \qquad
  \epsilon(x) = 1, \\
  &\epsilondot:A \rightarrow \F, &
  &\epsilondot(e) = 1, \qquad
  \epsilondot(x) = 0.
\end{align*}
We define a \emph{raising map} $\id_{ex}^{(0,2)}$ and a
\emph{lowering map} $\id_{xe}^{(0,-2)}$:
\begin{align*}
  &\id_{ex}:A \rightarrow A, &
  &\id_{ex}(e) = 0, \qquad
  \id_{ex}(x) = e, \\
  &\id_{xe}:A \rightarrow A, &
  &\id_{xe}(e) = x, \qquad
  \id_{xe}(x) = 0.
\end{align*}
We define \emph{projection maps} $\id_{ee}^{(0,0)}$ and
$\id_{xx}^{(0,0)}$:
\begin{align*}
  &\id_{ee}:A \rightarrow A, &
  &\id_{ee}(e) = e, \qquad
  \id_{ee}(x) = 0, \\
  &\id_{xx}:A \rightarrow A, &
  &\id_{xx}(e) = 0, \qquad
  \id_{xx}(x) = x.
\end{align*}
We define a \emph{multiplication map} $m^{(0,-1)}$:
\begin{align*}
  &m:A \otimes A \rightarrow A, &
  &m(e \otimes e) = e, \qquad
  m(e \otimes x) = m(x \otimes e) = x, \qquad
  m(x \otimes x) = 0.
\end{align*}
We define a \emph{comultiplication map} $\Delta^{(0,-1)}$:
\begin{align*}
  &\Delta:A \rightarrow A \otimes A, &
  &\Delta(e) = e \otimes x + x \otimes e, \qquad
  \Delta(x) = x \otimes x.
\end{align*}

The graded vector space $A$, together with the multiplication $m$,
comultiplication $\Delta$, unit $\eta$, and counit $\epsilon$, gives
Khovanov's Frobenius algebra.
For notational simplicity, given an $\F$-vector space $V$ we often
identify $V$ with $\F \otimes V$ and $V \otimes \F$.

\section{Chain complex $(C_{T^\pm},\, \partial_{T^\pm})$}
\label{sec:chain-complex}

Consider an oriented tangle diagram $T$ in the annulus
$S^1 \times [0,1]$.
We choose a point $b_0 \in S^1$ and take the boundary of $T$ to be the
points $(b_0,0)$ and $(b_1,1)$ on the inner and outer bounding circles
of the annulus.
We close $T$ with an unknotted overpass arc $A^+$ or underpass arc
$A^-$ that crosses $T$ transversely to obtain an oriented link diagram
$T^\pm := T \cup A^\pm$.
The image of $A^\pm$ under the projection
$S^1 \times [0,1] \rightarrow S^1$ is a loop based at $b_0$, and we
choose $A^\pm$ such that this loop is contractible.
We construct a bigraded chain complex
$(C_{T^\pm},\, \partial_{T^\pm})$ for the link diagram $T^\pm$ from a
cube of resolutions of the tangle diagram $T$ as follows.

First we describe the cube of resolutions.
Let $m_+(T)$ and $m_-(T)$ denote the number of positive and negative
crossings of $T$, and let $m(T) = m_+(T) + m_-(T)$ denote the total
number of crossings.
We can specify a planar resolution of $T$ by specifying how each
crossing is to be resolved.
Define the 0-resolution, respectively 1-resolution, of a crossing such
that the overpass turns left, respectively right.
We fix an arbitrary ordering of the crossings and encode the data
needed to specify a planar resolution as a binary string of length
$m(T)$, so the $k$-th bit of the binary string specifies the
resolution of the $k$-th crossing of $T$.
Let $I = \{0,1\}^{m(T)}$ denote the set of binary strings of length
$m(T)$.
For each binary string $i \in I$, let $T_i$ denote the corresponding
planar resolution of $T$.
We define the \emph{resolution degree} $r(T_i)$ of the planar tangle
$T_i$ to be the number of 1's in the binary string $i$.
For each pair of binary strings $i, j \in I$ such that
$r(T_j) = r(T_i) + 1$, we define a saddle $T_i \rightarrow T_j$.
The strings $i$ and $j$ are identical except for a single bit that
is 0 in $i$ but 1 in $j$, and the saddle $T_i \rightarrow T_j$ changes
the resolution of the corresponding crossing of $T$ from the
0-resolution in $T_i$ to the 1-resolution in $T_j$.

Next we describe the bigraded vector space $C_{T^\pm}$, which is
constructed from the planar tangles in the cube of resolutions of $T$.
A planar tangle $p$ consists of an arc component connecting the points
$(b_0,0)$ and $(b_0,1)$ on the inner and outer boundary of the annulus
$S^1 \times [0,1]$, together with some number of circle components.
We orient the arc component of $p$ in the \emph{outward} direction
from $(b_0,0)$ to $(b_0,1)$.
The image of the arc component of $p$ under the projection
$S^1 \times [0,1] \rightarrow S^1$ is an oriented loop based at $b_0$,
and we define the \emph{winding number} $w(p)$ to be the number of
times this loop winds \emph{counterclockwise} around $S^1$.
We define the \emph{circle number} $c(p)$ to be the number of circle
components of $p$.
Let $a_+(A^\pm,T)$ and $a_-(A^\pm,T)$ denote the number of positive
and negative crossings between $A^\pm$ and $T$.
For each planar tangle $T_i$ in the cube of resolutions of $T$, we
define a corresponding bigraded vector space
\begin{align*}
  C_{T_i}^\pm = A^{\otimes c(T_i)}[h^\pm(T,T_i),\, q^\pm(T,T_i)],
\end{align*}
where the bigrading shift is given by
\begin{align}
  \label{eqn:h}
  h^\pm(T,T_i) &=
  -m_-(T) + \frac{1}{2}(a_+(A^\pm,T) - a_-(A^\pm,T) \pm w(T_i)) +
  r(T_i), \\
  \label{eqn:q}
  q^\pm(T,T_i) &=
  m_+(T) - 2m_-(T) + \frac{3}{2}(a_+(A^\pm,T) - a_-(A^\pm,T) \pm
  w(T_i)) + r(T_i).
\end{align}
We define the bigraded vector space $C_{T^\pm}$ as
\begin{align*}
  C_{T^\pm} = \bigoplus_{i \in I} C_{T_i}^\pm.
\end{align*}

Next we describe the differential $\partial_{T^\pm}$, which is
constructed from the saddles in the cube of resolutions of $T$.
We classify the saddles as type $C_\pm^R$, $C_\pm^L$, $C_\pm^C$,
or $W_\pm$, as shown in Figure \ref{fig:saddle-types}.
A saddle of type $C_\pm^R$, $C_\pm^L$, or $C_\pm^C$ splits $(+)$ or
merges $(-)$ a circle component.
A saddle of type $C_\pm^R$ or $C_\pm^L$ connects a circle component
to the right $(R)$ or left $(L)$ side of the outward-oriented arc
component.
A saddle of type $C_\pm^C$ connects two circle components.
A saddle of type $W_\pm$ raises $(+)$ or lowers $(-)$ the winding
number by two.

We assign several linear maps to each saddle in the cube of
resolutions.
We denote a planar tangle in the cube of resolutions by $p$ or
$q$, and indicate that a planar tangle has winding number $n$ and
circle number $r$ using the notation $p(n,r)$ or $q(n,r)$.

\begin{figure}
  \centering
  \includegraphics[scale=0.5]{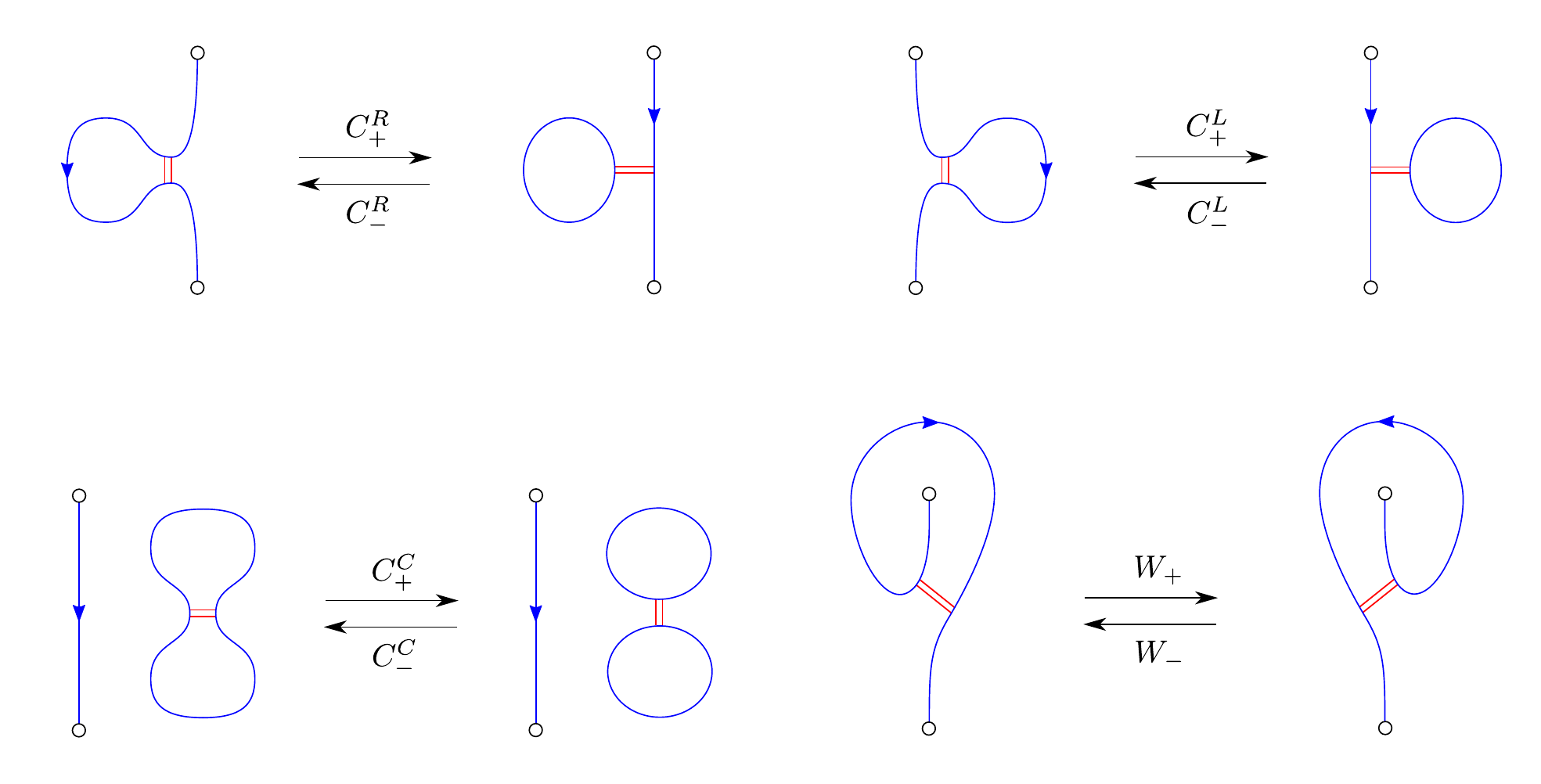}
\caption{
\label{fig:saddle-types}
Example saddles of type $C_\pm^R$, $C_\pm^L$, $C_\pm^C$, and $W_\pm$.
}
\end{figure}

For each saddle $s:p \rightarrow q$, we define a map
$T(s):C_p^\pm \rightarrow C_q^\pm$ as follows:
\begin{itemize}
\item
If $s:p(n,r) \rightarrow q(n,r+1)$ is type $C_+^L$ or $C_+^R$, we
define
\begin{align*}
  T(s) = (\id_{A^{\otimes r}} \otimes \etadot)[1,1]:
  A^{\otimes r}[h,q] \rightarrow (A^{\otimes r} \otimes A)[h+1,q+1].
\end{align*}

\item
If $s:p(n,r+1) \rightarrow q(n,r)$ is type $C_-^L$ or $C_-^R$, we
define
\begin{align*}
  T(s) = (\id_{A^{\otimes r}} \otimes \epsilondot)[1,1]:
  (A^{\otimes r} \otimes A)[h,q] \rightarrow A^{\otimes r}[h+1,q+1].
\end{align*}

\item
If $s:p(n,r+1) \rightarrow q(n,r+2)$ is type $C_+^C$, we define
\begin{align*}
  T(s) = (\id_{A^{\otimes r}} \otimes \Delta)[1,1]:
  (A^{\otimes r} \otimes A)[h,q] \rightarrow
  (A^{\otimes r} \otimes A \otimes A)[h+1,q+1].
\end{align*}

\item
If $s:p(n,r+2) \rightarrow q(n,r+1)$ is type $C_-^C$, we define
\begin{align*}
  T(s) = (\id_{A^{\otimes r}} \otimes m)[1,1]:
  (A^{\otimes r} \otimes A \otimes A)[h,q] \rightarrow
  (A^{\otimes r} \otimes A)[h+1,q+1].
\end{align*}

\item
If $s$ is type $W_\pm$, we define $T(s) = 0$.
\end{itemize}
We extend $T(s)$ by zero to obtain a map
$T(s):C_{T^\pm} \rightarrow C_{T^\pm}$.

For each saddle $s:p \rightarrow q$ and each orientation
$\beta \in \{L,R\}$, we define a map
$\Ttilde^\beta(s):C_p^\pm \rightarrow C_q^\pm$ as follows:
\begin{itemize}
\item
If $s:p(n,r) \rightarrow q(n,r+1)$ is type $C_+^\beta$, we
define
\begin{align*}
  \Ttilde^\beta(s) = (\id_{A^{\otimes r}} \otimes \eta)[1,1]:
  A^{\otimes r}[h,q] \rightarrow (A^{\otimes r} \otimes A)[h+1,q+1].
\end{align*}

\item
If $s:p(n,r+1) \rightarrow q(n,r)$ is type $C_-^\beta$, we define
\begin{align*}
  \Ttilde^\beta(s) = (\id_{A^{\otimes r}} \otimes \epsilon)[1,1]:
  (A^{\otimes r} \otimes A)[h,q] \rightarrow A^{\otimes r}[h+1,q+1].
\end{align*}

\item
If $s$ is not type $C_+^\beta$ or $C_-^\beta$, we define
$\Ttilde^\beta(s) = 0$.
\end{itemize}
We extend $\Ttilde^\beta(s)$ by zero to obtain a map
$\Ttilde^\beta(s):C_{T^\pm} \rightarrow C_{T^\pm}$.

For each saddle $s:p \rightarrow q$, we define a map
$P(s):C_p^- \rightarrow C_q^-$ as follows:
\begin{itemize}
\item
If $s:p(n,r) \rightarrow q(n+2,r)$ is type $W_+$, we define
\begin{align*}
  P(s) = \id_{A^{\otimes r}}[0, -2]:
  A^{\otimes r}[h,q] \rightarrow A^{\otimes r}[h,q-2].
\end{align*}

\item
If $s$ is not type $W_+$, we define $P(s) = 0$.
\end{itemize}
We extend $P(s)$ by zero to obtain a map
$P(s):C_{T^-} \rightarrow C_{T^-}$.

For each saddle $s:p \rightarrow q$, we define a map
$Q(s):C_p^+ \rightarrow C_q^+$ as follows:
\begin{itemize}
\item
If $s:p(n,r) \rightarrow q(n-2,r)$ is type $W_-$, we define
\begin{align*}
  Q(s) = \id_{A^{\otimes r}}[0, -2]:
  A^{\otimes r}[h,q] \rightarrow A^{\otimes r}[h,q-2].
\end{align*}

\item
If $s$ is not type $W_-$, we define $Q(s) = 0$.
\end{itemize}
We extend $Q(s)$ by zero to obtain a map
$Q(s):C_{T^+} \rightarrow C_{T^+}$.

In the above expressions, we have assumed a specific choice of
ordering of the factors of $A$ corresponding to the circle components
of the planar tangles.
For a different ordering of these factors we would need to modify the
above expressions accordingly.

We define a map
\begin{align*}
  \Ttilde(s) = \Ttilde^L(s) + \Ttilde^R(s).
\end{align*}
We define maps
\begin{align*}
  &T = \sum_s T(s), &
  &\Ttilde^\beta = \sum_s \Ttilde^\beta(s), &
  &P = \sum_s P(s), &
  &Q = \sum_s Q(s),
\end{align*}
where the sums are over all saddles in the cube of resolutions of $T$.
The bigradings of these maps are
\begin{align*}
  &(T)^{(1,0)}, & (\Ttilde^\beta)^{(1,2)}, &
  &(Q)^{(0,-2)}, &
  &(P)^{(0,-2)}.
\end{align*}
We define maps
\begin{align*}
  \partial_T^0 &= T, &
  \partial_T^+ &= Q\Ttilde^L + \Ttilde^L Q, &
  \partial_T^- &= P\Ttilde^L + \Ttilde^L P,
\end{align*}
which have bigrading $(1,0)$.
The differential $\partial_{T^\pm}$ for the link diagram $T^\pm$ is
then given by
\begin{align*}
  \partial_{T^\pm} = \partial_T^0 + \partial_T^\pm,
\end{align*}
and has bigrading $(1,0)$.
We note that for $\partial_T^0$ the assignment of linear maps to
saddles is the same as for the usual chain complex for reduced
Khovanov homology, where the arc component plays the role of the
marked circle component.
The new feature of our chain complex is the term $\partial_T^\pm$,
which describes \emph{long differentials} corresponding to pairs of
successive saddles.

Our chain complex can be viewed as a generalization of the usual chain
complex for reduced Khovanov homology in the following sense.
Consider a tangle diagram $T$ such that $w(T_i) = 0$ for all
$i \in I$.
Then $T$ can be closed with an arc $A^0$ that does not cross $T$ and
whose image under the projection $S^1 \times [0,1] \rightarrow S^1$ is
a contractible loop.
Since $A^0$ can be viewed as either an overpass or underpass arc,
we can take $A^+ = A^- = A^0$ and obtain a link diagram
$T^+ = T^- = T \cup A^0$.
There are no long differentials, so $\partial_{T^\pm} = \partial_T^0$,
and equations (\ref{eqn:h}) and (\ref{eqn:q}) for the bigrading shift
reduce to the usual expressions for Khovanov homology:
\begin{align}
  \label{eqn:h-q-ordinary}
  &h^\pm(T,T_i) = -m_-(T) + r(T_i), &
  &q^\pm(T,T_i) = m_+(T) - 2m_-(T) + r(T_i).
\end{align}
It follows that
$(C_{T^\pm},\,\partial_{T^\pm}) = (C_{T^\pm},\, \partial_T^0)$
is the usual chain complex for the reduced Khovanov homology of the
link diagram $T^+ = T^-$, where the marked point is taken to be
$(b_0,0)$ or $(b_0,1)$.

For a tangle diagram $T$ of this special form, equation
(\ref{eqn:h-q-ordinary}) shows that the homological grading
$h^\pm(T,T_i)$ is just the resolution degree $r(T_i)$ shifted by a
constant.
But in general this is not the case, since equation (\ref{eqn:h}) for
$h^\pm(T,T_i)$ contains an additional term that depends on $w(T_i)$.
In particular, both $\partial_T^0$ and $\partial_T^\pm$ increase the
homological grading by one, but $\partial_T^0$ increases the
resolution degree by one and $\partial_T^\pm$ increases the
resolution degree by two.
The fact that
$(\partial_{T^\pm})^2 = (\partial_T^0 + \partial_T^\pm)^2 = 0$
must hold in each resolution degree thus implies

\begin{lemma}
\label{lemma:d0}
We have $(\partial_T^0)^2 = 0$.
\end{lemma}

We can use the resolution degree to define a filtration
$\calF_{T^\pm}^0 = C_{T^\pm} \supset \calF_{T^\pm}^1 \supset
\calF_{T^\pm}^2 \supset \cdots$
of the chain complex $(C_{T^\pm}, \partial_{T^\pm})$ as follows:
\begin{align*}
  \calF_{T^\pm}^s := \bigoplus_{i \in I \mid r(T_i) \geq s} C_{T_i}^\pm.
\end{align*}
This filtration, and the fact that $(C_{T^\pm}, \partial_T^0)$ is a
chain complex by Lemma \ref{lemma:d0}, yield the spectral sequence
described in Theorem \ref{theorem:spectral-intro}.

\section{Example}
\label{sec:example}

We will now illustrate our results using the example tangle diagram
$T$ shown in Figure \ref{fig:example-T}.
The number of positive and negative crossings $m_+(T)$ and $m_-(T)$ of
$T$ are
\begin{align*}
  & m_+(T) = 2, &
  & m_-(T) = 0.
\end{align*}
The cube of resolutions of $T$ is shown in Figure
\ref{fig:example-T} and reproduced here:
\begin{eqnarray*}
\begin{tikzcd}
  {} & T_{10}(-1,0) \arrow{dr}{C_+^R} & {} \\
  T_{00}(1,0)
  \arrow{ur}{W_-} \arrow{dr}[swap]{W_-} & {} &
  T_{11}(-1,1), \\
  {} & T_{01}(-1,0) \arrow{ur}[swap]{C_+^L} & {}
\end{tikzcd}
\end{eqnarray*}
where $T_i(n_i,r_i)$ indicates that the planar tangle
$T_i$ has winding number $n_i$ and circle number $r_i$.

\subsection{Link diagram $T^+$}

Closing $T$ with an overpass arc $A^+$ yields a link diagram
$T^+ = T \cup A^+$ for the unknot.
We have $a_+(A^+,T) - a_-(A^+,T) = -1$, so the chain complex
$(C_{T^+},\, \partial_{T^+})$ is
\begin{eqnarray*}
\begin{tikzcd}
  {} & C_{T_{10}}^+ = \F[0,0] \arrow{dr}{\etadot[1,1]} & {} \\
  C_{T_{00}}^+ = \F[0,2] \arrow[Rightarrow]{rr}{\eta[1,-1]} & {} &
  C_{T_{11}}^+ = A[1,1]. \\
  {} & C_{T_{01}}^+ = \F[0,0] \arrow{ur}[swap]{\etadot[1,1]} & {}
\end{tikzcd}
\end{eqnarray*}
The differential is
$\partial_{T^+} = \partial_T^0 + \partial_T^+$,
where $\partial_T^0$ consists of the two arrows labeled
$\etadot[1,1]$ and $\partial_T^+$ consists of the double arrow
labeled $\eta[1,-1]$.
Note that the pair of saddles $T_{00} \rightarrow T_{01}$ of type
$W_-$ and $T_{01} \rightarrow T_{11}$ of type $C_+^L$ give the long
differential $\eta[1,-1]:C_{T_{00}}^+ \rightarrow C_{T_{11}}^+$.
The cohomology of $(C_{T^+},\, \partial_{T^+})$
is the reduced Khovanov homology for the unknot:
\begin{align*}
  H^*(C_{T^+},\, \partial_{T^+}) = \Khr(T^+) = \F[0,0].
\end{align*}

We next consider the spectral sequence for $(C_{T^+},\partial_{T^+})$.
The $E_2$ page is given by
\begin{align*}
  E_2 = H^*(C_{T^+},\, \partial_T^0) =
  \F[0,0] \oplus \F[0,2] \oplus \F[1,2],
\end{align*}
and the differential $d_2:E_2 \rightarrow E_2$, which is induced by
$\partial_T^+$, maps the generator in bigrading $(0,2)$ to the
generator in bigrading $(1,2)$.
The $E_3$ page is thus given by
\begin{align*}
  E_3 = H^*(E_2,d_2) = \Khr(T^+) = \F[0,0],
\end{align*}
and the differential $d_3:E_3 \rightarrow E_3$ is zero, so the
spectral sequence collapses at this page.

\subsection{Link diagram $T^-$}

Closing $T$ with an underpass arc $A^-$ yields a link diagram
$T^- = T \cup A^-$ for the right trefoil.
We have $a_+(A^-,T) - a_-(A^-,T) = 1$, so
the chain complex $(C_{T^-},\, \partial_{T^-})$ is
\begin{eqnarray*}
\begin{tikzcd}
  {} & C_{T_{10}}^- = \F[2,6] \arrow{dr}{\etadot[1,1]} & {} \\
  C_{T_{00}}^- = \F[0,2] & {} &
  C_{T_{11}}^- = A[3,7]. \\
  {} & C_{T_{01}}^- = \F[2,6] \arrow{ur}[swap]{\etadot[1,1]} & {}
\end{tikzcd}
\end{eqnarray*}
The differential is
$\partial_{T^-} = \partial_T^0 + \partial_T^-$,
where $\partial_T^0$ consists of the two arrows labeled
$\etadot[1,1]$ and $\partial_T^- = 0$.
The cohomology of
$(C_{T^-},\, \partial_{T^-})$
is the reduced Khovanov homology for the right trefoil:
\begin{align*}
  H^*(C_{T^-},\, \partial_{T^-}) = \Khr(T^-) =
  \F[0,2] \oplus \F[2,6] \oplus \F[3,8].
\end{align*}

We next consider the spectral sequence for
$(C_{T^-},\, \partial_{T^-})$.
The $E_2$ page is given by
\begin{align*}
  E_2 = H^*(C_{T^-},\, \partial_T^0) = \Khr(T^-) =
  \F[0,2] \oplus \F[2,6] \oplus \F[3,8],
\end{align*}
and the differential $d_2:E_2 \rightarrow E_2$ is zero, so the
spectral sequence collapses at this page.
Note that if we ignore bigradings, the $E_2$ pages of the spectral
sequences for $T^+$ and $T^-$ are the same.

\section{Notation for planar tangles and saddles}
\label{sec:notation}

We now introduce some notation that will be useful for describing
planar tangles and saddles.
Recall that a planar tangle in the annulus $S^1 \times [0,1]$ consists
of an arc component connecting the points $y_0 := (b_0,0)$ and
$y_1 := (b_0,1)$ on the inner and outer bounding circles, together
with some number of circle components.
We orient the arc component in the outward direction from $y_0$ to
$y_1$.
For the purpose of defining our notation we choose arbitrary
orientations for the circle components, and view different choices of
orientations as providing equivalent descriptions of the same
underlying tangle.

The orientation of the arc component defines an ordering of its
points, and we will use strings of inequalities such as
$x_1 < \cdots < x_n$
to indicate that points $x_1, \cdots, x_n$ on the arc component are
ordered in the stated manner.
We use the notation $(x_1 < \cdots < x_n)$ to indicate that points
$x_1, \cdots, x_n$ lie on a circle component and are encountered in
the stated order as we move around the circle in the direction given
by its chosen orientation.
We can cyclically permute the ordering of points on a circle
component, so
\begin{align}
  \label{eqn:circle-rot}
  (x_1 < x_2 < \cdots < x_n) = (x_2 < \cdots < x_n < x_1).
\end{align}
We can flip the orientation of a circle component to obtain an
equivalent description of the same planar tangle, so
\begin{align}
  \label{eqn:circle-flip}
  (x_1 < x_2 < \cdots < x_n) = (x_n < \cdots < x_2 < x_1).
\end{align}

This notation is useful for describing saddles.
The attaching sphere of a saddle $s:p \rightarrow q$ consists of two
points, which we call the \emph{attaching points} of
the saddle.
We indicate that an attaching point of the saddle $s$ attaches to the
right or left side of an oriented component using the notation $s R$
or $s L$.
For example, the planar tangles in the cube of resolutions in Figure
\ref{fig:example-T} are described as
\begin{align*}
  &T_{00}:\quad  y_0 < s_1R < s_2R < s_1L < s_2L < y_1 &
  &T_{10}:\quad  y_0 < s_1L < s_2L < s_1R < s_2L < y_1 \\
  &T_{01}:\quad  y_0 < s_1R < s_2L < s_1R < s_2R < y_1 &
  &T_{11}:\quad  y_0 < s_1L < s_2R < y_1,\,(s_1 R < s_2 R),
\end{align*}
where we have oriented the circle component in $T_{11}$
counterclockwise.
Note that the saddles $T_{10} \rightarrow T_{11}$ and
$T_{01} \rightarrow T_{11}$ split a circle component from the right
and left sides of the arc component, respectively.
Since the arc component is oriented, each saddle naturally assigns an
orientation to this circle component, but because the saddles
split the circle from opposite sides of the arc component, the two
orientations are not consistent.
This type of inconsistency is the reason we treat clockwise and
counterclockwise orientations of circle components on equal footing.

We can adapt rules (\ref{eqn:circle-rot}) and
(\ref{eqn:circle-flip}) for circle components to include the
orientation data for the attaching points of saddles:
\begin{align*}
  (x_1\alpha_1 < x_2\alpha_2 < \cdots < x_n\alpha_n) =
  (x_2\alpha_2 < \cdots < x_n\alpha_n < x_1 \alpha_1) =
  (x_n \alphabar_n < \cdots < x_2 \alphabar_2 < x_1\alphabar_1),
\end{align*}
where $\alpha_i \in \{R, L\}$ and $\alphabar_i$ denotes the opposite
orientation of $\alpha_i$, so
$\overline{R} = L$ and
$\overline{L} = R$.
Using these rules, we see that the circle component in the planar
tangle $T_{11}$ from Figure \ref{fig:example-T} could be denoted
in any of the following ways:
\begin{align*}
  (s_1 R < s_2 R) = (s_2 R < s_1 R) = (s_2 L < s_1 L) =
  (s_1 L < s_2 L).
\end{align*}

We can use this notation to describe saddles of type
$C_\pm^\alpha$, $C_\pm^C$, and $W_\pm$, as shown in Figure
\ref{fig:saddle-types}.
For a saddle $s$ of type $C_+^\alpha$, which attaches to the \emph{same}
side of the arc component at its attaching points, the saddle splits off
the segment of the arc component between its attaching points to form
a new circle component:
\begin{align}
  \label{eqn:C-alpha-p}
  &y_0 < s\alphabar < x_1\beta_1 < \cdots <
  x_n\beta_n < s\alphabar < y_1 &
  & \longrightarrow &
  &y_0 < s\alpha < y_1,\,
  (x_1\beta_1 < \cdots < x_n\beta_n < s\alpha).
\end{align}
For a saddle of type $C_-^\alpha$, we flip the direction of the arrow
in (\ref{eqn:C-alpha-p}).

For a saddle $s$ of type $C_+^C$, which attaches to the \emph{same} side
of a circle component at its attaching points, the saddle splits off
the segment of the circle component between its attaching points to
form a new circle component:
\begin{align}
  \label{eqn:C-C-p}
  &(s\alphabar < x_1\beta_1 < \cdots < x_n\beta_n < s\alphabar) &
  & \longrightarrow &
  &(s\alpha),\, (x_1\beta_1 < \cdots < x_n\beta_n < s\alpha).
\end{align}
For a saddle of type $C_-^C$, we flip the direction of the arrow
in (\ref{eqn:C-C-p}).

For a saddle $s$ of type $W_\pm$, which attaches to \emph{opposite}
sides of the arc component at its attaching points, the saddle
flips the orientation of the segment of the arc component between its
attaching points:
\begin{align*}
  &y_0 < s\alpha < x_1\beta_1 < \cdots <
  x_n\beta_n < s\alphabar < y_1 &
  & \longrightarrow &
  &y_0 < s\alphabar < x_n\betabar_n < \cdots <
  x_1\betabar_1 < s\alpha < y_1.
\end{align*}
This segment forms a single loop around the annulus, so flipping its
orientation changes the winding number by two.
Given our convention that \emph{counterclockwise} winding is
\emph{positive}, we see that
if $\alpha = R$ then $s$ \emph{lowers} the winding number by two and
hence is type $W_-$, and 
if $\alpha = L$ then $s$ \emph{raises} the winding number by two and
hence is type $W_+$.

The following two lemmas describe restrictions on the possible pairs
of saddles:

\begin{lemma}
\label{lemma:impossible-nest}
A pair of saddles $s$ and $t$ of the following form is not possible:
\begin{align*}
  s\alpha_1 < t\beta < t\betabar < s\alpha_2,
\end{align*}
\end{lemma}

\begin{proof}
If we glue the attaching points of the saddle $s$ together then the
segment between them forms a circle, and the saddle $t$ cannot attach
to points on opposite sides of this circle.
\end{proof}

\begin{lemma}
\label{lemma:impossible-pairs}
Pairs of saddles $s$ and $t$ of the following forms are not possible:
\begin{align*}
  &s\alphabar < t\alpha < s\alpha < t\alpha, &
  &s\alpha < t\alphabar < s\alphabar < t\alpha, &
  &s\alpha < t\alpha < s\alpha < t\alphabar.
\end{align*}
\end{lemma}

\begin{proof}
These planar tangles make up three corners of a square of saddles:
\begin{eqnarray*}
  \begin{tikzcd}
    s\alphabar < t\alpha < s\alpha < t\alpha
    \arrow[leftrightarrow]{d}{t}
    \arrow[leftrightarrow]{r}{s} &
    s\alpha < t\alphabar < s\alphabar < t\alpha
    \arrow[leftrightarrow]{d}{t} \\
    s\alphabar < t\alphabar,\, (s\alpha < t\alphabar)
    \arrow[leftrightarrow]{r}{s} &
    s\alpha < t\alpha < s\alpha < t\alphabar.
  \end{tikzcd}
\end{eqnarray*}
Such a square is not possible, since the saddles $s$ and $t$ connect
the arc component to opposite sides of the circle component
$(s\alpha < t\alphabar)$ in the planar tangle in the bottom left
corner.
\end{proof}

\section{Squares of saddles}
\label{sec:squares}

Recall that the differential
$\partial_{T^+} = \partial_T^0 + \partial_T^+$ is the sum of a
term $\partial_T^0$ corresponding to single saddles and a term
$\partial_T^+$ corresponding to pairs of successive saddles in the
cube of resolutions of $T$.
Each pair of successive saddles belongs to a unique commuting square
of saddles $\Box$ of the following form:
\begin{eqnarray*}
\begin{tikzcd}
  \Box_{00}
  \arrow{r}{\Box_T}
  \arrow{d}{\Box_L}
  &
  \Box_{01}
  \arrow{d}{\Box_R}
  \\
  \Box_{10}
  \arrow{r}{\Box_B} &
  \Box_{11},
\end{tikzcd}
\end{eqnarray*}
where $\Box_{00}$, $\Box_{01}$, $\Box_{10}$, and $\Box_{11}$ indicate
planar tangles at the corners of the square, and $\Box_T$, $\Box_B$,
$\Box_L$, and $\Box_R$ indicate saddles at the top, bottom, left, and
right sides of the square.
The square contains two pairs of successive saddles
$(\Box_T,\Box_R)$ and $(\Box_L,\Box_B)$, each of which contributes to
$\partial_T^+$:
\begin{align*}
  &(\Box_T,\Box_R) &
  &\rightsquigarrow &
  & \Ttilde^L(\Box_R) Q(\Box_T) +
  Q(\Box_R) \Ttilde^L(\Box_T), \\
  &(\Box_L,\Box_B) &
  &\rightsquigarrow &
  &\Ttilde^L(\Box_B) Q(\Box_L) +
  Q(\Box_B) \Ttilde^L(\Box_L).
\end{align*}
We define a map $\partial_T^+(\Box)$ that gives the net contribution
of the square $\Box$ to $\partial_T^+$ by summing the contributions
of $(\Box_T,\Box_R)$ and $(\Box_L,\Box_B)$:
\begin{align*}
  \partial_T^+(\Box) &=
  \Ttilde^L(\Box_R) Q(\Box_T) +
  Q(\Box_R) \Ttilde^L(\Box_T) +
  \Ttilde^L(\Box_B) Q(\Box_L) +
  Q(\Box_B) \Ttilde^L(\Box_L).
\end{align*}
We can then express $\partial_T^+$ as a sum over all the squares in
the cube of resolutions of $T$:
\begin{align*}
  \partial_T^+ = \sum_\Box \partial_T^+(\Box).
\end{align*}

We will show that the map $\partial_T^+(\Box)$ vanishes unless the
square $\Box$ is one of several special types.
We say that a square $\Box$ is \emph{interleaved}, or type $I_\pm$,
if it has one of the following forms:
\begin{eqnarray*}
\begin{array}{cc}
\begin{tikzcd}
  sR < tR < sL < tL
  \arrow{r}{W_-}[swap]{s}
  \arrow{d}{W_-}[swap]{t}
  &
  sL < tL < sR < tL
  \arrow{d}{C_+^R}[swap]{t}
  \\
  sR < tL < sR < tR
  \arrow{r}{C_+^L}[swap]{s} &
  sL < tR,\, (sR < tR)
\end{tikzcd}
&
\begin{tikzcd}
  sR < tL,\, (sL < tL)
  \arrow{r}{C_-^R}[swap]{s}
  \arrow{d}{C_-^L}[swap]{t} &
  sL < tR < sL < tL
  \arrow{d}{W_-}[swap]{t}
  \\
  sR < tR < sL < tR
  \arrow{r}{W_-}[swap]{s}
  &
  sL < tL < sR < tR
\end{tikzcd} \\
I_+ & I_-
\end{array}
\end{eqnarray*}
For example, the cube of resolutions in Figure \ref{fig:example-T}
is an interleaved square of type $I_+$.
We say that a square $\Box$ is \emph{nested}, or type $N_\pm^\beta$,
if it has one of the following forms:
\begin{eqnarray*}
\begin{array}{cc}
\begin{tikzcd}
  sR < t\beta < t\beta < sL
  \arrow{r}{W_-}[swap]{s}
  \arrow{d}{C_+^\betabar}[swap]{t}
  &
  sL < t\betabar < t\betabar < sR
  \arrow{d}{C_+^\beta}[swap]{t}
  \\
  sR < t\betabar < sL,\, (t\betabar)
  \arrow{r}{W_-}[swap]{s}
  &
  sL < t\beta < sR,\,(t\beta)
\end{tikzcd}
&
\begin{tikzcd}
  sR < t\betabar < sL,\, (t\betabar)
  \arrow{d}{C_-^\betabar}[swap]{t}
  \arrow{r}{W_-}[swap]{s}
  &
  sL < t\beta < sR,\,(t\beta)
  \arrow{d}{C_-^\beta}[swap]{t}
  \\
  sR < t\beta < t\beta < sL
  \arrow{r}{W_-}[swap]{s}
  &
  sL < t\betabar < t\betabar < sR
\end{tikzcd} \\
N_+^\beta & N_-^\beta
\end{array}
\end{eqnarray*}
We say that a square $\Box$ is \emph{disjoint}, or type $D$, if it has
one of the following forms:
\begin{eqnarray*}
\begin{array}{cc}
\begin{tikzcd}
  t\betabar < t\betabar < sR < sL
  \arrow{r}{W_-}[swap]{s}
  \arrow{d}{C_+^\beta}[swap]{t}
  &
  t\betabar < t\betabar < sL < sR
  \arrow{d}{C_+^\beta}[swap]{t}
  \\
  t\beta < sR < sL,\, (t\beta)
  \arrow{r}{W_-}[swap]{s}
  &
  t\beta < sL < sR,\, (t\beta)
\end{tikzcd}
&
\begin{tikzcd}
  t\beta < sR < sL,\, (t\beta)
  \arrow{r}{W_-}[swap]{s}
  \arrow{d}{C_-^\beta}[swap]{t}
  &
  t\beta < sL < sR,\, (t\beta)
  \arrow{d}{C_-^\beta}[swap]{t}
  \\
  t\betabar < t\betabar < sR < sL
  \arrow{r}{W_-}[swap]{s}
  &
  t\betabar < t\betabar < sL < sR
\end{tikzcd}
\\
\begin{tikzcd}
  sR < sL < t\betabar < t\betabar
  \arrow{r}{W_-}[swap]{s}
  \arrow{d}{C_+^\beta}[swap]{t}
  &
  sL < sR < t\betabar < t\betabar
  \arrow{d}{C_+^\beta}[swap]{t}
  \\
  sR < sL < t\beta,\, (t\beta)
  \arrow{r}{W_-}[swap]{s}
  &
  sL < sR < t\beta,\, (t\beta)
\end{tikzcd}
&
\begin{tikzcd}
  sR < sL < t\beta,\, (t\beta)
  \arrow{r}{W_-}[swap]{s}
  \arrow{d}{C_-^\beta}[swap]{t}
  &
  sL < sR < t\beta,\, (t\beta)
  \arrow{d}{C_-^\beta}[swap]{t}
  \\
  sR < sL < t\betabar < t\betabar
  \arrow{r}{W_-}[swap]{s}
  &
  sL < sR < t\betabar < t\betabar
\end{tikzcd}
\end{array}
\end{eqnarray*}

\begin{lemma}
\label{lemma:zero-unless-I-N}
We have $\partial_T^+(\Box) = 0$ unless $\Box$ is interleaved or
nested.
\end{lemma}

\begin{proof}
For a pair of successive saddles $(\Box_T,\Box_R)$ or
$(\Box_L,\Box_B)$ in a square $\Box$ to give a nonzero contribution to
$\partial_T^+(\Box)$, one saddle must be type $W_-$ and one must be
type $C_\pm^L$.
We let $s$ denote the saddle of type $W_-$ and $t$ denote the saddle
of type $C_\pm^L$.
One corner of $\Box$ must therefore contain both
$s\alpha < s\alphabar$ and $t\betabar < t\betabar$.
We enumerate the six possible orderings of the attaching points of $s$
and $t$ for this corner and classify $\Box$ for each ordering:
\begin{align*}
  & t\betabar < t\betabar < s\alpha < s\alphabar &
  & D \\
  & t\betabar < s\alpha < t\betabar < s\alphabar &
  &
  \textup{$I_-$ if $\alpha = \beta = R$,
    $I_+$ if $\alpha = \beta = L$,
    not possible if $\alpha = \betabar$
    by Lemma \ref{lemma:impossible-pairs}} \\
  & t\betabar < s\alpha < s\alphabar < t\betabar &
  & \textup{not possible by Lemma \ref{lemma:impossible-nest}} \\
  & s\alpha < t\betabar < t\betabar < s\alphabar &
  &
  \textup{$N_\pm^\betabar$ if $\alpha = R$,
    $N_\pm^\beta$ if $\alpha = L$} \\
  & s\alpha < t\betabar < s\alphabar < t\betabar &
  &
  \textup{$I_-$ if $\alpha = \betabar = R$,
    $I_+$ if $\alpha = \betabar = L$,
    not possible if $\alpha = \beta$ by Lemma
    \ref{lemma:impossible-pairs}} \\
  & s\alpha < s\alphabar < t\betabar < t\betabar &
  & D
\end{align*}
If $\Box$ is disjoint then $\partial_T^+(\Box) = 0$, since the
contributions from $(\Box_T,\Box_R)$ and $(\Box_L,\Box_B)$ cancel.
\end{proof}

\section{Proof of the main theorem}
\label{sec:proof}

We are now ready to outline the proof of Theorem
\ref{theorem:main-intro} from the Introduction, which we restate here:

\begin{theorem}
\label{theorem:main}
The chain complex $(C_{T^\pm},\partial_{T^\pm})$ is homotopy
equivalent to the chain complex for the reduced Khovanov homology of
the link diagram $T^\pm$.
\end{theorem}

We construct the homotopy equivalence using the following lemma:

\begin{lemma}[Reduction Lemma]
\label{lemma:reduction}
If $(C,\partial)$ is a chain complex such that
$C = A \oplus B \oplus B$ and
$\partial:C \rightarrow C$ has the form
\begin{align*}
  \partial =
  \left(\begin{array}{ccc}
    \partial_A & \beta_1 & \beta_2 \\
    \alpha_1 & \partial_B & 0 \\
    \alpha_2 & \id_B & \partial_B
  \end{array}\right)
\end{align*}
relative to this decomposition, then
$(A, \partial_A + \beta_1\alpha_2)$ is a chain complex homotopy
equivalent to $(C,\partial)$.
\end{lemma}

\begin{proof}
The fact that $\partial^2 = 0$ implies that
$(\partial_A + \beta_1\alpha_2)^2 = 0$, so
$(A, \partial_A + \beta_1\alpha_2)$ is a chain complex.
Define linear maps
$F:C \rightarrow A$,
$G:A \rightarrow C$ and
$H:C \rightarrow C$ by
\begin{align*}
  F &= \left(\begin{array}{ccc}
    \id_A & 0 & \beta_1
  \end{array}\right), &
  G &= \left(\begin{array}{c}
    \id_A \\
    \alpha_2 \\
    0
  \end{array}\right), &
  H &= \left(\begin{array}{ccc}
    0 & 0 & 0 \\
    0 & 0 & \id_B \\
    0 & 0 & 0
  \end{array}\right).
\end{align*}
The fact that $\partial^2 = 0$ implies that $F$ and $G$ are chain
maps.
We find
\begin{align*}
  GF &= \id_C + \partial H + H \partial, &
  FG &= \id_A,
\end{align*}
so $F$ and $G$ are homotopy equivalences.
\end{proof}

We choose a disk $D$ in the annulus $S^1 \times [0,1]$ as shown
in Figure \ref{fig:T-and-Tbar} and say that a tangle diagram $T$ is
in \emph{standard position} if all the crossings of $T$ are contained
in $D$.
Given a tangle diagram $T$ in standard position, we define the
\emph{loop number} $\ell(T)$ as the number of times $T$ crosses the
overpass arc $A_s^+$ shown in Figure \ref{fig:T-and-Tbar}.

\begin{figure}
  \centering
  \includegraphics[scale=0.5]{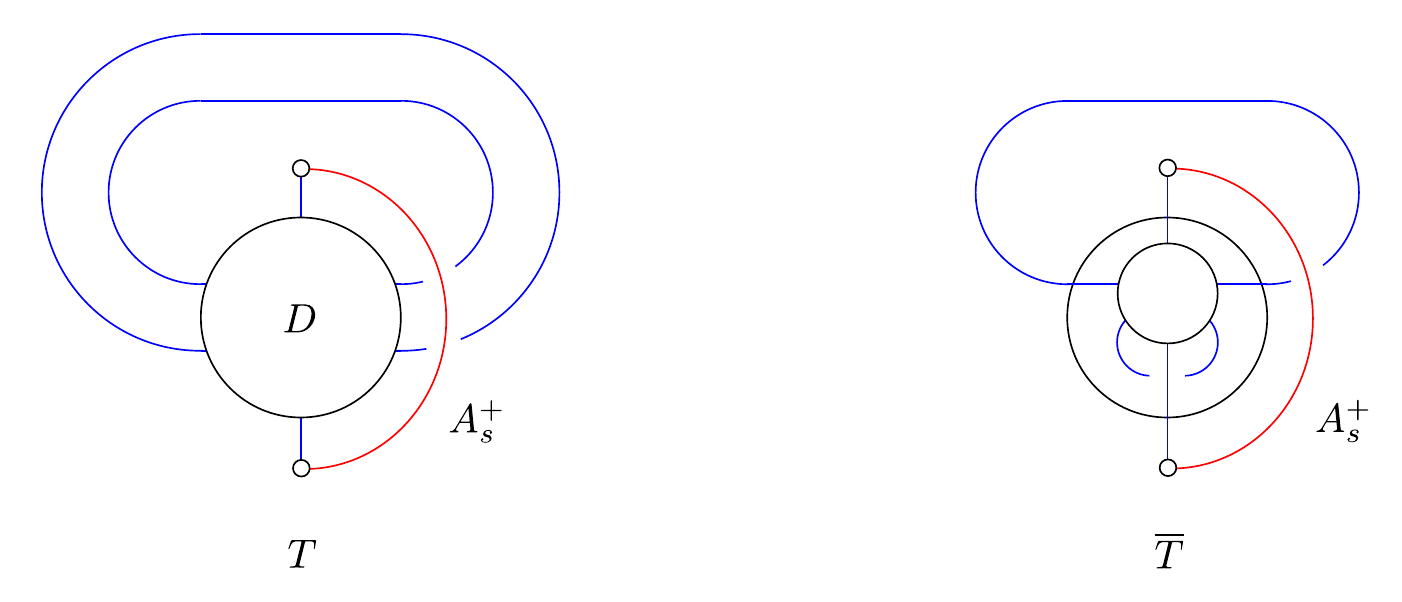}
\caption{
\label{fig:T-and-Tbar}
(Left)
A tangle diagram $T$ in standard position and the overpass arc
$A_s^+$.
All the crossings of $T$ are contained in the disk $D$.
(Right)
We flip the outermost loop of $T$ under the annulus and then perform
an isotopy of the annulus to obtain a tangle diagram $\Tbar$ in
standard position with one additional crossing.
}
\end{figure}

\begin{proof}[Proof of Theorem \ref{theorem:main}]

We prove the claim for $T^+$; the proof for $T^-$ is similar.
By performing an isotopy of the annulus and an isotopy of the overpass
arc, we can reduce to the case where the tangle diagram $T$ is in
standard position and the overpass arc is $A_s^+$.
The latter isotopy is always possible due to the condition we imposed
in Section \ref{sec:chain-complex} that the image of the overpass arc
under the projection $S^1 \times [0,1] \rightarrow S^1$ must be a
contractible loop, and the Khovanov chain complexes for the link
diagrams before and after the isotopies are homotopy equivalent.
We prove the claim by induction on the loop number $k$.

For the base case $k=0$, note that if $T$ is a tangle diagram in
standard position with loop number $\ell(T)=0$, then the chain
complex $(C_{T^+},\partial_{T^+}) = (C_{T^+},\partial_T^0)$ is the
Khovanov chain complex of $T^+ := T \cup A_s^+$ with marked point
$(b_0,0)$ or $(b_0,1)$, so the claim is trivially true.

For the induction step, assume the claim is true for each tangle
diagram $T$ in standard position with loop number $\ell(T) < k$, and
consider a tangle diagram $T$ in standard position with loop number
$\ell(T) = k$.
As shown in Figure \ref{fig:T-and-Tbar}, we can flip the outermost
loop of $T$ under the annulus and then perform an isotopy of the
annulus to obtain a tangle diagram $\Tbar$ in standard position with
one more crossing than $T$ and loop number $\ell(\Tbar) = k - 1$.
Since the link diagrams
$T^+ := T \cup A_s^+$ and $\Tbar^+ := \Tbar \cup A_s^+$ describe
isotopic links in $S^3$, the Khovanov chain complexes of $T^+$ and
$\Tbar^+$ are homotopy equivalent.
The induction hypothesis implies the Khovanov chain complex
of $\Tbar^+$ is homotopy equivalent to
$(C_{\Tbar^+},\partial_{\Tbar^+})$.
We will show that $(C_{\Tbar^+},\partial_{\Tbar^+})$ satisfies the
hypotheses of the Reduction Lemma \ref{lemma:reduction}, which we
apply to obtain a homotopy equivalent chain complex
$(C_{red},\partial_{red})$.
To complete the proof, we will show that
$(C_{red},\partial_{red}) = (C_{T^+},\partial_{T^+})$, so
$(C_{T^+},\partial_{T^+})$ is a chain complex, and the string of
homotopy equivalences implies that $(C_{T^+},\partial_{T^+})$ is
homotopy equivalent to the Khovanov chain complex of $T^+$.
\end{proof}

It remains to show that the \emph{induced chain complex}
$(C_{\Tbar^+},\partial_{\Tbar^+})$ for $\Tbar$
satisfies the hypotheses of the Reduction Lemma \ref{lemma:reduction}
and that $(C_{red},\partial_{red}) = (C_{T^+},\partial_{T^+})$.

\begin{remark}
\label{remark:pillowcase}
We can use the notion of standard position to describe an interesting
relationship between our chain complex and the symplectic
interpretation of Khovanov homology due to Hedden, Herald, Hogancamp,
and Kirk that was discussed in the Introduction.
In \cite{Hedden-3}
they consider link diagrams $T_D^0$ and $T_D^1$ obtained by closing a
2-tangle diagram $T_D$ in the disk as shown in Figure
\ref{fig:disk-tangles}.
They use the tangle diagram $T_D$ to construct an object of the
twisted Fukaya category of the character variety $R^*(S^2,4)$, which
in turn is used to construct chain complexes for the link diagrams
$T_D^0$ and $T_D^1$.
These chain complexes turn out to be precisely the Khovanov chain
complexes for $T_D^0$ and $T_D^1$.

\begin{figure}
  \centering
  \includegraphics[scale=0.5]{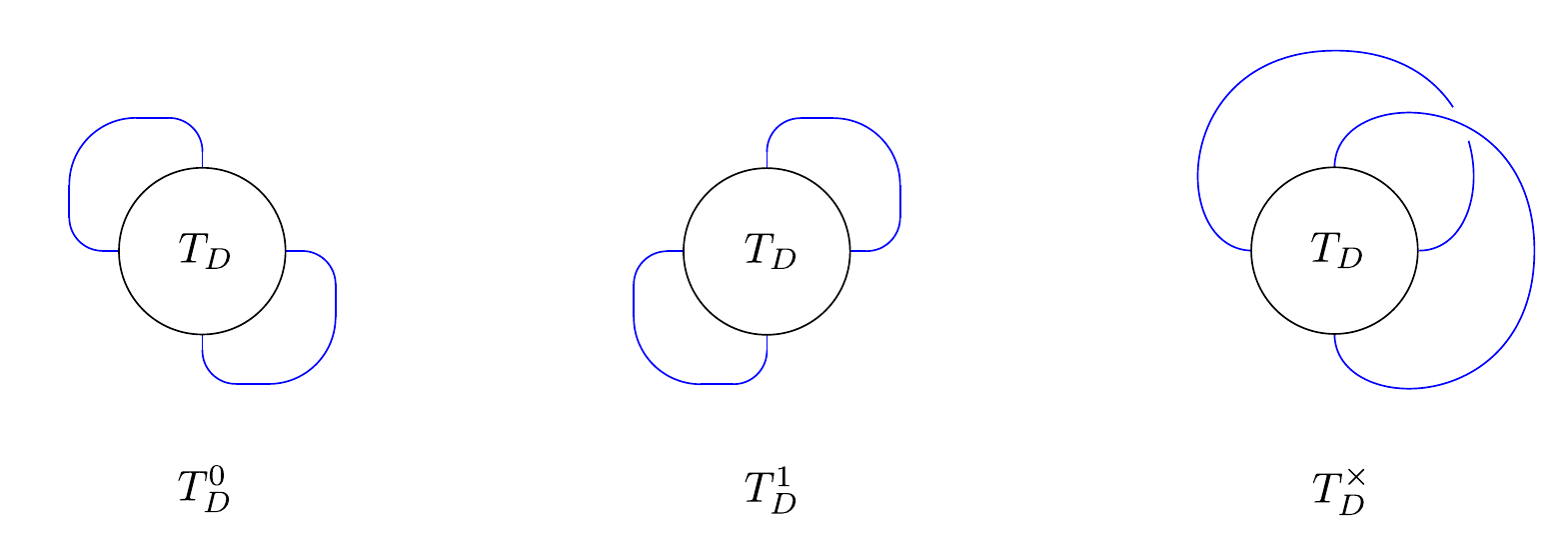}
\caption{
\label{fig:disk-tangles}
Link diagrams $T_D^0$, $T_D^1$, and $T_D^\times$ are constructed by
closing a 2-tangle diagram $T_D$ in the disk.
}
\end{figure}

Consider a 1-tangle diagram $T$ in the annulus in standard position
with loop number $\ell(T) = 1$.
If we restrict $T$ to the disk $D$ shown in Figure
\ref{fig:T-and-Tbar}, we obtain a 2-tangle diagram $T_D$ in the disk.
We close $T$ with the overpass arc $A_s^+$ shown in
Figure \ref{fig:T-and-Tbar} to obtain the link diagram
$T^+ := T \cup A_s^+$, which is identical to the link diagram
$T_D^\times$ shown in Figure \ref{fig:disk-tangles} that is obtained
by closing $T_D$ with a 2-tangle diagram with a single crossing.
In \cite{Boozer-Fukaya} we apply methods described in \cite{Hedden-3}
to construct a chain complex for $T_D^\times$ via the twisted Fukaya
category of $R^*(S^2,4)$, and we show that the resulting chain complex
is precisely our chain complex $(C_{T^+}, \partial_{T^+})$, provided
one properly assigns certain gradings to generators of the hom spaces
of the Fukaya category.
As explained in the Introduction, our chain complex
$(C_{T^+}, \partial_{T^+})$ was predicted via the Fukaya category of
$R^*(T^2,2)$.
The fact that the chain complexes for $T_D^\times$ and $T^+$
constructed via $R^*(S^2,4)$ and $R^*(T^2,2)$ are identical appears
to be due to a close relationship between the Fukaya categories of
these character varieties.
This relationship is discussed in \cite{Boozer-Fukaya}, which shows,
for example, that $R^*(S^2,4)$ is a symplectic submanifold of
$R^*(T^2,2)$.
\end{remark}

\section{Induced chain complex $(C_{\Tbar^+},\, \partial_{\Tbar^+})$}
\label{sec:induced-chain-complex}

We begin by analyzing the structure of the induced chain complex
$(C_{\Tbar^+},\, \partial_{\Tbar^+})$ for the tangle diagram $\Tbar$.
We first consider the vector space $C_{\Tbar^+}$.
For each planar tangle $p$ in the cube of resolutions for $T$, there
are two \emph{induced planar tangles} that we denote
$[p]_0$ and $[p]_1$ in the cube of resolutions for $\Tbar$, which are
obtained from $p$ by resolving the additional crossing of $\Tbar$
using the 0-resolution or 1-resolution.
Thus
\begin{align*}
  C_{\Tbar^+} &=
  \bigoplus_{i \in I} \left(C_{[T_i]_0}^+ \oplus C_{[T_i]_1}^+ \right) =
  \bigoplus_{i \in I} \Cbar_{T_i}^+,
\end{align*}
where we have grouped the vector spaces for each pair of induced
planar tangles into a single space:
\begin{align*}
  \Cbar_p^+ := C_{[p]_0}^+ \oplus C_{[p]_1}^+.
\end{align*}
For simplicity, we will initially ignore the bigrading structure of
the various vector spaces.
We denote the underlying ungraded vector space of a bigraded vector
space $V^+$ as $V$.
We will show that each vector space $\Cbar_p$ can be decomposed as
\begin{align*}
  \Cbar_p = A_p \oplus B_p \oplus B_p,
\end{align*}
where
\begin{align}
  \label{eqn:ap-cp}
  A_p = C_p
\end{align}
and either $B_p = C_p$ or $B_p = 0$.
We thus obtain a decomposition of $C_{\Tbar^+}$:
\begin{align}
  \label{eqn:C-Tbar-decomp}
  C_{\Tbar^+} &= \bigoplus_p \Cbar_p = A \oplus B \oplus B, &
  A &= \bigoplus_p A_p, &
  B &= \bigoplus_p B_p.
\end{align}

Next we consider the differential $\partial_{\Tbar^+}$ for the induced
chain complex, which is constructed from linear maps
corresponding to saddles in the cube of resolutions of $\Tbar$.
The saddles are of two different types:
\begin{itemize}
\item
For each planar tangle $p$ in the cube of resolutions of $T$, there is
a saddle $n_p:[p]_0 \rightarrow [p]_1$ in the cube of resolutions of
$\Tbar$ due to the additional crossing of $\Tbar$.
We call these \emph{ancillary saddles}.
We denote a linear map corresponding to an ancillary saddle using the
subscript $10$, as for example $Q_{10}(n_p) := Q(n_p)$.
Examples of ancillary saddles are shown in Figures
\ref{fig:ancillary-saddle-1A},
\ref{fig:ancillary-saddle-1B}, and
\ref{fig:ancillary-saddle-2}.

\item
For each saddle $s:p \rightarrow q$ in the cube of resolutions of $T$,
there are two \emph{induced saddles} 
$[s]_0:[p]_0 \rightarrow [q]_0$ and
$[s]_1:[p]_1 \rightarrow [q]_1$ in the cube of resolutions of $\Tbar$.
We denote linear maps corresponding to induced saddles $[s]_0$ and
$[s]_1$ using the subscripts $00$ and $11$, as for example
$Q_{00}(s) := Q([s]_0)$ and $Q_{11}(s) := Q([s]_1)$.
Examples of induced saddles are shown in Figure
\ref{fig:induced-saddles}.
\end{itemize}
It is useful to express $\partial_{\Tbar^+}$ as
\begin{align*}
  \partial_{\Tbar^+} :=
  \partial_\Tbar^0 + \partial_\Tbar^+ =
  \partialbar^0 + \partialbar^1 + \partialbar^2,
\end{align*}
where $\partialbar^0$, $\partialbar^1$, and $\partialbar^2$ collect
together the terms corresponding to different types of saddles.

\begin{figure}
  \centering
  \includegraphics[scale=0.5]{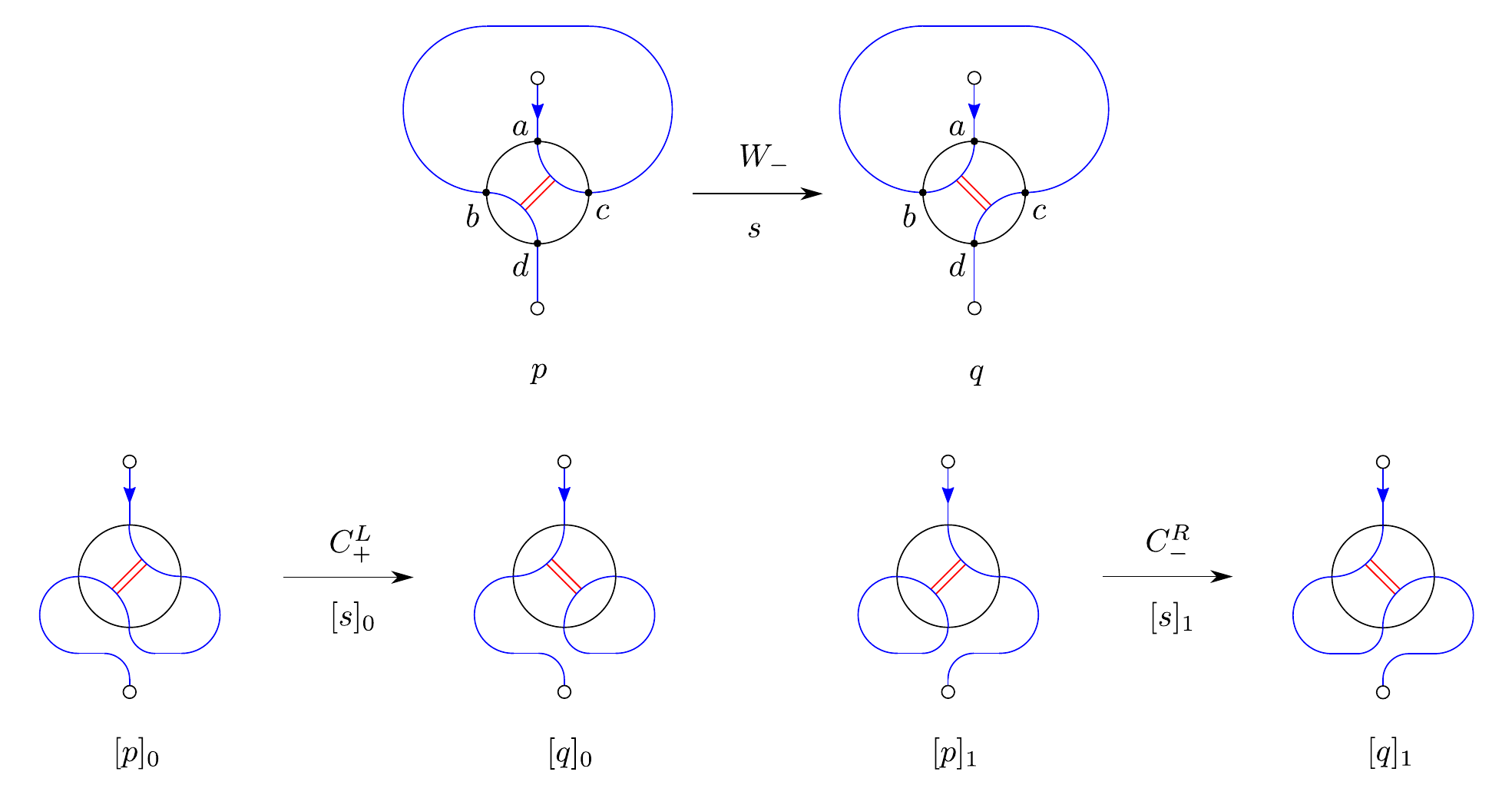}
\caption{
\label{fig:induced-saddles}
Example type $W_-$ saddle $s:p \rightarrow q$ in the cube of
resolutions of $T$ and the corresponding induced saddles
$[s]_0:[p]_0 \rightarrow [q]_0$ and
$[s]_1:[p]_1 \rightarrow [q]_1$ in the cube of resolutions of $\Tbar$.
}
\end{figure}

We define $\partialbar^0$ to be the part of $\partial_\Tbar^0$
corresponding to the ancillary saddles:
\begin{align*}
  \partialbar^0 = \sum_p T_{10}(n_p),
\end{align*}
where the sum is taken over all planar tangles in the cube of
resolutions of $T$.

Given a saddle $s:p \rightarrow q$ in the cube of resolutions of $T$,
we define a map
\begin{align}
  \label{eqn:dbar-1-s-def}
  \partialbar^1(s) =
  T_{00}(s) + T_{11}(s) +
  \Ttilde^L_{10}(n_q) Q_{00}(s) +
  Q_{10}(n_q) \Ttilde^L_{00}(s) +
  \Ttilde^L_{11}(s) Q_{10}(n_p) +
  Q_{11}(s) \Ttilde^L_{10}(n_p),
\end{align}
which collects the terms of $\partial_\Tbar^0$ corresponding to
the induced saddles of $s$, as well as the terms of $\partial_\Tbar^+$
corresponding to pairs of successive saddles, one of which is an
induced saddle of $s$ and one of which is an ancillary saddle.
We define
\begin{align*}
  \partialbar^1 = \sum_s \partialbar^1(s),
\end{align*}
where the sum is taken over all saddles in the cube of resolutions of
$T$.

Given a pair of successive saddles $(s,t)$ in the cube of resolutions
of $T$, we define a map
\begin{align}
  \label{eqn:dbar-2-t-s}
  \partialbar^2(s,t) =
  \Ttilde_{00}^L(t) Q_{00}(s) +
  Q_{00}(t)\Ttilde_{00}^L(s) +
  \Ttilde_{11}^L(t) Q_{11}(s) +
  Q_{11}(t)\Ttilde_{11}^L(s),
\end{align}
which collects the terms of $\partial_\Tbar^+$ corresponding to
pairs of successive saddles, one of which is an induced saddle of $s$
and one of which is an induced saddle of $t$.
Given a square $\Box$ in the cube of resolutions of $T$, we define a
map
\begin{align}
  \label{eqn:dbar-2-box-def}
  \partialbar^2(\Box) =
  \partialbar^2(\Box_T,\Box_R) +
  \partialbar^2(\Box_L,\Box_B),
\end{align}
which collects the terms of $\partial_\Tbar^+$ corresponding to
squares of saddles induced by $\Box$.
We define
\begin{align*}
  \partialbar^2 = \sum_\Box \partialbar^2(\Box),
\end{align*}
where the sum is taken over all squares in the cube of resolutions of
$T$.

In Section \ref{sec:saddle-0},
we compute the map $T_{10}(n_p)$ corresponding to the ancillary
saddle $n_p$ for each planar tangle $p$ in the cube of resolutions of
$T$ and show that
\begin{align}
  \label{eqn:dbar-0}
  &\partialbar^0 := \sum_p T_{10}(n_p) =
  \left(\begin{array}{ccc}
    0 & 0 & 0 \\
    0 & 0 & 0 \\
    0 & \id_B & 0
  \end{array}\right)
\end{align}
relative to the decomposition $C_{\Tbar^+} = A \oplus B \oplus B$ in
equation (\ref{eqn:C-Tbar-decomp}).
We also compute the maps $Q_{10}(n_p)$ and $\Ttilde_{10}^L(n_p)$ that
contribute to $\partialbar^1$.

In Section \ref{sec:saddle-1},
we compute the map $\partialbar^1(s)$ for each saddle $s$ in the
cube of resolutions of $T$ and show it has the form
\begin{align}
  \label{eqn:dbar-1-s-form}
  &\partialbar^1(s) =
  \left(\begin{array}{ccc}
    \partial_A(s) & \beta_1(s) & \beta_2(s) \\
    \alpha_1(s) & \partial_B(s) & 0 \\
    \alpha_2(s) & 0 & \partial_B(s)
  \end{array}\right)
\end{align}
relative to the decomposition $C_{\Tbar^+} = A \oplus B \oplus B$ in
equation (\ref{eqn:C-Tbar-decomp}).
We also compute the maps
$Q_{00}(s)$, $Q_{11}(s)$, $\Ttilde_{00}^L(s)$, and $\Ttilde_{11}^L(s)$
that contribute to $\partialbar^2$.

In Section \ref{sec:saddle-2},
we compute the map $\partialbar^2(\Box)$ for each square $\Box$ in
the cube of resolutions of $T$ and show it has the form
\begin{align}
  \label{eqn:dbar-2-box-form}
  &\partialbar^2(\Box) =
  \left(\begin{array}{ccc}
    \partial_A(\Box) & 0 & \beta_2(\Box) \\
    \alpha_1(\Box) & \partial_B(\Box) & 0 \\
    0 & 0 & \partial_B(\Box)
  \end{array}\right).
\end{align}
relative to the decomposition $C_{\Tbar^+} = A \oplus B \oplus B$ in
equation (\ref{eqn:C-Tbar-decomp}).

From equations (\ref{eqn:dbar-0}), (\ref{eqn:dbar-1-s-form}), and
(\ref{eqn:dbar-2-box-form}), it follows that $\partial_\Tbar^+$ has
the form
\begin{align*}
  \partial_{\Tbar^+} =
  \left(\begin{array}{ccc}
    \partial_A & \beta_1 & \beta_2 \\
    \alpha_1 & \partial_B & 0 \\
    \alpha_2 & \id_B & \partial_B
  \end{array}\right),
\end{align*}
so we can apply the Reduction Lemma \ref{lemma:reduction} to obtain a
chain complex $(C_{red}, \partial_{red})$.
From equation (\ref{eqn:ap-cp}) it follows that
\begin{align}
  \label{eqn:c-red-proof}
  C_{red} := A = \sum_p A_p = \sum_p C_p = C_{T^+}
\end{align}
as ungraded vector spaces, and in Appendix \ref{appendix:bigradings}
we show that equation (\ref{eqn:c-red-proof}) also holds when we
include the bigrading structure.
From equations (\ref{eqn:dbar-0}), (\ref{eqn:dbar-1-s-form}), and
(\ref{eqn:dbar-2-box-form}), it follows that
\begin{align}
  \label{eqn:d-red}
  \partial_{red} :=
  \partial_A + \beta_1\alpha_2 =
  \sum_s \partial_A(s) +
  \sum_\Box \partial_{red}(\Box),
\end{align}
where we have defined
\begin{align}
  \label{eqn:d-red-box}
  \partial_{red}(\Box) :=
  \partial_A(\Box) +
  \beta_1(\Box_R)\alpha_2(\Box_T) +
  \beta_1(\Box_B)\alpha_2(\Box_L).
\end{align}
In Section \ref{sec:saddle-1} we show
\begin{align}
  \label{eqn:d-1}
  \partial_A(s) = T(s).
\end{align}
In Section \ref{sec:saddle-2} we show
\begin{align}
  \label{eqn:d-2}
  \partial_{red}(\Box) = \partial_T^+(\Box).
\end{align}
We substitute equations (\ref{eqn:d-1}) and (\ref{eqn:d-2}) into
equation (\ref{eqn:d-red}) to obtain
\begin{align}
  \label{eqn:d-red-proof}
  \partial_{red} =
  \sum_s T(s) + \sum_\Box \partial_T^+(\Box) =
  \partial_T^0 + \partial_T^+ =
  \partial_{T^+},
\end{align}
where we have expressed $\partial_T^+$ as a sum over squares as
described in Section \ref{sec:squares}.
Equations (\ref{eqn:c-red-proof}) and (\ref{eqn:d-red-proof}) show
that $(C_{red},\partial_{red}) = (C_{T^+}, \partial_{T^+})$,
thus completing the proof of Theorem \ref{theorem:main}.

\section{Induced planar tangles and ancillary saddles}
\label{sec:saddle-0}

Given a planar tangle diagram $T$ in standard position, we let $a$ and
$d$ denote the points where $T$ intersects the disk $D$ on the top and
bottom, and we let $b$ and $c$ denote the points where the outermost
loop of $T$ intersects $D$ on the left and right.
We classify the planar tangles in the cube of resolutions of $T$ into
three types:

\begin{itemize}
\item
We say a planar tangle $p$ is type $\ptb$
if $b$ and $c$ lie on the arc component of $p$ and $b < c$, so the
outermost loop of $p$ belongs to the arc component and is oriented
clockwise.

\item
We say a planar tangle $p$ is type $\ptc$
if $b$ and $c$ lie on the arc component of $p$ and $c < b$, so the
outermost loop of $p$ belongs to the arc component and is oriented
counterclockwise.

\item
We say a planar tangle $p$ is type $\pto$
if $b$ and $c$ lie on a circle component of $p$, so the outermost loop
of $p$ belongs to a circle component.
\end{itemize}

Examples of type $\ptb$, $\ptc$, and $\pto$ planar tangles are shown
in Figures \ref{fig:ancillary-saddle-1A},
\ref{fig:ancillary-saddle-1B}, and \ref{fig:ancillary-saddle-2}.
For each planar tangle $p$ in the cube of resolutions of $T$, we
specify a decomposition $\Cbar_p = A_p \oplus B_p \oplus B_p$ and we
determine the maps $\Ttilde_{10}^L(n_p)$ and $Q_{10}(n_p)$
corresponding to the ancillary saddle $n_p:[p]_0 \rightarrow [p]_1$
that contribute to $\partialbar^1$.
We will use the notation $p(\ptb)$, $p(\ptc)$, and $p(\pto)$ to
indicate that a planar tangle $p$ in the cube of resolutions of $T$ is
type $\ptb$, $\ptc$, or $\pto$.

\subsection{Type $\ptb$ planar tangles}

A planar tangle $p$ of type $\ptb$ has the form $a < b < c < d$, as
shown in Figure \ref{fig:ancillary-saddle-1A}.
The ancillary saddle $n_p:[p]_0 \rightarrow [p]_1$
merges a circle with the left side of the arc component of $[p]_0$,
and is thus type $C_-^L$:
\begin{align*}
  & n_p:[p]_0 \rightarrow [p]_1: &
  & a < b < n_p L,\, (d < c < n_p L) &
  & \longrightarrow &
  & a < b < n_p R < d < c < n_p R.
\end{align*}
So as ungraded vector spaces we have
\begin{align*}
  & C_{[p]_0} = C_p \otimes A^{dc}, &
  & C_{[p]_1} = C_p,
\end{align*}
where $C_p = A^{\otimes c(p)}$ and $c(p)$ is the circle number of $p$.
The superscript on $A^{dc}$ indicates that this factor corresponds to
the additional circle component of $[p]_0$ that contains points $d$
and $c$.
We decompose $\Cbar_p$ as
\begin{align*}
  \Cbar_p := C_{[p]_0} \oplus C_{[p]_1} =
  A_p \oplus B_p^1 \oplus B_p^2,
\end{align*}
where
\begin{align}
  \label{eqn:decomp-1A}
  & A_p := C_p = C_p \otimes x \subset C_{[p]_0}, &
  & B_p^1 := C_p = C_p \otimes e \subset C_{[p]_0}, &
  & B_p^2 := C_p = C_{[p]_1}.
\end{align}
The superscripts on $B_p^1$ and $B_p^2$ indicate the ordering of
these summands in the decomposition of $\Cbar_p$.
In the definitions of $A_p$ and $B_p^1$, we have used the fact that
$C_p$ is canonically isomorphic to the subspaces $C_p \otimes x$ and
$C_p \otimes e$ of $C_{[p]_0} = C_p \otimes A^{dc}$.

\begin{figure}
  \centering
  \includegraphics[scale=0.5]{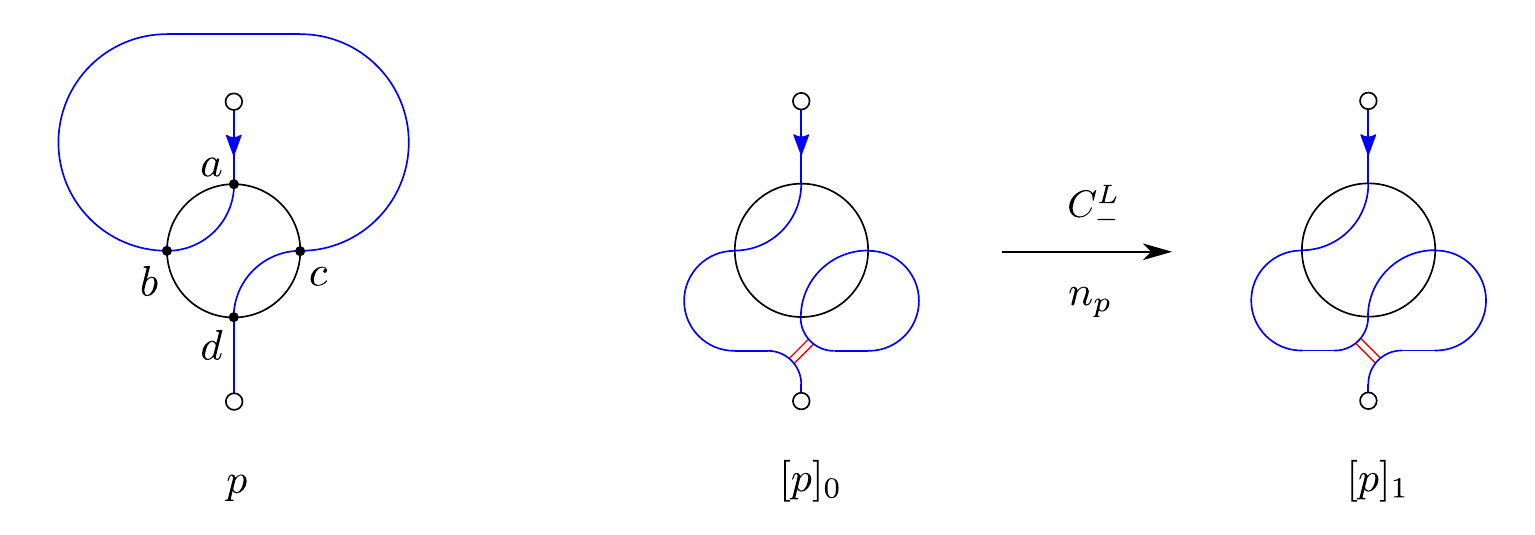}
\caption{
\label{fig:ancillary-saddle-1A}
Example type $\ptb$ planar tangle $p$ in the cube of resolutions of $T$
and the corresponding induced planar tangles $[p]_0$ and $[p]_1$ and
ancillary saddle $n_p:[p]_0 \rightarrow [p]_1$ in the cube of resolutions
of $\Tbar$.
}
\end{figure}

\begin{lemma}
\label{lemma:1A}
For a planar tangle $p$ of type $\ptb$, the ancillary saddle
$n_p:[p]_0 \rightarrow [p]_1$ contributes the maps
\begin{align*}
  T_{10}(n_p) =
  \left(\begin{array}{ccc}
    0 & 0 & 0 \\
    0 & 0 & 0 \\
    0 & \id_{C_p} & 0
  \end{array}\right), &
  &\Ttilde_{10}^L(n_p) =
  \left(\begin{array}{ccc}
    0 & 0 & 0 \\
    0 & 0 & 0 \\
    \id_{C_p} & 0 & 0
  \end{array}\right).
\end{align*}
\end{lemma}

\begin{proof}
Recall that the ancillary saddle $n_p$ is type $C_-^L$.
The corresponding linear maps
$T_{10}(n_p)$ and $\Ttilde_{10}^L(n_p)$
are nonzero only in the block
$C_{[p]_0} = C_p \otimes A^{dc} \rightarrow C_{[p]_1} = C_p$,
and their restrictions to that block are
$\id_{C_p} \otimes \epsilondot$
and $\id_{C_p} \otimes \epsilon$, respectively.
From the definitions of $A_p$, $B_p^1$, and $B_p^2$ in equation
(\ref{eqn:decomp-1A}), it follows that
$\id_{C_p} \otimes \epsilondot$ is the identity map $\id_{C_p}$ from
$B_p^1 = C_p$ to $B_p^2 = C_p$, and
$\id_{C_p} \otimes \epsilon$ is the identity map $\id_{C_p}$ from
$A_p = C_p$ to $B_p^2 = C_p$.
\end{proof}

\subsection{Type $\ptc$ planar tangles}

A planar tangle $p$ of type $\ptc$ has the form $a < c < b < d$, as
shown in Figure \ref{fig:ancillary-saddle-1B}.
The ancillary saddle $n_p:[p]_0 \rightarrow [p]_1$ splits a circle
from the right side of the arc component of $[p]_0$, and is thus type
$C_+^R$:
\begin{align*}
  & n_p:[p]_0 \rightarrow [p]_1: &
  & a < c < n_p L < d < b < n_p L &
  & \longrightarrow &
  & a < c < n_p R,\, (d < b < n_p R).
\end{align*}
So as ungraded vector spaces we have
\begin{align*}
  & C_{[p]_0} = C_p, &
  & C_{[p]_1} = C_p \otimes A^{db},
\end{align*}
where $C_p = A^{\otimes c(p)}$ and $c(p)$ is the circle number of $p$.
The superscript on $A^{db}$ indicates that this factor corresponds to
the additional circle component of $[p]_1$ that contains points $d$ and
$b$.
We decompose $\Cbar_p$ as
\begin{align*}
  \Cbar_p := C_{[p]_0} \oplus C_{[p]_1} =
  A_p \oplus B_p^1 \oplus B_p^2,
\end{align*}
where
\begin{align}
  \label{eqn:decomp-1B}
  & A_p := C_p = C_p \otimes e \subset C_{[p]_1}, &
  & B_p^1 := C_p = C_{[p]_0}, &
  & B_p^2 := C_p = C_p \otimes x \subset C_{[p]_1}.
\end{align}
The superscripts on $B_p^1$ and $B_p^2$ indicate the ordering of
these summands in the decomposition of $\Cbar_p$.
In the definitions of $A_p$ and $B_p^2$, we have used the fact that
$C_p$ is canonically isomorphic to the subspaces
$C_p \otimes x$ and $C_p \otimes e$ of $C_{[p]_1} = C_p \otimes A^{db}$.

\begin{figure}
  \centering
  \includegraphics[scale=0.5]{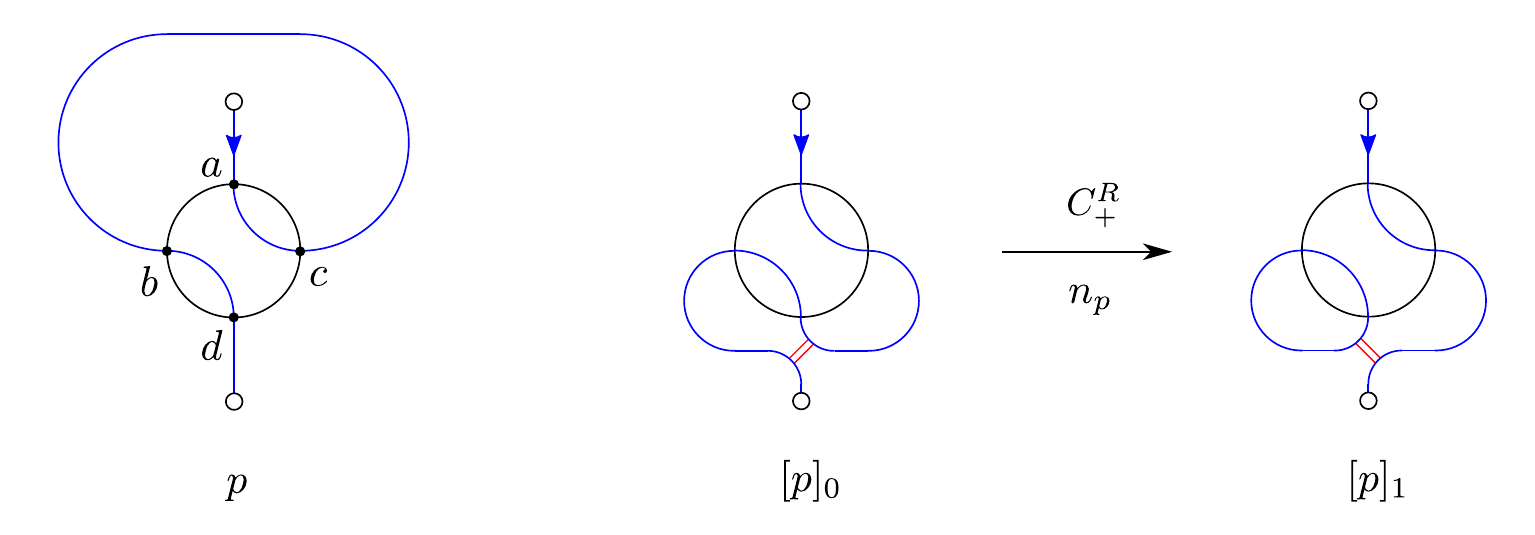}
\caption{
\label{fig:ancillary-saddle-1B}
Example type $\ptc$ planar tangle $p$ in the cube of resolutions of $T$
and the corresponding induced planar tangles $[p]_0$ and $[p]_1$ and
ancillary saddle $n_p:[p]_0 \rightarrow [p]_1$ in the cube of resolutions
of $\Tbar$.
}
\end{figure}

\begin{lemma}
\label{lemma:1B}
For $p$ of type $\ptc$, the ancillary saddle $n_p$ contributes the map
\begin{align*}
  T_{10}(n_p) =
  \left(\begin{array}{ccc}
    0 & 0 & 0 \\
    0 & 0 & 0 \\
    0 & \id_{C_p} & 0
  \end{array}\right).
\end{align*}
\end{lemma}

\begin{proof}
Recall that the ancillary saddle $n_p$ is type $C_+^R$.
The corresponding linear map
$T_{10}(n_p)$
is nonzero only in the block
$C_{[p]_0} = C_p \rightarrow C_{[p]_1} = C_p \otimes A^{db}$,
and its restriction to that block is
$\id_{C_p} \otimes \etadot$.
From the definitions of $A_p$ and $B_p^1$ in equation
(\ref{eqn:decomp-1A}), it follows that
$\id_{C_p} \otimes \etadot$ is the identity map $\id_{C_p}$ from
$B_p^1 = C_p$ to $B_p^2 = C_p$.
\end{proof}

\subsection{Type $\pto$ planar tangles}

A planar tangle $p$ of type $\pto$ has the form $a < d,\,(b < c)$, as
shown in Figure \ref{fig:ancillary-saddle-2}.
The ancillary saddle $n_p$ lowers the winding number by two, and is
thus type $W_-$:
\begin{align*}
  & n_p:[p]_0 \rightarrow [p]_1: &
  & a < d < n_p R < c < b < n_p L &
  & \longrightarrow &
  & a < d < n_p L < b < c < n_p R.
\end{align*}
If $p$ has circle number $c(p) = r + 1$ then $[p]_0$ and $[p]_1$ have
circle number $c([p]_0) = c([p]_1) = r$.
We let $A^{bc}$ denote the vector space factor corresponding to the
additional circle component of $p$ that contains $b$ and $c$, and we
define a vector space $W_p = A^{\otimes r}$ corresponding to the
remaining circle components of $p$.
As ungraded vector spaces, we have
\begin{align*}
  & C_p = W_p \otimes A^{bc}, &
  & C_{[p]_0} = W_p = W_p \otimes e \subset C_p, &
  & C_{[p]_1} = W_p = W_p \otimes x \subset C_p.
\end{align*}
We have used the fact that $W_p$ is canonically isomorphic to the
subspaces $W_p \otimes e$ and $W_p \otimes e$ of
$W_p \otimes A^{bc}$ to identify $C_{[p]_0}$ and $C_{[p]_1}$ with
subspaces of $C_p$.
We decompose $\Cbar_p$ as
\begin{align*}
  \Cbar_p := C_{[p]_0} \oplus C_{[p]_1} =
  A_p \oplus B_p^1 \oplus B_p^2,
\end{align*}
where
\begin{align}
  \label{eqn:decomp-2}
  & A_p := C_p = W_p \otimes A^{bc}, &
  & B_p^1 := 0, &
  & B_p^2 := 0.
\end{align}
Since we are viewing $C_{[p]_0}$ and $C_{[p]_1}$ as subspaces of
$C_p = A_p$, we can define a map
\begin{align*}
  &\id_{C_{[p]_1} C_{[p]_0}}:A_p \rightarrow A_p, &
  \id_{C_{[p]_1} C_{[p]_0}} = \id_{W_p} \otimes \id_{xe}.
\end{align*}
We extend $\id_{C_{[p]_1} C_{[p]_0}}$ by zero to obtain a map
$A \rightarrow A$, where $A = \bigoplus_p A_p$.

\begin{figure}
  \centering
  \includegraphics[scale=0.5]{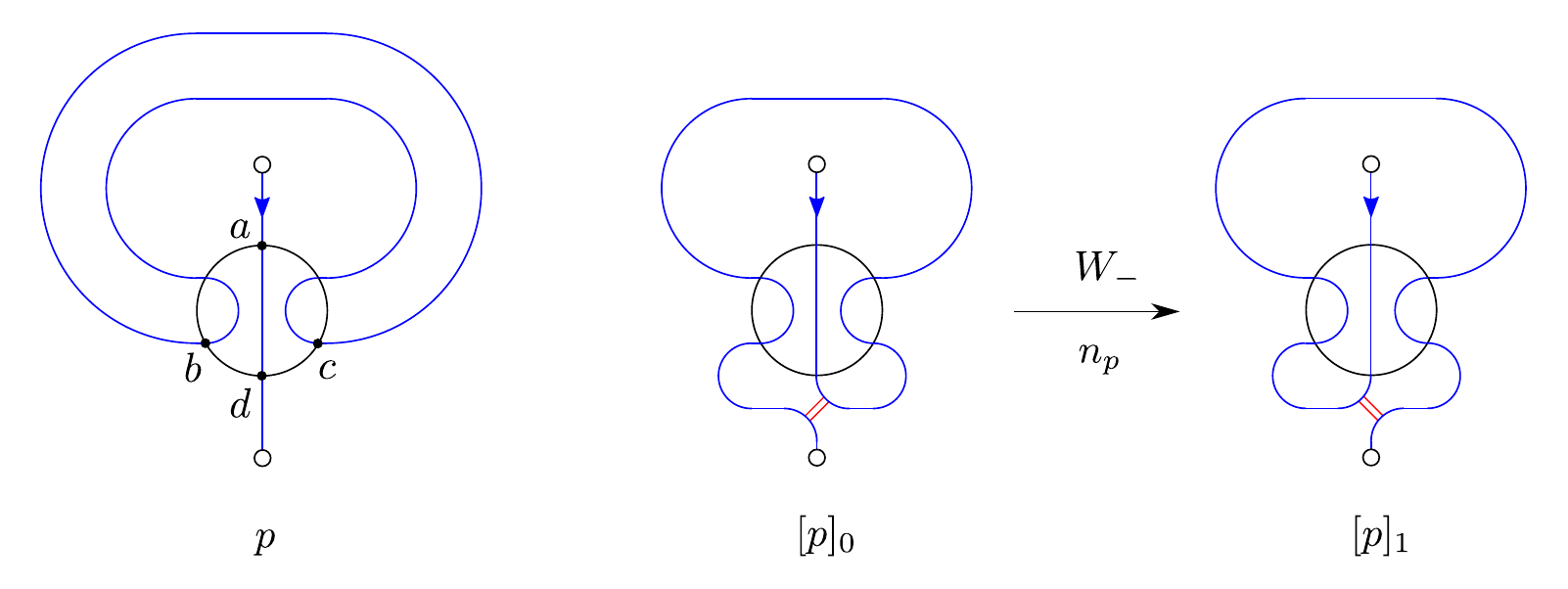}
\caption{
\label{fig:ancillary-saddle-2}
Example type $\pto$ planar tangle $p$ in the cube of resolutions of $T$
and the corresponding induced planar tangles $[p]_0$ and $[p]_1$ and
ancillary saddle $n_p:[p]_0 \rightarrow [p]_1$ in the cube of resolutions
of $\Tbar$.
}
\end{figure}

\begin{lemma}
\label{lemma:2}
For a planar tangle $p$ of type $\pto$, the ancillary saddle
$n_p:[p]_0 \rightarrow [p]_1$ contributes the map
\begin{align*}
  Q_{10}(n_p) =
  \left(\begin{array}{ccc}
    \id_{C_{[p]_1} C_{[p]_0}} & 0 & 0 \\
    0 & 0 & 0 \\
    0 & 0 & 0
  \end{array}\right).
\end{align*}
\end{lemma}

\begin{proof}
Recall that the ancillary saddle $n_p$ is of type $W_-$.
The corresponding linear map $Q_{10}(n_p)$ is nonzero only in
the block
$C_{[p]_0} = W_p \otimes e \rightarrow C_{[p]_1} = W_p \otimes x$,
and its restriction to that block is
$\id_{W_p} \otimes \id_{xe} = \id_{C_{[p]_1 [p]_0}}$.
The result now follows from the definitions of $A_p$, $B_p^1$, and
$B_p^2$ in equation (\ref{eqn:decomp-2}).
\end{proof}

\subsection{Summary}

Lemmas \ref{lemma:1A}, \ref{lemma:1B}, and \ref{lemma:2}, together
with the definitions of $A_p$, $B_p^1$, and $B_p^2$ in equations
(\ref{eqn:decomp-1A}), (\ref{eqn:decomp-1B}), and
(\ref{eqn:decomp-2}), prove

\begin{lemma}
We have
\begin{align*}
  &\partialbar^0 := \sum_p T_{10}(n_p) =
  \left(\begin{array}{ccc}
    0 & 0 & 0 \\
    0 & 0 & 0 \\
    0 & \id_B & 0
  \end{array}\right).
\end{align*}
\end{lemma}

\section{Terms $\partialbar^1(s)$}
\label{sec:saddle-1}

For each saddle $s:p\rightarrow q$ in the cube of resolutions of $T$,
we compute the map $\partialbar^1(s)$ defined in equation
(\ref{eqn:dbar-1-s-def}), which collects the terms of
$\partial_\Tbar^0$ corresponding to the induced saddles of $s$, as
well as the terms of $\partial_\Tbar^+$ corresponding to pairs of
successive saddles, one of which is an induced saddle of $s$ and one
of which is an ancillary saddle.
The form of the map $\partialbar^1(s)$ depends on the relative
positions of the points $b$ and $c$ and the attaching points of
$s$, so we divide the saddle types $W_\pm$,
$C_\pm^\beta$, and $C_\pm^C$ into subtypes according to the ordering
of these points as described in Table \ref{table:saddle-subtypes}.
The following three lemmas describe restrictions on the subtypes of
saddles that are actually possible:

\begin{table}
\begin{align*}
  &W1_\pm: &
  &x_1 < x_2 < sL < sR &
  &\longleftrightarrow &
  &x_1 < x_2 < sR < sL &
  & \textup{(not possible by Lemma \ref{lemma:w1-impossible})} \\
  &W2_\pm: &
  &sL < x_1 < x_2 < sR &
  &\longleftrightarrow &
  &sR < x_2 < x_1 < sL &
  & \textup{(restricted by Lemma \ref{lemma:w2-impossible})} \\
  &W3_\pm: &
  &sL < sR < x_1 < x_2 &
  &\longleftrightarrow &
  &sR < sL < x_1 < x_2 \\
  &W4_\pm: &
  &sL < sR,\,(x_1 < x_2) &
  &\longleftrightarrow &
  &sR < sL,\,(x_1 < x_2) \\
  \\
  &C1_\pm^\beta: &
  &x_1 < x_2 < s\betabar < s\betabar &
  &\longleftrightarrow &
  &x_1 < x_2 < s\beta,\,(s\beta) \\
  &C2_\pm^\beta: &
  &s\betabar < x_1 < x_2 < s\betabar &
  &\longleftrightarrow &
  &s\beta,\,(x_1 < x_2 < s\beta) &
  & \textup{(restricted by Lemma \ref{lemma:c2-impossible})} \\
  &C3_\pm^\beta: &
  &s\betabar < s\betabar < x_1 < x_2 &
  &\longleftrightarrow &
  &s\beta < x_1 < x_2,\,(s\beta) \\
  &C4_\pm^\beta: &
  &s\betabar < s\betabar,\,(x_1 < x_2) &
  &\longleftrightarrow &
  &s\beta,\,(s\beta),\,(x_1 < x_2) \\
  \\
  &C1_\pm^C: &
  &(x_1 < x_2 < s\betabar < s\betabar) &
  &\longleftrightarrow &
  &(s\beta),\,(x_1 < x_2 < s\beta) \\
  &C2_\pm^C: &
  &(x_1 < x_2),\, (s\betabar < s\betabar) &
  &\longleftrightarrow &
  &(x_1 < x_2),\, (s\beta),\, (s\beta) \\
  &C3_\pm^C: &
  &x_1 < x_2,\,(s\betabar < s\betabar) &
  &\longleftrightarrow &
  &x_1 < x_2,\,(s\beta),\,(s\beta).
\end{align*}
\caption{
\label{table:saddle-subtypes}
Saddle subtypes.
Subscripts $+$ and $-$ correspond to saddles
$p \rightarrow q$ and $p \leftarrow q$.
The pair $(x_1,x_2)$ is either $(b,c)$ or $(c,b)$.
Saddles of type $W1_\pm$ are not possible by Lemma
\ref{lemma:w1-impossible}, and saddles of type $W2_\pm$ and
$C2_\pm^\beta$ are restricted by Lemmas
\ref{lemma:w2-impossible} and \ref{lemma:c2-impossible}.
}
\end{table}

\begin{lemma}
\label{lemma:w1-impossible}
Type $W1_\pm$ saddles are not possible.
\end{lemma}

\begin{proof}
We will prove the claim for a type $W1_\pm$ saddle
$s:p(\ptb) \rightarrow q(\ptb)$.
We have
\begin{align*}
  & s:p \rightarrow q: &
  & a < b < c < s\alpha < s\alphabar < d &
  & \longrightarrow &
  & a < b < c < s\alphabar < s\alpha < d \\
  & [s]_0:[p]_0 \rightarrow [q]_0: &
  & a < b,\, (c < s\alpha < s\alphabar < d) &
  & \longrightarrow &
  & a < b,\, (c < s\alphabar < s\alpha < d).
\end{align*}
So the saddle $[s]_0$ attaches to opposite sides of the circle
component $(c < s\alpha < s\alphabar < d)$ in $[p]_0$, which is not
possible.
The proof for $s:p(\ptc) \rightarrow q(\ptc)$ is similar.
\end{proof}

\begin{lemma}
\label{lemma:w2-impossible}
Type $W2_+$ saddles $p(\ptc) \rightarrow q(\ptb)$ and
type $W2_-$ saddles $p(\ptb) \rightarrow q(\ptc)$
are not possible.
\end{lemma}

\begin{proof}
A type $W2_+$ saddle $p \rightarrow q$ flips the orientation of the
outermost loop of $p$ from clockwise to counterclockwise, so $p$ must
be type $\ptb$.
A type $W2_-$ saddle $p \rightarrow q$ flips the orientation of the
outermost loop of $p$ from counterclockwise to clockwise, so $p$ must
be type $\ptc$.
\end{proof}

\begin{lemma}
\label{lemma:c2-impossible}
Type $C2_+^L$ saddles $p(\ptc) \rightarrow q(\pto)$,
type $C2_-^L$ saddles $p(\pto) \rightarrow q(\ptc)$,
type $C2_+^R$ saddles $p(\ptb) \rightarrow q(\pto)$, and
type $C2_-^R$ saddles $p(\pto) \rightarrow q(\ptb)$
are not possible.
\end{lemma}

\begin{proof}
For a type $C2_+^\beta$ saddle $s:p \rightarrow q$, the planar tangle
$q$ has a circle component $(x_1 < x_2 < s\beta)$.
The saddle $s$ must attach to the \emph{outside} of this circle, so
$\beta = L$ if the circle is oriented clockwise and $\beta = R$ if the
circle is oriented counterclockwise.
The points $x_1$ and $x_2$ lie on the boundary of the disk $D$ and the
attaching point $s\beta$ lies inside the disk $D$, so if the circle
is oriented clockwise then $(x_1 < x_2 < s\beta) = (b < c < sL)$ and
if the circle is oriented counterclockwise then
$(x_1 < x_2 < s\beta) = (c < b < sR)$.
The proof for a type $C2_-^\beta$ saddle is similar.
\end{proof}

For each of the saddle subtypes in Table \ref{table:saddle-subtypes},
we compute the types of the corresponding induced saddles.
For example, consider a type $W2_-$ saddle
$s:p(\ptc) \rightarrow q(\ptb)$, as shown in Figure
\ref{fig:induced-saddles}:
\begin{align*}
  & s:p \rightarrow q: &
  & a < sR < c < b < sL < d&
  & \longrightarrow &
  & a < sL < b < c < sR < d \\
  & [s]_0:[p]_0 \rightarrow [q]_0: &
  & a < sR < c < d < sR < b &
  & \longrightarrow &
  & a < sL < b,\, (c < d < sL) \\
  & [s]_1:[p]_1 \rightarrow [q]_1: &
  & a < sR < c,\, (b < d < sR) &
  & \longrightarrow &
  & a < sL < b < d < sL < c.
\end{align*}
We see that $[s]_0$ is type $C_+^L$ and $[s]_1$ is type $C_-^R$.
We perform similar calculations for the remaining saddle subtypes and
display the results in Table \ref{table:summary-saddles}.

\begin{table}
\begin{align*}
  \begin{array}{lllllll}
    s & p \rightarrow q & [s]_0 & [s]_1 &
    \partial_A(s) & \beta_1(s) & \alpha_2(s) \\
    W2_- & p(\ptc) \rightarrow q(\ptb) & C_+^L & C_-^R &
    0 & Q(s) & Q(s) \\
    W2_+ & p(\ptb) \rightarrow q(\ptc) & C_-^L & C_+^R &    
    0 & 0 & 0 \\
    \\
    W3_\pm & p(\ptb) \rightarrow q(\ptb) & W_\pm & W_\pm & 0 & 0 & 0 \\
    W3_\pm & p(\ptc) \rightarrow q(\ptc) & W_\pm & W_\pm & 0 & 0 & 0 \\
    \\
    W4_\pm & p(\pto) \rightarrow q(\pto) & W_\pm & W_\pm & 0 & 0 & 0 \\
    \\
    C1_\pm^\beta & p(\ptb) \rightarrow q(\ptb) & C_\pm^C & C_\pm^\betabar &
    T(s) & \Ttilde(s) & 0 \\
    C1_\pm^\beta & p(\ptc) \rightarrow q(\ptc) & C_\pm^\betabar & C_\pm^C &
    T(s) & 0 & \Ttilde(s) \\
    \\
    C2_+^L & p(\ptb) \rightarrow q(\pto) & C_-^R & W_- &
    T(s) &
    \Ttilde^L(s) & 0 \\
    C2_-^L & p(\pto) \rightarrow q(\ptb) & C_+^R & W_+ &
    T(s) & 0 & 0 \\
    C2_+^R & p(\ptc) \rightarrow q(\pto) & W_+ & C_-^L &
    T(s) & 0 & 0 \\
    C2_-^R & p(\pto) \rightarrow q(\ptc) & W_- & C_+^L &
    T(s) & 0 &
    \Ttilde^R(s) \\
    \\
    C3_\pm^\beta & p(\ptb) \rightarrow q(\ptb) &
    C_\pm^\beta & C_\pm^\beta & T(s) & 0 & 0 \\
    C3_\pm^\beta & p(\ptc) \rightarrow q(\ptc) &
    C_\pm^\beta & C_\pm^\beta & T(s) & 0 & 0 \\
    \\
    C4_\pm^\beta & p(\pto) \rightarrow q(\pto) &
    C_\pm^\beta & C_\pm^\beta & T(s) & 0 & 0 \\
    \\
    C1_\pm^C & p(\pto) \rightarrow q(\pto) &
    C_\pm^\sigma & C_\pm^\sigmabar & T(s) & 0 & 0 \\
    \\
    C2_\pm^C & p(\pto) \rightarrow q(\pto) &
    C_\pm^C & C_\pm^C & T(s) & 0 & 0 \\
    \\
    C3_\pm^C & p(\ptb) \rightarrow q(\ptb) &
    C_\pm^C & C_\pm^C & T(s) & 0 & 0 \\
    C3_\pm^C & p(\ptc) \rightarrow q(\ptc) &
    C_\pm^C & C_\pm^C & T(s) & 0 & 0
  \end{array}
\end{align*}
\caption{
\label{table:summary-saddles}
Induced saddle types.
For each saddle $s:p \rightarrow q$, we list the types of the
induced saddles $[s]_0$ and $[s]_1$ and the blocks
$\partial_A(s)$, $\beta_1(s)$, and $\alpha_2(s)$ of
$\partialbar^1(s)$.
}
\end{table}

Using the types of the induced saddles summarized in Table
\ref{table:summary-saddles}, we compute the map $\partialbar^1(s)$
described in equation (\ref{eqn:dbar-1-s-def}) for each subtype of
saddle.
We verify that $\partialbar^1(s)$ has the form given in equation
(\ref{eqn:dbar-1-s-form}) and we read off the blocks $\partial_A(s)$,
$\beta_1(s)$, and $\alpha_2(s)$.
These calculations are straightforward but somewhat lengthy, so we
postpone these calculations to Appendix
\ref{appendix:induced-saddles}
and summarize the results in Table \ref{table:summary-saddles}.
The expressions for $\partial_A(s)$ in Table
\ref{table:summary-saddles} prove

\begin{lemma}
For each saddle $s$ in the cube of resolutions of $T$, the map
$\partialbar^1(s)$ has the form given in equation
(\ref{eqn:dbar-1-s-form}) with
$\partial_A(s) = T(s)$.
\end{lemma}

\section{Squares of induced saddles}
\label{sec:saddle-2}

Recall that in equation (\ref{eqn:dbar-2-box-def}) we defined a map
$\partialbar^2(\Box)$, which collects the terms of
$\partial_{\Tbar}^+$ corresponding to pairs of successive saddles that
are both induced by saddles in $\Box$.
In this section we prove:

\begin{lemma}
\label{lemma:d-red-box}
For each square $\Box$ in the cube of resolutions of $T$, the map
$\partialbar^2(\Box)$ has the form given in equation
(\ref{eqn:dbar-2-box-form}), and we have
\begin{align*}
  \partial_{red}(\Box) :=
  \partial_A(\Box) +
  \beta_1(\Box_R)\alpha_2(\Box_T) +
  \beta_1(\Box_B)\alpha_2(\Box_L) =  
  \partial_T^+(\Box).
\end{align*}
\end{lemma}

We consider separately the cases in which the square is interleaved or
nested and cases in which it is not interleaved or nested.

\subsection{Squares that are interleaved or nested}

The form of the map $\partialbar^2(\Box)$ depends on the relative
positions of the points $b$ and $c$ and the attaching points of the
saddles that make up the square, so we divide interleaved and nested
squares into subtypes according to the ordering of these points as
described in Tables \ref{table:I-types} and \ref{table:N-types}.
The squares marked $-$ are not possible,
since they would contain a saddle of type $W1_-$, which by Lemma
\ref{lemma:w1-impossible} is not possible.
In enumerating these subtypes we have also
applied Lemmas \ref{lemma:w2-impossible} and
\ref{lemma:c2-impossible}, which impose restrictions on the possible
saddles of type $W2_\pm$ and $C2_\pm^\beta$.

\begin{table}
\begin{align*}
  & I_+ & 
  & \Box_{00} &
  & &
  & I_- &
  & \Box_{11} \\
  & - &
  & c < b < sR < tR < sL < tL &
  & &
  & - &  
  & b < c < sL < tL < sR < tR & \\
  & - &
  & sR < c < b < tR < sL < tL &
  & &
  & - &
  & sL < b < c < tL < sR < tR & \\
  & I3_+: &
  & sR < tR < c < b < sL < tL &
  & &
  & I3_-: &
  & sL < tL < b < c < sR < tR & \\
  & I4_+: &
  & sR < tR < sL < c < b < tL &
  & &
  & I4_-: &
  & sL < tL < sR < b < c < tR & \\
  & I5_+: &
  & sR < tR < sL < tL < x_1 < x_2 &
  & &
  & I5_-: &
  & sL < tL < sR < tR < x_1 < x_2 & \\
  & I6_+: &
  & sR < tR < sL < tL,\,(b < c) &
  & &
  & I6_-: &
  & sL < tL < sR < tR,\,(b < c)
\end{align*}
\caption{
\label{table:I-types}
Subtypes of interleaved squares.
For $I_+$ we define subtypes according to the form of the planar
tangle $\Box_{00}$.
For $I_-$ we define subtypes according to the form of the planar
tangle $\Box_{11}$.
The pair $(x_1,x_2)$ is either $(b,c)$ or $(c,b)$.
}
\end{table}

\begin{table}
\begin{align*}
  & N_+^\beta & 
  & \Box_{00} &
  & &
  & N_-^\beta & 
  & \Box_{11} \\
  & - &
  & c < b < sR < t\beta < t\beta < sL &
  & &
  & - &
  & b < c < sL < t\betabar < t\betabar < sR & \\
  & N2_+^\beta: &
  & sR < c < b < t\beta < t\beta < sL &
  & &
  & N2_-^\beta: &
  & sL < b < c < t\betabar < t\betabar < sR & \\
  & N3_+^L: &
  & sR < tL < c < b < tL < sL &
  & &
  & N3_-^L: &
  & sL < tR < b < c < tR < sR & \\
  & N4_+^\beta: &
  & sR < t\beta < t\beta < c < b < sL &
  & &
  & N4_-^\beta: &
  & sL < t\betabar < t\betabar < b < c < sR & \\
  & N5_+^\beta: &
  & sR < t\beta < t\beta < sL < x_1 < x_2 &
  & &
  & N5_-^\beta: &
  & sL < t\betabar < t\betabar < sR < x_1 < x_2 & \\
  & N6_+^\beta: &
  & sR < t\beta < t\beta < sL,\, (b < c) &
  & &
  & N6_-^\beta: &
  & sL < t\betabar < t\betabar < sR,\, (b < c) &
\end{align*}
\caption{
\label{table:N-types}
Subtypes of nested squares.
For $N_+^\beta$ we define subtypes according to the form of the planar
tangle $\Box_{00}$.
For $N_-^\beta$ we define subtypes according to the form of the
planar tangle $\Box_{11}$.
The pair $(x_1,x_2)$ is either $(b,c)$ or $(c,b)$.
}
\end{table}

We prove Lemma \ref{lemma:d-red-box} by treating each subtype of
square in Tables \ref{table:I-types} and \ref{table:N-types} as a
separate case.
We prove three representative cases in detail:

\begin{lemma}
Lemma \ref{lemma:d-red-box} holds for squares of type $I3_+$.
\end{lemma}

\begin{proof}
A square $\Box$ of type $I3_+$ has the form
\begin{eqnarray*}
\begin{tikzcd}
  sR < tR < c < b < sL < tL
  \arrow{r}{W2_-}[swap]{s}
  \arrow{d}{W2_-}[swap]{t}
  &
  sL < b < c < tL < sR < tL
  \arrow{d}{C1_+^R}[swap]{t}
  \\
  sR < tL < sR < b < c < tR
  \arrow{r}{C3_+^L}[swap]{s} &
  sL < b < c < tR,\, (sR < tR).
\end{tikzcd}
\end{eqnarray*}
According to Table \ref{table:summary-saddles}, none of the saddles
induced by $\Box_T$, $\Box_B$, $\Box_R$, or $\Box_L$ are type $W_-$,
so $\partialbar^2(\Box) = 0$.
Thus $\partialbar^2(\Box)$ has the form given in equation
(\ref{eqn:dbar-2-box-form}) with
\begin{align}
  \label{eqn:I3-p-residual}
  \partial_A(\Box) = 0.
\end{align}
According to Table \ref{table:summary-saddles}, we have
\begin{align}
  \label{eqn:I3-p-reduction}
  &\beta_1(\Box_R)\alpha_2(\Box_T) =
  \Ttilde(\Box_R)Q(\Box_T), &
  &\beta_1(\Box_B)\alpha_2(\Box_L) = 0.
\end{align}
From equations (\ref{eqn:I3-p-residual}) and
(\ref{eqn:I3-p-reduction}), it follows that
\begin{align*}
  \partial_{red}(\Box) :=
  \partial_A(\Box) + \beta_1(\Box_R)\alpha_2(\Box_T) + 
  \beta_1(\Box_B)\alpha_2(\Box_L) =
  \Ttilde(\Box_R) Q(\Box_T).
\end{align*}
Since the saddle $\Box_R$ is type $C1_+^R$, we have
\begin{align*}
  \Ttilde(\Box_R) :=
  \Ttilde^L(\Box_R) + \Ttilde^R(\Box_R) = \Ttilde^R(\Box_R),
\end{align*}
and from the definitions of the maps
$\Ttilde^R(\Box_R)$, $Q(\Box_T)$, $\Ttilde^L(\Box_B)$, $Q(\Box_L)$, it
follows that
\begin{align*}
  \Ttilde^R(\Box_R) Q(\Box_T) =
  \Ttilde^L(\Box_B) Q(\Box_L).
\end{align*}
Thus
\begin{align*}
  \partial_{red}(\Box) =
  \Ttilde^L(\Box_B) Q(\Box_L) = \partial_T^+(\Box).
\end{align*}
\end{proof}

\begin{lemma}
Lemma \ref{lemma:d-red-box} holds for squares of type $I5_+$.
\end{lemma}

\begin{proof}
A square $\Box$ of type $I5_+$ has the form
\begin{eqnarray*}
\begin{tikzcd}
  sR < tR < sL < tL < x_1 < x_2
  \arrow{r}{W3_-}[swap]{s}
  \arrow{d}{W3_-}[swap]{t}
  &
  sL < tL < sR < tL < x_1 < x_2
  \arrow{d}{C3_+^R}[swap]{t}
  \\
  sR < tL < sR < tR < x_1 < x_2
  \arrow{r}{C3_+^L}[swap]{s} &
  sL < tR < x_1 < x_2,\, (sR < tR).
\end{tikzcd}
\end{eqnarray*}
According to Table \ref{table:summary-saddles},
the induced saddles $(\Box_T)_0$, $(\Box_T)_1$, $(\Box_L)_0$, and
$(\Box_L)_1$ are type $W_-$,
the induced saddles $(\Box_R)_0$ and $(\Box_R)_1$ are type $C_+^R$,
and
the induced saddles $(\Box_B)_0$ and $(\Box_B)_1$ are type $C_+^L$.
It follows that
\begin{align}
  \label{eqn:dbar-2-I5-p}
  \partialbar^2(\Box) =
  \Ttilde_{00}^L(\Box_B) Q_{00}(\Box_L) +
  \Ttilde_{11}^L(\Box_B) Q_{11}(\Box_L).
\end{align}
The maps
$Q_{00}(\Box_L)$,
$Q_{11}(\Box_L)$,
$\Ttilde_{00}^L(\Box_B)$, and
$\Ttilde_{11}^L(\Box_B)$,
are computed in Lemmas \ref{lemma:W3-m} and \ref{lemma:C3} in Appendix
\ref{appendix:induced-saddles}.
We substitute these maps into equation (\ref{eqn:dbar-2-I5-p}) to
obtain
\begin{align*}
  \partialbar^2(\Box) =
  \left(\begin{array}{ccc}
    \partial_A(\Box) & 0 & 0 \\
    0 & \partial_B(\Box) & 0 \\
    0 & 0 & \partial_B(\Box)
  \end{array}\right),
\end{align*}
where
\begin{align}
  \label{eqn:I5-p-reduction}
  \partial_A(\Box) = \partial_B(\Box) = \Ttilde^L(\Box_B) Q(\Box_L).
\end{align}
So $\partialbar^2(\Box)$ has the form given in equation
(\ref{eqn:dbar-2-box-form}).
According to Table \ref{table:summary-saddles}, we have
\begin{align}
  \label{eqn:I5-p-residual}
  &\beta_1(\Box_R)\alpha_2(\Box_T) = 0, &
  &\beta_1(\Box_B)\alpha_2(\Box_L) = 0.
\end{align}
From equations (\ref{eqn:I5-p-reduction}) and
(\ref{eqn:I5-p-residual}), it follows that
\begin{align*}
  \partial_{red}(\Box) :=
  \partial_A(\Box) + \beta_1(\Box_R)\alpha_2(\Box_T) + 
  \beta_1(\Box_B)\alpha_2(\Box_L) =
  \Ttilde^L(\Box_B) Q(\Box_L) = \partial_T^+(\Box).
\end{align*}
\end{proof}

\begin{lemma}
Lemma \ref{lemma:d-red-box} holds for squares of type $N5_+^\beta$.
\end{lemma}

\begin{proof}
A square $\Box$ of type $N5_+^\beta$ has the form
\begin{eqnarray*}
\begin{tikzcd}
  sR < t\beta < t\beta < sL < x_1 < x_2
  \arrow{r}{W3_-}[swap]{s}
  \arrow{d}{C3_+^\betabar}[swap]{t}
  &
  sL < t\betabar < t\betabar < sR < x_1 < x_2
  \arrow{d}{C3_+^\beta}[swap]{t}
  \\
  sR < t\betabar < sL < x_1 < x_2,\, (t\betabar)
  \arrow{r}{W3_-}[swap]{s}
  &
  sL < t\beta < sR < x_1 < x_2,\,(t\beta).
\end{tikzcd}
\end{eqnarray*}
According to Table \ref{table:summary-saddles},
the induced saddles $(\Box_T)_0$, $(\Box_T)_1$, $(\Box_B)_0$, and
$(\Box_B)_1$ are type $W_-$,
the induced saddles $(\Box_R)_0$ and $(\Box_R)_1$ are type
$C_+^\beta$, and
the induced saddles $(\Box_L)_0$ and $(\Box_L)_1$ are type
$C_+^\betabar$.
It follows that
\begin{align}
  \label{eqn:dbar-2-N5-p}
  \partialbar^2(\Box) =
  \Ttilde_{00}^L(\Box_R) Q_{00}(\Box_T) +
  \Ttilde_{11}^L(\Box_R) Q_{11}(\Box_T) +
  Q_{00}(\Box_B) \Ttilde_{00}^L(\Box_L) +
  Q_{11}(\Box_B) \Ttilde_{11}^L(\Box_L).
\end{align}
The maps
$Q_{00}(\Box_T)$,
$Q_{11}(\Box_T)$,
$Q_{00}(\Box_B)$,
$Q_{11}(\Box_B)$,
$\Ttilde_{00}^L(\Box_R)$,
$\Ttilde_{11}^L(\Box_R)$,
$\Ttilde_{00}^L(\Box_L)$, and
$\Ttilde_{11}^L(\Box_L)$
are computed in Lemmas \ref{lemma:W3-m} and \ref{lemma:C3} in Appendix
\ref{appendix:induced-saddles}.
We substitute these maps into equation (\ref{eqn:dbar-2-N5-p})
to obtain
\begin{align*}
  \partialbar^2(\Box) =
  \left(\begin{array}{ccc}
    \partial_A(\Box) & 0 & 0 \\
    0 & \partial_B(\Box) & 0 \\
    0 & 0 & \partial_B(\Box)
  \end{array}\right),
\end{align*}
where
\begin{align}
  \label{eqn:N5-p-reduction}
  \partial_A(\Box) = \partial_B(\Box) =
  \Ttilde^L(\Box_R) Q(\Box_T) + Q(\Box_B) \Ttilde^L(\Box_L).
\end{align}
So $\partialbar^2(\Box)$ has the form given in equation
(\ref{eqn:dbar-2-box-form}).
Note that if $\beta = L$ then the first term in equation
(\ref{eqn:N5-p-reduction}) vanishes and if $\beta = R$ then the second
term in equation (\ref{eqn:N5-p-reduction}) vanishes.
According to Table \ref{table:summary-saddles}, we have
\begin{align}
  \label{eqn:N5-p-residual}
  &\beta_1(\Box_R)\alpha_2(\Box_T) = 0, &
  &\beta_1(\Box_B)\alpha_2(\Box_L) = 0.
\end{align}
From equations (\ref{eqn:N5-p-reduction}) and
(\ref{eqn:N5-p-residual}), it follows that
\begin{align*}
  \partial_{red}(\Box) :=
  \partial_A(\Box) + \beta_1(\Box_R)\alpha_2(\Box_T) + 
  \beta_1(\Box_B)\alpha_2(\Box_L) =
  \Ttilde^L(\Box_R) Q(\Box_T) + Q(\Box_B) \Ttilde^L(\Box_L) =
  \partial_T^+(\Box).
\end{align*}
\end{proof}

The remaining cases are similar.
The results of the calculations for each case are
summarized in Table \ref{table:summary-squares}.

\begin{table}
\begin{eqnarray*}
  \begin{array}{llllllll}
    \Box &
    \Box_T & \Box_R &
    \Box_L & \Box_B &
    \partial_A(\Box) &
    \beta_1(\Box_R)\alpha_2(\Box_T) &
    \beta_1(\Box_B)\alpha_2(\Box_L) \\
    \\
    I3_+ &
    W2_- & C1_+^R &
    W2_- & C3_+^L &
    0 & \Ttilde(\Box_R)Q(\Box_T) & 0 \\
    I4_+ &
    W3_- & C2_+^R &
    W2_- & C2_+^L &
    0 & 0 & \Ttilde^L(\Box_B) Q(\Box_L) \\
    I5_+ &
    W3_- & C3_+^R &
    W3_- & C3_+^L &
    \Ttilde^L(\Box_B) Q(\Box_L) & 0 & 0 \\
    I6_+ &
    W4_- & C4_+^R &
    W4_- & C4_+^L &
    \Ttilde^L(\Box_B) Q(\Box_L) & 0 & 0 \\
    \\
    I3_- &
    C3_-^R & W2_- &
    C1_-^L & W2_- &
    0 & 0 & Q(\Box_B) \Ttilde(\Box_L) \\
    I4_- &
    C2_-^R & W2_- &
    C2_-^L & W3_- &
    0 & Q(\Box_R) \Ttilde^R(\Box_T) & 0 \\
    I5_- &
    C3_-^R & W3_- &
    C3_-^L & W3_- &
    Q(\Box_B) \Ttilde^L(\Box_L) & 0 & 0 \\
    I6_- &
    C4_-^R & W4_- &
    C4_-^L & W4_- &
    Q(\Box_B) \Ttilde^L(\Box_L) & 0 & 0 \\
    \\
    N2_+^\beta &
    W2_- & C3_+^\beta &
    C1_+^\betabar & W2_- &
    0 & 0 & Q(\Box_B)\Ttilde(\Box_L) \\
    N3_+^L &
    W2_- & C2_+^L &
    C2_+^R & W4_- &
    0 & \Ttilde^L(\Box_R)Q(\Box_T) & 0 \\
    N4_+^\beta &
    W2_- & C1_+^\beta &
    C3_+^\betabar & W2_- &
    0 & \Ttilde(\Box_R)Q(\Box_T) & 0 \\
    N5_+^\beta &
    W3_- & C3_+^\beta &
    C3_+^\betabar & W3_- &
    \Ttilde^L(\Box_R) Q(\Box_T) +
    Q(\Box_B) \Ttilde^L(\Box_L) & 0 & 0 \\
    N6_+^\beta &
    W4_- & C4_+^\beta &
    C4_-^\betabar & W4_- &
    \Ttilde^L(\Box_R) Q(\Box_T) +
    Q(\Box_B) \Ttilde^L(\Box_L) & 0 & 0 \\
    \\
    N2_-^\beta &
    W2_- & C1_-^\beta &
    C3_-^\betabar & W2_- &
    0 & \Ttilde(\Box_R) Q(\Box_T) & 0 \\
    N3_-^L &
    W4_- & C2_-^L &
    C2_-^R & W2_- &
    0 & 0 & Q(\Box_B) \Ttilde^R(\Box_L) \\
    N4_-^\beta &
    W2_- & C3_-^\beta &
    C1_-^\betabar & W2_- &
    0 & 0 & Q(\Box_B) \Ttilde(\Box_L) \\
    N5_-^\beta &
    W3_- & C3_-^\beta &
    C3_-^\betabar & W3_- &
    \Ttilde^L(\Box_R) Q(\Box_T) +
    Q(\Box_B) \Ttilde^L(\Box_L) & 0 & 0 \\
    N6_-^\beta &
    W4_- & C4_-^\beta &
    C4_-^\betabar & W4_- &
    \Ttilde^L(\Box_R) Q(\Box_T) +
    Q(\Box_B) \Ttilde^L(\Box_L)  & 0 & 0 \\
  \end{array}
\end{eqnarray*}
\caption{
\label{table:summary-squares}
Interleaved and nested squares in the cube of resolutions of $T$.
For each square $\Box$, we list the terms $\Box$ contributes to
$\partial_{red}(\Box) =
\partial_A(\Box) + \beta_1(\Box_R)\alpha_2(\Box_T) +
\beta_1(\Box_B)\alpha_2(\Box_L)$.
}
\end{table}

\subsection{Squares that are not interleaved or nested}

Recall from Lemma \ref{lemma:zero-unless-I-N} that
$\partial_T^+(\Box) = 0$ if $\Box$ is not interleaved or nested, so
for such squares Lemma \ref{lemma:d-red-box} follows from the
following two lemmas:

\begin{lemma}
For each square $\Box$ in the cube of resolutions of $T$ that is not
interleaved or nested, we have
\begin{align*}
  \beta_1(\Box_R) \alpha_2(\Box_T) +
  \beta_1(\Box_B)\alpha_2(\Box_L) = 0.
\end{align*}
\end{lemma}

\begin{proof}
Using the results summarized in Table \ref{table:summary-saddles}, we
enumerate the types of pairs of successive saddles $(s,t)$ in the cube
of resolutions of $T$ that could yield a nonzero term
$\beta_1(t)\alpha_2(s)$, and we display the results in Table
\ref{table:reduction-terms}.
For each such pair, we enumerate the types of squares to which the
pair could belong.

For example, consider a pair of successive saddles $(s,t)$, where $s$
is type $W2_-$ and $t$ is type $C1_+^R$:
\begin{eqnarray*}
\begin{tikzcd}
  p_1(\ptc) \arrow{r}{W2_-}[swap]{s} &[1 em]
  p_2(\ptb) \arrow{r}{C1_+^R}[swap]{t} &[1 em]
  p_3(\ptb).
\end{tikzcd}
\end{eqnarray*}
According to Table \ref{table:summary-saddles}, the pair $(s,t)$
yields the term
\begin{align*}
  \beta_1(t)\alpha_2(s) = \Ttilde(t)Q(s).
\end{align*}
Given that $s$ is type $W2_-$ and $t$ is type $C1_+^R$, there are
three possibilities for the orderings of the attaching points
of the saddles $s$ and $t$ in the planar tangle $p_2(\ptb)$, which we
enumerate in the second column of the following table:
\begin{align*}
  & p_1(\ptc) = \Box_{00}&
  & p_2(\ptb) &
  & \Box \\
  & sR < tR < tR < c < b < sL &
  & sL < b < c < tL < tL < sR &
  & N4_+^R \\
  & sR < tR < c < b < sL < tL &
  & sL < b < c < tL < sR < tL &
  & I3_+ \\
  & sR < c < b < sL < tL < tL &
  & sL < b < c < sR < tL < tL &
  & D
\end{align*}
We describe the corresponding planar tangle $p_1(\ptc) = \Box_{00}$
for each ordering in the first column of the table, and use this
description to classify the corresponding square, as indicated in the
third column.

We repeat this calculation for each of the remaining pairs of saddles
$(s,t)$ that yield a nonzero term $\beta_1(t)\alpha_2(s)$ and display
the results in Table \ref{table:reduction-terms}.
In each case, we find that $(s,t)$ must belong to a square that is
interleaved, nested, or disjoint.
If the square $\Box$ is disjoint, then
\begin{align*}
  \beta_1(\Box_R) \alpha_2(\Box_T) +
  \beta_1(\Box_B) \alpha_2(\Box_L) = 0,
\end{align*}
since the two terms cancel.
\end{proof}

\begin{table}
\begin{eqnarray*}
\begin{array}{lll}
\begin{tikzcd}
  p_1 \arrow{r}[swap]{s} &
  p_2 \arrow{r}[swap]{t} &
  p_3
\end{tikzcd}
&
\beta_1(t) \alpha_2(s) &
\Box \\
\begin{tikzcd}
  p_1(\ptc) \arrow{r}{W2_-}[swap]{s} &
  p_2(\ptb) \arrow{r}{C1_+^L}[swap]{t} &
  p_3(\ptb)
\end{tikzcd}
& \Ttilde(t)Q(s) &
N4_+^L,\, D \\
\begin{tikzcd}
  p_1(\ptc) \arrow{r}{W2_-}[swap]{s} &
  p_2(\ptb) \arrow{r}{C1_+^R}[swap]{t} &
  p_3(\ptb)
\end{tikzcd}
& \Ttilde(t)Q(s) &
I3_+,\, N4_+^R,\, D \\
\begin{tikzcd}
  p_1(\ptc) \arrow{r}{W2_-}[swap]{s} &
  p_2(\ptb) \arrow{r}{C1_-^\beta}[swap]{t} &
  p_3(\ptb)
\end{tikzcd}
&\Ttilde(t)Q(s) &
N2_-^\beta,\, D \\
\begin{tikzcd}
  p_1(\ptc) \arrow{r}{C1_+^\beta}[swap]{s} &
  p_2(\ptc) \arrow{r}{W2_-}[swap]{t} &
  p_3(\ptb)
\end{tikzcd}
& Q(t)\Ttilde(s) &
N2_+^\betabar,\, D \\
\begin{tikzcd}
  p_1(\ptc) \arrow{r}{C1_-^L}[swap]{s} &
  p_2(\ptc) \arrow{r}{W2_-}[swap]{t} &
  p_3(\ptb)
\end{tikzcd}
& Q(t)\Ttilde(s) &
I3_-,\, N4_-^R,\, D \\
\begin{tikzcd}
  p_1(\ptc) \arrow{r}{C1_-^R}[swap]{s} &
  p_2(\ptc) \arrow{r}{W2_-}[swap]{t} &
  p_3(\ptb)
\end{tikzcd}
& Q(t)\Ttilde(s) &
N4_-^L,\, D \\
\begin{tikzcd}
  p_1(\ptc) \arrow{r}{W2_-}[swap]{s} &
  p_2(\ptb) \arrow{r}{C2_+^L}[swap]{t} &
  p_3(\pto)
\end{tikzcd}
& \Ttilde^L(t) Q(s) &
I4_+,\, N3_+^L \\
\begin{tikzcd}
  p_1(\pto) \arrow{r}{C2_-^R}[swap]{s} &
  p_2(\ptc) \arrow{r}{W2_-}[swap]{t} &
  p_3(\ptb)
\end{tikzcd}
& Q(t) \Ttilde^R(s) &
I4_-,\, N3_-^L \\
\end{array}
\end{eqnarray*}
\caption{
\label{table:reduction-terms}
Pairs of successive of saddles $(s,t)$ in the cube of resolutions of
$T$ that yield a nonzero term $\beta_1(t)\alpha_2(s)$.
For each such pair $(s,t)$, we compute
$\beta_1(t)\alpha_2(s)$ and enumerate the squares to which the pair
$(s,t)$ could belong.
}
\end{table}

\begin{lemma}
For each square $\Box$ in the cube of resolutions of $T$ that is not
interleaved or nested, the map $\partialbar^2(\Box)$ has the form
given in equation (\ref{eqn:dbar-2-box-form}) with
$\partial_A(\Box) = 0$.
\end{lemma}

\begin{proof}
Recall from equations (\ref{eqn:dbar-2-t-s}) and
(\ref{eqn:dbar-2-box-def}) that given a square
$\Box$ in the cube of resolutions of $T$, the corresponding map
$\partialbar^2(\Box)$ is given by
\begin{align*}
  \partialbar^2(\Box) :=
  \partialbar^2(\Box_R,\Box_T) +
  \partialbar^2(\Box_B,\Box_L),
\end{align*}
where
\begin{align*}
  \partialbar^2(s,t) :=
  \Ttilde_{00}^L(t) Q_{00}(s) +
  Q_{00}(t)\Ttilde_{00}^L(s) +
  \Ttilde_{11}^L(t) Q_{11}(s) +
  Q_{11}(t)\Ttilde_{11}^L(s)
\end{align*}
collects the terms in $\partial_\Tbar^+$ corresponding to the pairs of
induced saddles $([s]_0,[t]_0)$ and $([s]_1,[t]_1)$.
We determine which pairs of saddles $(s,t)$ in the cube of resolutions
of $T$ could yield a nonzero map $\partialbar^2(s,t)$, and for each
such pair we enumerate the squares to which $(s,t)$ could belong.
We show that for each of these squares the map $\partialbar^2(\Box)$
has the form given in equation (\ref{eqn:dbar-2-box-form}), and if
$\partial_A(\Box)$ is nonzero then the square must be interleaved or
nested.

For $\partialbar^2(s,t)$ to be nonzero, at least one of the pairs
$([s]_0,[t]_0)$ or $([s]_1,[t]_1)$ must contain one saddle of type
$W_-$ and one saddle of type $C_\pm^L$.
According to Table \ref{table:summary-saddles}, if a saddle $u$ in the
cube of resolutions of $T$ induces a saddle $[u]_0$ or $[u]_1$ of type
$W_-$, then $u$ must be type $W3_-$, $W4_-$, $C2_+^L$, or $C2_-^R$.
We consider these cases separately in Tables
\ref{table:residual-W3-1A} --
\ref{table:residual-C2-m-R}.
For each case, we enumerate the saddles $v$ in the cube of
resolutions of $T$ that induce a saddle $[v]_0$ or $[v]_1$ of type
$C_\pm^L$ and could pair with the saddle $[u]_0$ or $[u]_1$ of
type $W_-$ to yield a nonzero map $\partialbar^2(s,t)$, where
$(s,t)$ is either $(u,v)$ or $(v,u)$.

For example, consider a saddle $s:p_1(\ptb) \rightarrow p_2(\ptb)$ of type
$W3_-$.
According to Table \ref{table:summary-saddles}, the induced saddles
$[s]_0$ and $[s]_1$ are type $W_-$.
We can obtain a nonzero term $\partialbar^2(s,t)$ by pairing $s$ with
a saddle $t:p_2(\ptb) \rightarrow p_3(\ptb)$ of type $C3_+^L$, since
according to Table \ref{table:summary-saddles} the induced saddles
$[t]_0$ and $[t]_1$ are type $C_+^L$:
\begin{align}
  \label{eqn:dbar-2-example}
  \partialbar^2(s,t) =
  \Ttilde_{00}^L(t) Q_{00}(s) +
  \Ttilde_{11}^L(t) Q_{11}(s).
\end{align}
The maps
$Q_{00}(s)$, $Q_{11}(s)$, $\Ttilde_{00}^L(t)$, and $\Ttilde_{11}^L(t)$
are computed in Lemmas \ref{lemma:W3-m} and \ref{lemma:C3} in Appendix
\ref{appendix:induced-saddles}.
We substitute these maps into equation (\ref{eqn:dbar-2-example}) to
obtain
\begin{align*}
  \partialbar^2(s,t) =
  \left(\begin{array}{ccc}
    \partial_A(s,t) & 0 & 0 \\
    0 & \partial_{B^1}(s,t) & 0 \\
    0 & 0 & \partial_{B^2}(s,t)
  \end{array}
  \right),
\end{align*}
where
\begin{align*}
  \partial_A(s,t) = \partial_{B^1}(s,t) = \partial_{B^2}(s,t) =
  \Ttilde^L(t) Q(s).
\end{align*}
Given that $s$ is type $W3_-$ and $t$ is type $C3_+^L$, there are six
possibilities for the orderings of the attaching points of the saddles
$s$ and $t$ in the planar tangle $p_2(\ptb)$, which we enumerate in
the second column of the following table:
\begin{align*}
  & p_1(\ptb) = \Box_{00} &
  & p_2(\ptb) &
  & \Box \\
  & sR < sL < tR < tR < b < c &
  & sL < sR < tR < tR < b < c &
  & D \\
  & sR < tL < sL < tR < b < c &
  & sL < tR < sR < tR < b < c &
  & \textup{not possible by Lemma \ref{lemma:impossible-pairs}} \\
  & tR < sR < sL < tR < b < c &
  & tR < sL < sR < tR < b < c &
  & \textup{not possible by Lemma \ref{lemma:impossible-nest}} \\
  & sR < tL < tL < sL < b < c &
  & sL < tR < tR < sR < b < c &
  & \textup{$N5_+^L$} \\
  & tR < sR < tL < sL < b < c &
  & tR < sL < tR < sR < b < c &
  & \textup{$I5_+$} \\
  & tR < tR < sR < sL < b < c &
  & tR < tR < sL < sR < b < c &
  & D
\end{align*}
We describe the corresponding planar tangle $p_1(\ptb) = \Box_{00}$ for
each ordering in the first column of the table, and use this
description to classify the corresponding square, as indicated in the
third column.

If a pair of saddles $(s,t)$ belongs to a disjoint square $\Box$,
then
\begin{align*}
  \partialbar^2(\Box) :=
  \partialbar^2(\Box_R,\Box_T) +
  \partialbar^2(\Box_B,\Box_L) = 0,
\end{align*}
since the two terms cancel, so $\partialbar^2(\Box)$ is of the form
given in equation (\ref{eqn:dbar-2-box-form}) with
$\partial_A(\Box) = 0$.

We note that in certain cases, such as the pairs of saddles $(s,t)$ in
Tables \ref{table:residual-C2-p-L} and \ref{table:residual-C2-m-R},
the only nonvanishing blocks of $\partialbar^2(s,t)$ are
$\alpha_1(s,t)$ or $\beta_2(s,t)$.
In such cases we do not enumerate the squares to which $(s,t)$ could
belong, since we already know that for each such square the
contribution of  $\partialbar^2(s,t)$ to $\partialbar^2(\Box)$ is of
the form given in equation (\ref{eqn:dbar-2-box-form}) with
$\partial_A(s,t) = 0$.
\end{proof}

\begin{table}
\begin{eqnarray*}
\begin{array}{lll}
\begin{tikzcd}
  p_1 \arrow{r}[swap]{s} &
  p_2 \arrow{r}[swap]{t} &
  p_3
\end{tikzcd}
& \textup{nonzero blocks of $\partialbar^2(s,t)$}
& \Box \\
\begin{tikzcd}
  p_1(\ptb) \arrow{r}{W3_-}[swap]{s} &
  p_2(\ptb) \arrow{r}{W2_+}[swap]{t} &
  p_3(\ptc)
\end{tikzcd}
&\alpha_1(s,t)
& - \\
\begin{tikzcd}
  p_1(\ptc) \arrow{r}{W2_-}[swap]{s} &
  p_2(\ptb) \arrow{r}{W3_-}[swap]{t} &
  p_3(\ptb)
\end{tikzcd}
&\partial_{B^1}(s,t) = Q(t)Q(s)
& D \\
\begin{tikzcd}
  p_1(\ptb) \arrow{r}{W3_-}[swap]{s} &
  p_2(\ptb) \arrow{r}{C1_\pm^R}[swap]{t} &
  p_3(\ptb)
\end{tikzcd}
&\partial_{B^2}(s,t) = \Ttilde^R(t) Q(s)
& D \\
\begin{tikzcd}
  p_1(\ptb) \arrow{r}{C1_\pm^R}[swap]{s} &
  p_2(\ptb) \arrow{r}{W3_-}[swap]{t} &
  p_3(\ptb)
\end{tikzcd}
&\partial_{B^2}(s,t) = Q(t) \Ttilde^R(s)
& D \\
\begin{tikzcd}
  p_1(\ptb) \arrow{r}{W3_-}[swap]{s} &
  p_2(\ptb) \arrow{r}{C3_+^L}[swap]{t} &
  p_3(\ptb)
\end{tikzcd}
&\partial_A(s,t) = \partial_{B^1}(s,t) = \partial_{B^2}(s,t) =
\Ttilde^L(t) Q(s)
& I5_+,\, N5_+^L,\, D \\
\begin{tikzcd}
  p_1(\ptb) \arrow{r}{W3_-}[swap]{s} &
  p_2(\ptb) \arrow{r}{C3_-^L}[swap]{t} &
  p_3(\ptb)
\end{tikzcd}
&\partial_A(s,t) = \partial_{B^1}(s,t) = \partial_{B^2}(s,t) =
\Ttilde^L(t) Q(s)
& N5_-^L,\, D \\
\begin{tikzcd}
  p_1(\ptb) \arrow{r}{C3_+^L}[swap]{s} &
  p_2(\ptb) \arrow{r}{W3_-}[swap]{t} &
  p_3(\ptb)
\end{tikzcd}
&\partial_A(s,t) = \partial_{B^1}(s,t) = \partial_{B^2}(s,t) =
Q(t)\Ttilde^L(s)
& N5_+^R,\, D \\
\begin{tikzcd}
  p_1(\ptb) \arrow{r}{C3_-^L}[swap]{s} &
  p_2(\ptb) \arrow{r}{W3_-}[swap]{t} &
  p_3(\ptb)
\end{tikzcd}
&\partial_A(s,t) = \partial_{B^1}(s,t) = \partial_{B^2}(s,t) =
Q(t)\Ttilde^L(s)
& I5_-,\, N5_-^R,\, D \\
\end{array}
\end{eqnarray*}
\caption{
\label{table:residual-W3-1A}
Nonzero maps $\partialbar^2(s,t)$ for which the induced saddle of
type $W_-$ is induced by a saddle $p(\ptb) \rightarrow q(\ptb)$ of type
$W3_-$.
For each pair of saddles $(s,t)$ we describe the nonzero blocks of
$\partialbar^2(s,t)$ and enumerate the squares to which the pair
$(s,t)$ could belong.
}
\end{table}

\begin{table}
\begin{eqnarray*}
\begin{array}{lll}
\begin{tikzcd}
  p_1 \arrow{r}[swap]{s} &
  p_2 \arrow{r}[swap]{t} &
  p_3
\end{tikzcd}
& \textup{nonzero blocks of $\partialbar^2(s,t)$}
& \Box \\
\begin{tikzcd}
  p_1(\ptc) \arrow{r}{W3_-}[swap]{s} &
  p_2(\ptc) \arrow{r}{W2_-}[swap]{t} &
  p_3(\ptb)
\end{tikzcd}
& \partial_{B^2}(s,t) = Q(t)Q(s)
& D \\
\begin{tikzcd}
  p_1(\ptb) \arrow{r}{W2_+}[swap]{s} &
  p_2(\ptc) \arrow{r}{W3_-}[swap]{t} &
  p_3(\ptc)
\end{tikzcd}
&\alpha_1(s,t)
& - \\
\begin{tikzcd}
  p_1(\ptc) \arrow{r}{W3_-}[swap]{s} &
  p_2(\ptc) \arrow{r}{C1_\pm^R}[swap]{t} &
  p_3(\ptc)
\end{tikzcd}
&\partial_{B^1}(s,t) = \Ttilde^R(t)Q(s)
& D \\
\begin{tikzcd}
  p_1(\ptc) \arrow{r}{C1_\pm^R}[swap]{s} &
  p_2(\ptc) \arrow{r}{W3_-}[swap]{t} &
  p_3(\ptc)
\end{tikzcd}
&\partial_{B^1}(s,t) = Q(t)\Ttilde^R(s)
& D \\
\begin{tikzcd}
  p_1(\ptc) \arrow{r}{W3_-}[swap]{s} &
  p_2(\ptc) \arrow{r}{C2_+^R}[swap]{t} &
  p_3(\pto)
\end{tikzcd}
& \beta_2(s,t)
& - \\
\begin{tikzcd}
  p_1(\pto) \arrow{r}{C2_-^R}[swap]{s} &
  p_2(\ptc) \arrow{r}{W3_-}[swap]{t} &
  p_3(\ptc)
\end{tikzcd}
& \partial_A(s,t) = Q(t) \Ttilde^R(s)
& D \\
\begin{tikzcd}
  p_1(\ptc) \arrow{r}{W3_-}[swap]{s} &
  p_2(\ptc) \arrow{r}{C3_+^L}[swap]{t} &
  p_3(\ptc)
\end{tikzcd}
&\partial_A(s,t) = \partial_{B^1}(s,t) = \partial_{B^2}(s,t) =
\Ttilde^L(t) Q(s)
& I5_+,\, N5_+^L,\, D \\
\begin{tikzcd}
  p_1(\ptc) \arrow{r}{W3_-}[swap]{s} &
  p_2(\ptc) \arrow{r}{C3_-^L}[swap]{t} &
  p_3(\ptc)
\end{tikzcd}
&\partial_A(s,t) = \partial_{B^1}(s,t) = \partial_{B^2}(s,t) =
\Ttilde^L(t) Q(s)
& N5_-^L,\, D \\
\begin{tikzcd}
  p_1(\ptc) \arrow{r}{C3_+^L}[swap]{s} &
  p_2(\ptc) \arrow{r}{W3_-}[swap]{t} &
  p_3(\ptc)
\end{tikzcd}
&\partial_A(s,t) = \partial_{B^1}(s,t) = \partial_{B^2}(s,t) =
Q(t)\Ttilde^L(s)
& N5_+^R,\, D \\
\begin{tikzcd}
  p_1(\ptc) \arrow{r}{C3_-^L}[swap]{s} &
  p_2(\ptc) \arrow{r}{W3_-}[swap]{t} &
  p_1(\ptc)
\end{tikzcd}
&\partial_A(s,t) = \partial_{B^1}(s,t) = \partial_{B^2}(s,t) =
Q(t)\Ttilde^L(s)
& I5_-,\, N5_-^R,\, D \\
\end{array}
\end{eqnarray*}
\caption{
\label{table:residual-W3-1B}
Nonzero maps $\partialbar^2(s,t)$ for which the induced saddle of
type $W_-$ is induced by a saddle $p(\ptc) \rightarrow q(\ptc)$ of type
$W3_-$.
For each pair of saddles $(s,t)$ we list the nonzero blocks of
$\partialbar^2(s,t)$ and enumerate the squares to which the pair
$(s,t)$ could belong.
}
\end{table}

\begin{table}
\begin{eqnarray*}
\begin{array}{lll}
\begin{tikzcd}
  p_1 \arrow{r}[swap]{s} &
  p_2 \arrow{r}[swap]{t} &
  p_3
\end{tikzcd}
& \textup{nonzero blocks of $\partialbar^2(s,t)$}
& \Box \\
\begin{tikzcd}
  p_1(\pto) \arrow{r}{W4_-}[swap]{s} &
  p_2(\pto) \arrow{r}{C2_-^R}[swap]{t} &
  p_3(\ptc)
\end{tikzcd}
& \partial_A(s,t) = \Ttilde^R(s)Q(t)
& D \\
\begin{tikzcd}
  p_1(\ptc) \arrow{r}{C2_+^R}[swap]{s} &
  p_2(\pto) \arrow{r}{W4_-}[swap]{t} &
  p_3(\pto)
\end{tikzcd}
& \beta_2(s,t)
& - \\
\begin{tikzcd}
  p_1(\pto) \arrow{r}{W4_-}[swap]{s} &
  p_2(\pto) \arrow{r}{C4_+^L}[swap]{t} &
  p_3(\pto)
\end{tikzcd}
& \partial_A(s,t) = \Ttilde^L(t) Q(s)
& I6_+,\, N6_+^L,\, D \\
\begin{tikzcd}
  p_1(\pto) \arrow{r}{W4_-}[swap]{s} &
  p_2(\pto) \arrow{r}{C4_-^L}[swap]{t} &
  p_3(\pto)
\end{tikzcd}
& \partial_A(s,t) = \Ttilde^L(t) Q(s)
& N6_-^L,\, D \\
\begin{tikzcd}
  p_1(\pto) \arrow{r}{C4_+^L}[swap]{s} &
  p_2(\pto) \arrow{r}{W4_-}[swap]{t} &
  p_3(\pto)
\end{tikzcd}
& \partial_A(s,t) = Q(t) \Ttilde^L(s)
& N6_+^R,\, D \\
\begin{tikzcd}
  p_1(\pto) \arrow{r}{C4_-^L}[swap]{s} &
  p_2(\pto) \arrow{r}{W4_-}[swap]{t} &
  p_3(\pto)
\end{tikzcd}
& \partial_A = Q(t) \Ttilde^L(s)
& I6_-,\, N6_-^R,\, D \\
\begin{tikzcd}
  p_1(\pto) \arrow{r}{W4_-}[swap]{s} &
  p_2(\pto) \arrow{r}{C1_\pm^C}[swap]{t} &
  p_3(\pto)
\end{tikzcd}
& \partial_A(s,t) =
\Ttilde^{L_0}(t)Q(s)\id_{C_{[p_1]_0}} +
\Ttilde^{L_1}(t)Q(s)\id_{C_{[p_1]_1}}
& D \\
\begin{tikzcd}
  p_1(\pto) \arrow{r}{C1_\pm^C}[swap]{s} &
  p_2(\pto) \arrow{r}{W4_-}[swap]{t} &
  p_3(\pto)
\end{tikzcd}
& \partial_A(s,t) =
Q(t)\Ttilde^{L_0}(s)\id_{C_{[p_1]_0}} +
Q(t)\Ttilde^{L_1}(s)\id_{C_{[p_1]_1}}
& D
\end{array}
\end{eqnarray*}
\caption{
\label{table:residual-W4}
Nonzero maps $\partialbar^2(s,t)$ for which the induced saddle of
type $W_-$ is induced by a saddle $p(\pto) \rightarrow q(\pto)$ of
type $W4_-$.
For each pair of saddles $(s,t)$ we list the nonzero blocks of
$\partialbar^2(s,t)$ and enumerate the squares to which the pair
$(s,t)$ could belong.
}
\end{table}

\begin{table}
\begin{eqnarray*}
\begin{array}{ll}
\begin{tikzcd}
  p_1 \arrow{r}[swap]{s} &
  p_2 \arrow{r}[swap]{t} &
  p_3
\end{tikzcd}
&
\textup{nonzero blocks of $\partialbar^2(s,t)$} \\
\begin{tikzcd}
  p_1(\ptb) \arrow{r}{C1_\pm^R}[swap]{s} &
  p_2(\ptb) \arrow{r}{C2_+^L}[swap]{t} &
  p_3(\pto)
\end{tikzcd}
& \beta_2(s,t) \\
\begin{tikzcd}
  p_1(\ptb) \arrow{r}{C2_+^L}[swap]{s} &
  p_2(\pto) \arrow{r}{C2_-^R}[swap]{t} &
  p_3(\ptc)
\end{tikzcd}
& \beta_2(s,t) \\
\begin{tikzcd}
  p_1(\ptb) \arrow{r}{C3_\pm^L}[swap]{s} &
  p_2(\ptb) \arrow{r}{C2_+^L}[swap]{t} &
  p_3(\pto)
\end{tikzcd}
& \beta_2(s,t) \\
\begin{tikzcd}
  p_1(\ptb) \arrow{r}{C2_+^L}[swap]{s} &
  p_2(\pto) \arrow{r}{C4_\pm^L}[swap]{t} &
  p_3(\pto)
\end{tikzcd}
& \beta_2(s,t) \\
\begin{tikzcd}
  p_1(\ptb) \arrow{r}{C2_+^L}[swap]{s} &
  p_2(\pto) \arrow{r}{C1_\pm^C}[swap]{t} &
  p_3(\pto)
\end{tikzcd}
& \beta_2(s,t)
\end{array}
\end{eqnarray*}
\caption{
\label{table:residual-C2-p-L}
Nonzero maps $\partialbar^2(s,t)$ for which the induced saddle of
type $W_-$ is induced by a saddle $p(\ptb) \rightarrow q(\pto)$ of type
$C2_+^L$.
For each pair of saddles $(s,t)$ we list the nonzero blocks of
$\partialbar^2(s,t)$.
}
\end{table}

\begin{table}
\begin{eqnarray*}
\begin{array}{ll}
\begin{tikzcd}
  p_1 \arrow{r}[swap]{s} &
  p_2 \arrow{r}[swap]{t} &
  p_3
\end{tikzcd}
&
\textup{nonzero blocks of $\partialbar^2(s,t)$} \\
\begin{tikzcd}
  p_1(\pto) \arrow{r}{C2_-^R}[swap]{s} &
  p_2(\ptc) \arrow{r}{W2_-}[swap]{t} &
  p_3(\ptb)
\end{tikzcd}
& \alpha_1(s,t) \\
\begin{tikzcd}
  p_1(\pto) \arrow{r}{C2_-^R}[swap]{s} &
  p_2(\ptc) \arrow{r}{C1_\pm^R}[swap]{t} &
  p_3(\ptc)
\end{tikzcd}
& \alpha_1(s,t) \\
\begin{tikzcd}
  p_1(\pto) \arrow{r}{C2_-^R}[swap]{s} &
  p_2(\ptc) \arrow{r}{C3_\pm^L}[swap]{t} &
  p_3(\ptc)
\end{tikzcd}
& \alpha_1(s,t) \\
\begin{tikzcd}
  p_1(\pto) \arrow{r}{C4_\pm^L}[swap]{s} &
  p_2(\pto) \arrow{r}{C2_-^R}[swap]{t} &
  p_3(\ptc)
\end{tikzcd}
& \alpha_1(s,t) \\
\begin{tikzcd}
  p_1(\pto) \arrow{r}{C1_\pm^C}[swap]{s} &
  p_2(\pto) \arrow{r}{C2_-^R}[swap]{t} &
  p_3(\ptc)
\end{tikzcd}
& \alpha_1(s,t)
\end{array}
\end{eqnarray*}
\caption{
\label{table:residual-C2-m-R}
Nonzero maps $\partialbar^2(s,t)$ for which the induced saddle of
type $W_-$ is induced by a saddle $p(\pto) \rightarrow q(\ptc)$ of type
$C2_-^R$.
For each pair of saddles $(s,t)$ we list the nonzero blocks of
$\partialbar^2(s,t)$.
}
\end{table}

\begin{appendix}

\section{Bigradings}
\label{appendix:bigradings}

Here we prove:

\begin{lemma}
For each planar tangle $p$ in the cube of resolutions of a tangle
diagram $T$, we have $A_p = C_p^+$ as bigraded vector spaces.
\end{lemma}

\begin{proof}
By performing an isotopy of the annulus and an isotopy of the overpass
arc, we can reduce to the case where the tangle diagram $T$ is in
standard position and the overpass arc is the arc $A_s^+$ shown in
Figure \ref{fig:T-and-Tbar}.
Note that the bigrading shift in $C_p^+$ given in equations
(\ref{eqn:h}) and (\ref{eqn:q}) is unchanged under these isotopies.
Consider the oriented link diagram $T^+ = T \cup A_s^+$.
Recall that $a_+(A_s^+,T)$ and $a_-(A_s^+,T)$ denote the number of
positive and negative crossings between $A_s^+$ and $T$.
The loop number $\ell(T)$ is given by
\begin{align}
  \label{eqn:ell}
  \ell(T) = a_+(A_s^+,T) + a_-(A_s^+,T).
\end{align}
Let $n_+(T^+)$ and $n_-(T^+)$ denote the total number of positive and
negative crossings of the link diagram $T^+$:
\begin{align}
  \label{eqn:n-pm}
  n_\pm(T^+) = m_\pm(T) + a_\pm(A_s^+,T).
\end{align}
Using equations (\ref{eqn:ell}) and (\ref{eqn:n-pm}), we can express
equations (\ref{eqn:h}) and (\ref{eqn:q}) for the bigrading shift in
$C_p^+$ as
\begin{align}
  \label{eqn:h-loop}
  h^+(T,p) &=
  -n_-(T^+) + \frac{1}{2}(\ell(T) + w(p)) + r(p), \\
  \label{eqn:q-loop}
  q^+(T,p) &=
  n_+(T^+) - 2n_-(T^+) + \frac{1}{2}(\ell(T) + 3w(p)) + r(p).
\end{align}
The loop number of $\Tbar$ and the number of positive and negative
crossings of $\Tbar^+$ are given by
\begin{align}
  \label{eqn:ell-npm}
  &\ell(\Tbar) = \ell(T) - 1, &
  &n_\pm(\Tbar^+) = n_\pm(T^+).
\end{align}
The resolution degrees of the planar tangles $[p]_0$ and $[p]_1$
induced by $p$ are
\begin{align}
  \label{eqn:r-p0-p1}
  r([p]_0) &= r(p), &
  r([p]_1) &= r(p) + 1.
\end{align}

Consider the case that $p$ is type $\ptb$.
The outermost loop of $p$ is oriented clockwise, so
\begin{align}
  \label{eqn:w-p0}
  w([p]_0) = w(p) + 1.
\end{align}
Substituting equations (\ref{eqn:ell-npm}), (\ref{eqn:r-p0-p1}), and
(\ref{eqn:w-p0}) into equations (\ref{eqn:h-loop}) and
(\ref{eqn:q-loop}), we find
\begin{align}
  \label{eqn:h-q-1A}
  &h^+(\Tbar, [p]_0) = h^+(T,p), &
  &q^+(\Tbar,[p]_0) = q^+(T,p) + 1.
\end{align}
As an ungraded vector space $A_p = C_p$ is identified as the subspace
$C_p \otimes x \subset C_p \otimes A^{dc} = C_{[p]_0}$, so from
equation (\ref{eqn:h-q-1A}) and the fact that $x$ has bigrading
$(0,-1)$, it follows that $A_p = C_p^+$ as bigraded vector spaces.

Consider the case that $p$ is type $\ptc$.
The outermost loop of $p$ is oriented counterclockwise, so
\begin{align}
  \label{eqn:w-p1}
  w([p]_1) = w(p) - 1.
\end{align}
Substituting equations (\ref{eqn:ell-npm}), (\ref{eqn:r-p0-p1}), and
(\ref{eqn:w-p1}) into equations (\ref{eqn:h-loop}) and
(\ref{eqn:q-loop}), we find
\begin{align}
  \label{eqn:h-q-1B}
  &h^+(\Tbar, [p]_1) = h^+(T,p), &
  &q^+(\Tbar,[p]_1) = q^+(T,p) - 1.
\end{align}
As an ungraded vector space $A_p = C_p$ is identified as the subspace
$C_p \otimes e \subset C_p \otimes A^{db} = C_{[p]_1}$,
so from equation (\ref{eqn:h-q-1B}) and the fact that and $e$ has
bigrading $(0,1)$ it follows that $A_p = C_p^+$ as bigraded vector
spaces.

Consider the case that $p$ is type $\pto$.
The winding numbers of the planar tangles induced by $p$ are given by
equation (\ref{eqn:w-p0}) for $[p]_0$ and equation (\ref{eqn:w-p1})
for $[p]_1$, so
equations (\ref{eqn:h-q-1A}) and (\ref{eqn:h-q-1B}) hold for
$[p]_0$ and $[p]_1$.
As an ungraded vector space $A = C_p$ is identified as
$W_p \otimes A^{bc}$, where
$C_{[p]_0} = W_p \otimes e$ and
$C_{[p]_1} = W_p \otimes x$,
so from equations (\ref{eqn:h-q-1A}) and (\ref{eqn:h-q-1B}) and the
bigradings of $e$ and $x$ it follows that $A_p = C_p^+$ as bigraded
vector spaces.
\end{proof}

\section{Induced saddles}
\label{appendix:induced-saddles}

For each saddle $s:p\rightarrow q$ in the cube of resolutions of $T$,
we compute the map $\partialbar^1(s)$ defined in equation
(\ref{eqn:dbar-1-s-def}), which collects the terms of
$\partial_\Tbar^0$ corresponding to the induced saddles of $s$, as
well as the terms of $\partial_\Tbar^+$ corresponding to pairs of
successive saddles, one of which is an induced saddle of $s$ and one
of which is an ancillary saddle:
\begin{align*}
  \partialbar^1(s) =
  T_{00}(s) + T_{11}(s) +
  \Ttilde^L_{10}(n_q) Q_{00}(s) +
  Q_{10}(n_q) \Ttilde^L_{00}(s) +
  \Ttilde^L_{11}(s) Q_{10}(n_p) +
  Q_{11}(s) \Ttilde^L_{10}(n_p).
\end{align*}
The terms of $\partialbar^1(s)$ involving pairs of saddles correspond
to the following square of induced and ancillary saddles:
\begin{eqnarray*}
\begin{tikzcd}
  C_{[p]_0}
  \arrow{r}[swap]{[s]_0}
  \arrow{d}[swap]{n_p} &
  C_{[q]_0}
  \arrow{d}[swap]{n_q} \\
  C_{[p]_1}
  \arrow{r}[swap]{[s]_1} &
  C_{[q]_1},
\end{tikzcd}
\end{eqnarray*}
and by Lemma \ref{lemma:zero-unless-I-N} give a nonzero contribution
to $\partialbar^1(s)$ only if this square is interleaved or nested.
We also compute the maps
$Q_{00}(s)$, $\Ttilde_{00}^L(s)$, $Q_{11}(s)$, and
$\Ttilde_{11}^L(s)$ due to the induced
saddles $[s]_0$ and $[s]_1$ of $s$ that contribute to $\partialbar^2$.
These results are straightforward calculations that are very similar
to the calculations used to prove Lemmas \ref{lemma:1A},
\ref{lemma:1B}, and \ref{lemma:2}, so we will be somewhat brief.

\begin{lemma}
\label{lemma:W2-m}
For a type $W2_-$ saddle $s:p(\ptc) \rightarrow q(\ptb)$, the map
$\partialbar^1(s)$ has the form given in equation
(\ref{eqn:dbar-1-s-form}), where
\begin{align*}
  &\partial_A(s) = 0, &
  &\beta_1(s) = Q(s), &
  &\alpha_2(s) = Q(s).
\end{align*}
The saddle $s$ gives the following map that contributes to
$\partialbar^2$:
\begin{align*}
  \Ttilde^L_{00}(s) &=
  \left(\begin{array}{ccc}
      0 & 0 & 0 \\
      0 & Q(s) & 0 \\
      0 & 0 & 0
  \end{array}\right).
\end{align*}
\end{lemma}

\begin{proof}
We have the following square of saddles, which is neither interleaved
nor nested:
\begin{eqnarray*}
\begin{tikzcd}
  C_{[p]_0}
  \arrow{r}{C_+^L}[swap]{[s]_0}
  \arrow{d}{C_+^R}[swap]{n_p} &
  C_{[q]_0}
  \arrow{d}{C_-^L}[swap]{n_q} \\
  C_{[p]_1}
  \arrow{r}{C_-^R}[swap]{[s]_1} &
  C_{[q]_1}.
\end{tikzcd}
\end{eqnarray*}
We have the following vector spaces and linear maps:
\begin{align*}
  & C_p = A^{\otimes c(p)}, &
  & C_q = C_p, &
  & Q(s) = \id_{C_p}, \\
  & C_{[p]_0} = C_p, &
  & C_{[q]_0} = C_p \otimes A^{dc}, &
  & T_{00}(s) = C_p \otimes \etadot,\, 
  \Ttilde^L_{00}(s) = C_p \otimes \eta, \\
  & C_{[p]_1} = C_p \otimes A^{db}, &
  & C_{[q]_1} = C_p, &
  & T_{11}(s) = C_p \otimes \epsilondot.
\end{align*}
Using the decompositions of $C_p$ and $C_q$ defined in equations
(\ref{eqn:decomp-1B}) and (\ref{eqn:decomp-1A}), we find that
\begin{align*}
  \partialbar^1(s) =
  T_{00}(s) + T_{11}(s) =
  \left(\begin{array}{ccc}
      0 & Q(s) & 0 \\
      0 & 0 & 0 \\
      0 & 0 & 0 \\
  \end{array}\right) +
  \left(\begin{array}{ccc}
      0 & 0 & 0 \\
      0 & 0 & 0 \\
      Q(s) & 0 & 0 \\
  \end{array}\right) =
  \left(\begin{array}{ccc}
      0 & Q(s) & 0 \\
      0 & 0 & 0 \\
      Q(s) & 0 & 0 \\
  \end{array}\right),
\end{align*}
and the saddle $[s]_0$ gives the map $\Ttilde^L_{00}(s)$ stated in the
lemma.
\end{proof}

\begin{lemma}
For a type $W2_+$ saddle $s:p(\ptb) \rightarrow q(\ptc)$, the map
$\partialbar^1(s)$ has the form given in equation
(\ref{eqn:dbar-1-s-form}), where
\begin{align*}
  &\partial_A(s) = 0, &
  &\beta_1(s) = 0, &
  &\alpha_2(s) = 0.
\end{align*}
The saddle $s$ gives the following map that contributes to
$\partialbar^2$:
\begin{align*}
  \Ttilde^L_{00}(s) &=
  \left(\begin{array}{ccc}
    0 & 0 & 0 \\
    P(s) & 0 & 0 \\
    0 & 0 & 0
  \end{array}\right).
\end{align*}
\end{lemma}

\begin{proof}
We have the following square of saddles, which is neither interleaved
nor nested:
\begin{eqnarray*}
\begin{tikzcd}
  C_{[p]_0}
  \arrow{r}{C_-^L}[swap]{[s]_0}
  \arrow{d}{C_-^L}[swap]{n_p} &
  C_{[q]_0}
  \arrow{d}{C_+^R}[swap]{n_q} \\
  C_{[p]_1}
  \arrow{r}{C_+^R}[swap]{[s]_1} &
  C_{[q]_1}.
\end{tikzcd}
\end{eqnarray*}
A calculation similar to the one used in the proof of Lemma
\ref{lemma:W2-m} shows that
\begin{align*}
  \partialbar^1(s) = T_{00}(s) + T_{11}(s) =
  \left(\begin{array}{ccc}
    0 & 0 & 0 \\
    0 & P(s) & 0 \\
    0 & 0 & 0
  \end{array}\right) +
  \left(\begin{array}{ccc}
    0 & 0 & 0 \\
    0 & 0 & 0 \\
    0 & 0 & P(s)
  \end{array}\right) =
  \left(\begin{array}{ccc}
    0 & 0 & 0 \\
    0 & P(s) & 0 \\
    0 & 0 & P(s)
  \end{array}\right),
\end{align*}
and the saddle $[s]_0$ gives the map $\Ttilde^L_{00}(s)$ stated
in the lemma.
\end{proof}

\begin{lemma}
\label{lemma:W3-m}
For a type $W3_\pm$ saddle $s:p(\ptb) \rightarrow q(\ptb)$ or
$s:p(\ptc) \rightarrow q(\ptc)$, the map $\partialbar^1(s)$ has the
form given in equation (\ref{eqn:dbar-1-s-form}), where
\begin{align*}
  &\partial_A(s) = 0, &
  &\beta_1(s) = 0, &
  &\alpha_2(s) = 0.
\end{align*}
A type $W3_-$ saddle $s:p(\ptb) \rightarrow q(\ptb)$ gives the
following maps that contribute to $\partialbar^2$:
\begin{align*}
  Q_{00}(s) &=
  \left(\begin{array}{ccc}
    Q(s) & 0 & 0 \\
    0 & Q(s) & 0 \\
    0 & 0 & 0
  \end{array}\right), &
  Q_{11}(s) &=
  \left(\begin{array}{ccc}
    0 & 0 & 0 \\
    0 & 0 & 0 \\
    0 & 0 & Q(s)
  \end{array}\right).
\end{align*}
A type $W3_-$ saddle $s:p(\ptc) \rightarrow q(\ptc)$ gives the
following maps that contribute to $\partialbar^2$:
\begin{align*}
  Q_{00}(s) &=
  \left(\begin{array}{ccc}
    0 & 0 & 0 \\
    0 & Q(s) & 0 \\
    0 & 0 & 0
  \end{array}\right), &
  Q_{11}(s) &=
  \left(\begin{array}{ccc}
    Q(s) & 0 & 0 \\
    0 & 0 & 0 \\
    0 & 0 & Q(s)
  \end{array}\right).
\end{align*}
A type $W3_+$ saddle does not give any maps that contribute to
$\partialbar^2$.
\end{lemma}

\begin{proof}
For a saddle $s:p(\ptb) \rightarrow q(\ptb)$, we have
the following square of saddles, which is neither interleaved nor
nested:
\begin{eqnarray*}
\begin{tikzcd}
  C_{[p]_0}
  \arrow{r}{W_\pm}[swap]{[s]_0}
  \arrow{d}{C_-^L}[swap]{n_p} &
  C_{[q]_0}
  \arrow{d}{C_-^L}[swap]{n_q} \\
  C_{[p]_1}
  \arrow{r}{W_\pm}[swap]{[s]_1} &
  C_{[q]_1}.
\end{tikzcd}
\end{eqnarray*}
Thus $\partialbar^1(s) = 0$.
A calculation similar to the one used in the proof of Lemma
\ref{lemma:W2-m} shows that for a saddle $s$ of type $W3_-$ the
saddles $[s]_0$ and $[s]_1$ give the maps $Q_{00}(s)$ and $Q_{11}(s)$
stated in the lemma.
The proof for a saddle $s:p(\ptc) \rightarrow q(\ptc)$ is similar.
\end{proof}

\begin{lemma}
For a type $W4_\pm$ saddle $s:p(\pto) \rightarrow q(\pto)$, the map
$\partialbar^1(s)$ has the form given in equation
(\ref{eqn:dbar-1-s-form}), where
\begin{align*}
  &\partial_A(s) = 0, &
  &\beta_1(s) = 0, &
  &\alpha_2(s) = 0.
\end{align*}
A type $W4_-$ saddle $s$ gives the following maps that contribute to
$\partialbar^2$:
\begin{align*}
  Q_{00}(s) &=
  \left(\begin{array}{ccc}
    Q(s) \id_{C_{[p]_0}} & 0 & 0 \\
    0 & 0 & 0 \\
    0 & 0 & 0
  \end{array}\right), &
  Q_{11}(s) &=
  \left(\begin{array}{ccc}
    Q(s) \id_{C_{[p]_1}} & 0 & 0 \\
    0 & 0 & 0 \\
    0 & 0 & 0
  \end{array}\right).
\end{align*}
A type $W4_+$ saddle does not give any maps that contribute to
$\partialbar^2$.
\end{lemma}

\begin{proof}
We have the following square of saddles, which is neither interleaved
nor nested:
\begin{eqnarray*}
\begin{tikzcd}
  C_{[p]_0}
  \arrow{r}{W_\pm}[swap]{[s]_0}
  \arrow{d}{W_-}[swap]{n_p} &
  C_{[q]_0}
  \arrow{d}{W_-}[swap]{n_q} \\
  C_{[p]_1}
  \arrow{r}{W_\pm}[swap]{[s]_1} &
  C_{[q]_1}.
\end{tikzcd}
\end{eqnarray*}
Thus $\partialbar^1(s) = 0$.
A calculation similar to the one used in the proof of Lemma
\ref{lemma:W2-m} shows that for a saddle $s$ of type $W4_-$ the
saddles $[s]_0$ and $[s]_1$ give the maps $Q_{00}(s)$ and $Q_{11}(s)$
stated in the lemma.
\end{proof}

\begin{lemma}
\label{lemma:C1-1A}
For a type $C1_\pm^\beta$ saddle $s:p(\ptb) \rightarrow q(\ptb)$, the
term $\partialbar^1(s)$ has the form given in equation
(\ref{eqn:dbar-1-s-form}), where
\begin{align*}
  &\partial_A(s) = T(s), &
  &\beta_1(s) = \Ttilde(s), &
  &\alpha_2(s) = 0.
\end{align*}
The saddle $s$ gives the following map that contributes to
$\partialbar^2$:
\begin{align*}
  \Ttilde^L_{11}(s) &=
  \left(\begin{array}{ccc}
    0 & 0 & 0 \\
    0 & 0 & 0 \\
    0 & 0 & \Ttilde^R(s)
  \end{array}\right).
\end{align*}
\end{lemma}

\begin{proof}
We have the following square of saddles, which is neither interleaved
nor nested:
\begin{eqnarray*}
\begin{tikzcd}
  C_{[p]_0}
  \arrow{r}{C_\pm^C}[swap]{[s]_0}
  \arrow{d}{C_-^L}[swap]{n_p} &
  C_{[q]_0}
  \arrow{d}{C_-^L}[swap]{n_q} \\
  C_{[p]_1}
  \arrow{r}{C_\pm^\betabar}[swap]{[s]_1} &
  C_{[q]_1}.
\end{tikzcd}
\end{eqnarray*}
For a saddle $s$ of type $C1_+^\beta$,
we have the following vector spaces and linear maps:
\begin{align*}
  &C_p = A^{\otimes c(p)}, &
  &C_q = C_p \otimes A, &
  &
  T(s) = \id_{C_p} \otimes \etadot,\,
  \Ttilde(s) = \id_{C_p} \otimes \eta, \\
  &C_{[p]_0} = C_p \otimes A^{dc}, &
  &C_{[q]_0} = C_q \otimes A^{dc} = C_p \otimes A \otimes A^{dc}, &
  &T_{00}(s) = \id_{C_p} \otimes \Delta, \\
  &C_{[p]_1} = C_p, &
  &C_{[q]_1} = C_q = C_p \otimes A, &
  &T_{11}(s) = \id_{C_p} \otimes \etadot,\,
  \Ttilde_{11}(s) = \id_{C_p} \otimes \eta.
\end{align*}
Note that
\begin{align*}
  T_{00}(s) =
  \id_{C_p} \otimes \Delta =
  \id_{C_p} \otimes \etadot \otimes \id_A +
  \id_{C_p} \otimes \eta \otimes \id_{xe} =
  T(s) \otimes \id_A + \Ttilde(s) \otimes \id_{xe}.
\end{align*}
Using the decompositions of $C_p$ and $C_q$ given in equation
(\ref{eqn:decomp-1A}), we find that
\begin{align*}
  \partialbar^1(s) =
  T_{00}(s) + T_{11}(s) =
  \left(\begin{array}{ccc}
    T(s) & \Ttilde(s) & 0 \\
    0 & T(s) & 0 \\
    0 & 0 & 0
  \end{array}\right) +
  \left(\begin{array}{ccc}
    0 & 0 & 0 \\
    0 & 0 & 0 \\
    0 & 0 & T(s)
  \end{array}\right) =
  \left(\begin{array}{ccc}
    T(s) & \Ttilde(s) & 0 \\
    0 & T(s) & 0 \\
    0 & 0 & T(s)
  \end{array}\right),
\end{align*}
and the saddle $[s]_1$ gives the map $\Ttilde_{11}^L(s)$ stated in the
lemma, where we have used the fact that $[s]_1$ attaches to the
opposite side of the arc component as $s$.
The proof for a saddle of type $C1_-^\beta$ is similar.
\end{proof}

\begin{lemma}
For a type $C1_\pm^\beta$ saddle $s:p(\ptc) \rightarrow q(\ptc)$, the
term $\partialbar^1(s)$ has the form given in equation
(\ref{eqn:dbar-1-s-form}), where
\begin{align*}
  &\partial_A(s) = T(s), &
  &\beta_1(s) = 0, &
  &\alpha_2(s) = \Ttilde(s).
\end{align*}
The saddle $s$ gives the following map that contributes to
$\partialbar^2$:
\begin{align*}
  \Ttilde^L_{00}(s) &=
  \left(\begin{array}{ccc}
    0 & 0 & 0 \\
    0 & \Ttilde^R(s) & 0 \\
    0 & 0 & 0
  \end{array}\right).
\end{align*}
\end{lemma}

\begin{proof}
We have the following square of saddles, which is neither interleaved
nor nested:
\begin{eqnarray*}
\begin{tikzcd}
  C_{[p]_0}
  \arrow{r}{C_\pm^\betabar}[swap]{[s]_0}
  \arrow{d}{C_+^R}[swap]{n_p} &
  C_{[q]_0}
  \arrow{d}{C_+^R}[swap]{n_q} \\
  C_{[p]_1}
  \arrow{r}{C_\pm^C}[swap]{[s]_1} &
  C_{[q]_1}.
\end{tikzcd}
\end{eqnarray*}
A calculation similar to the one used in the proof of Lemma
\ref{lemma:C1-1A} shows that
\begin{align*}
  \partialbar^1(s) = T_{00}(s) + T_{11}(s) =
  \left(\begin{array}{ccc}
    0 & 0 & 0 \\
    0 & T(s) & 0 \\
    0 & 0 & 0
  \end{array}\right) +
  \left(\begin{array}{ccc}
    T(s) & 0 & 0 \\
    0 & 0 & 0 \\
    \Ttilde(s) & 0 & T(s)
  \end{array}\right) =
  \left(\begin{array}{ccc}
    T(s) & 0 & 0 \\
    0 & T(s) & 0 \\
    \Ttilde(s) & 0 & T(s)
  \end{array}\right),
\end{align*}
and the saddle $[s]_0$ gives the map $\Ttilde_{00}^L(s)$ stated in the
lemma, where we have used the fact that $[s]_0$ attaches to the
opposite side of the arc component as $s$.
\end{proof}

\begin{lemma}
\label{lemma:C2-p-L}
For a type $C2_+^L$ saddle $s:p(\ptb) \rightarrow q(\pto)$, the map
$\partialbar^1(s)$ has the form given in equation
(\ref{eqn:dbar-1-s-form}), where
\begin{align*}
  &\partial_A(s) = T(s), &
  &\beta_1(s) = \Ttilde^L(s), &
  &\alpha_2(s) = 0.
\end{align*}
The saddle $s$ gives the following map that contributes to
$\partialbar^2$:
\begin{align*}
  Q_{11}(s) &=
  \left(\begin{array}{ccc}
      0 & 0 & T(s) \\
      0 & 0 & 0 \\
      0 & 0 & 0
  \end{array}\right).
\end{align*}
\end{lemma}

\begin{proof}
We have the following square of saddles, which is interleaved of type
$I_-$:
\begin{eqnarray*}
\begin{tikzcd}
  C_{[p]_0}
  \arrow{r}{C_-^R}[swap]{[s]_0}
  \arrow{d}{C_-^L}[swap]{n_p} &
  C_{[q]_0}
  \arrow{d}{W_-}[swap]{n_q} \\
  C_{[p]_1}
  \arrow{r}{W_-}[swap]{[s]_1} &
  C_{[q]_1}.
\end{tikzcd}
\end{eqnarray*}
We have the following vector spaces and linear maps:
\begin{align*}
  & C_p = A^{\otimes c(p)} = W_q, &
  & C_q = C_p \otimes A^{bc}, &
  &
  T(s) = \id_{C_p} \otimes \etadot,\,
  \Ttilde^L(s) = \id_{C_p} \otimes \eta, \\
  & C_{[p]_0} = C_p \otimes A^{dc}, &
  & C_{[q]_0} = C_p = C_p \otimes e, &
  & T_{00}(s) = \id_{C_p} \otimes \epsilondot, \\
  & C_{[p]_1} = C_p, &
  & C_{[q]_1} = C_p = C_p \otimes x, &
  & Q_{11}(s) = \id_{C_p}.
\end{align*}
Using the decompositions of $C_p$ and $C_q$ given in equations
(\ref{eqn:decomp-1A}) and (\ref{eqn:decomp-2}), we find that
\begin{align*}
  \partialbar^1(s) &=
  T_{00}(s) + Q_{11}(s) \Ttilde_{10}^L(n_p) \\
  &=
  \left(\begin{array}{ccc}
    0 & \tilde{T}^L(s) & 0 \\
    0 & 0 & 0 \\
    0 & 0 & 0
  \end{array}\right) +
  \left(\begin{array}{ccc}
    0 & 0 & T(s) \\
    0 & 0 & 0 \\
    0 & 0 & 0
  \end{array}\right)
  \left(\begin{array}{ccc}
    0 & 0 & 0 \\
    0 & 0 & 0 \\
    \id_{C_p} & 0 & 0
  \end{array}\right) =
  \left(\begin{array}{ccc}
    T(s) & \tilde{T}^L(s) & 0 \\
    0 & 0 & 0 \\
    0 & 0 & 0
  \end{array}\right),
\end{align*}
and the saddle $[s]_1$ gives the map $Q_{11}(s)$ stated in the
lemma.
\end{proof}

\begin{lemma}
For a a type $C2_-^L$ saddle $s:p(\pto) \rightarrow q(\ptb)$, the map
$\partialbar^1(s)$ has the form given in equation
(\ref{eqn:dbar-1-s-form}), where
\begin{align*}
  &\partial_A(s) = T(s), &
  &\beta_1(s) = 0, &
  &\alpha_2(s) = 0.
\end{align*}
The saddle $s$ does not give any maps that contribute to
$\partialbar^2$.
\end{lemma}

\begin{proof}
We have the following square of saddles, which is neither interleaved
nor nested:
\begin{eqnarray*}
\begin{tikzcd}
  C_{[p]_0}
  \arrow{r}{C_+^R}[swap]{[s]_0}
  \arrow{d}{W_-}[swap]{n_p} &
  C_{[q]_0}
  \arrow{d}{C_-^L}[swap]{n_q} \\
  C_{[p]_1}
  \arrow{r}{W_+}[swap]{[s]_1} &
  C_{[q]_1}.
\end{tikzcd}
\end{eqnarray*}
A calculation similar to the one used in the proof of Lemma
\ref{lemma:C2-p-L} shows that
\begin{align*}
  \partialbar^1(s) = T_{00}(s) =
  \left(\begin{array}{ccc}
    T(s) & 0 & 0 \\
    0 & 0 & 0 \\
    0 & 0 & 0
  \end{array}\right).
\end{align*}
\end{proof}

\begin{lemma}
For a type $C2_+^R$ saddle $s:p(\ptc) \rightarrow q(\pto)$, the map
$\partialbar^1(s)$ has the form given in equation
(\ref{eqn:dbar-1-s-form}), where
\begin{align*}
  &\partial_A(s) = T(s), &
  &\beta_1(s) = 0, &
  &\alpha_2(s) = 0.
\end{align*}
The saddle $s$ gives the following map that contributes to
$\partialbar^1(s)$:
\begin{align*}
  \Ttilde_{11}^L(s) &=
  \left(\begin{array}{ccc}
    0 & 0 & T(s) \\
    0 & 0 & 0 \\
    0 & 0 & 0
  \end{array}\right).
\end{align*}
\end{lemma}

\begin{proof}
We have the following square of saddles, which is neither interleaved
nor nested:
\begin{eqnarray*}
\begin{tikzcd}
  C_{[p]_0}
  \arrow{r}{W_+}[swap]{[s]_0}
  \arrow{d}{C_+^R}[swap]{n_p} &
  C_{[q]_0}
  \arrow{d}{W_-}[swap]{n_q} \\
  C_{[p]_1}
  \arrow{r}{C_-^L}[swap]{[s]_1} &
  C_{[q]_1}.
\end{tikzcd}
\end{eqnarray*}
A calculation similar to the one used in the proof of Lemma
\ref{lemma:C2-p-L} shows that
\begin{align*}
  \partialbar^1(s) = T_{11}(s) =
  \left(\begin{array}{ccc}
    T(s) & 0 & 0 \\
    0 & 0 & 0 \\
    0 & 0 & 0
  \end{array}\right),
\end{align*}
and the saddle $[s]_1$ gives the map $\Ttilde_{11}^L(s)$ stated in the
lemma.
\end{proof}

\begin{lemma}
For a type $C2_-^R$ saddle $s:p(\pto) \rightarrow q(\ptc)$, the map
$\partialbar^1(s)$ has the form given in equation
(\ref{eqn:dbar-1-s-form}), where
\begin{align*}
  &\partial_A(s) = T(s), &
  &\beta_1(s) = 0, &
  &\alpha_2(s) = \Ttilde^R(s).
\end{align*}
The saddle $s$ gives the following maps that contribute to
$\partialbar^2$:
\begin{align*}
  Q_{00}(s) &=
  \left(\begin{array}{ccc}
    0 & 0 & 0 \\
    T(s) & 0 & 0 \\
    0 & 0 & 0
  \end{array}\right), &
  \Ttilde_{11}^L(s) &=
  \left(\begin{array}{ccc}
    \Ttilde^R(s) & 0 & 0 \\
    0 & 0 & 0 \\
    0 & 0 & 0
  \end{array}\right).
\end{align*}
\end{lemma}

\begin{proof}
We have the following square of saddles, which is interleaved of type
$I_+$:
\begin{eqnarray*}
\begin{tikzcd}
  C_{[p]_0}
  \arrow{r}{W_-}[swap]{[s]_0}
  \arrow{d}{W_-}[swap]{n_p} &
  C_{[q]_0}
  \arrow{d}{C_+^R}[swap]{n_q} \\
  C_{[p]_1}
  \arrow{r}{C_+^L}[swap]{[s]_1} &
  C_{[q]_1}.
\end{tikzcd}
\end{eqnarray*}
We have the following vector spaces and linear maps:
\begin{align*}
  & C_p = C_q \otimes A^{bc}, &
  & C_q = A^{\otimes c(q)} = W_p, &
  &
  T(s) = \id_{C_q} \otimes \epsilondot,\,
  \Ttilde^R(s) = \id_{C_q} \otimes \epsilon, \\
  & C_{[p]_0} = C_q = C_q \otimes e, &
  & C_{[q]_0} = C_q, &
  & Q_{00}(s) = \id_{C_q}, \\
  & C_{[p]_1} = C_q = C_q \otimes x, &
  & C_{[q]_1} = C_q \otimes A^{db}, &
  &
  T_{11}(s) = \id_{C_q} \otimes \etadot,\,
  \Ttilde_{11}^L(s) = \id_{C_q} \otimes \eta.
\end{align*}
Using the decompositions of $C_p$ and $C_q$ given in equations
(\ref{eqn:decomp-2}) and (\ref{eqn:decomp-1B}), we find that
\begin{align*}
  \partialbar^1(s) &=
  T_{11}(s) + \Ttilde_{11}^L(s) Q_{10}(n_p) \\
  &=
  \left(\begin{array}{ccc}
    0 & 0 & 0 \\
    0 & 0 & 0 \\
    \Ttilde^R(s) & 0 & 0
  \end{array}\right) +
  \left(\begin{array}{ccc}
    \Ttilde^R(s) & 0 & 0 \\
    0 & 0 & 0 \\
    0 & 0 & 0
  \end{array}\right)
  \left(\begin{array}{ccc}
    \id_{C_{[p]_0} C_{[p]_1}} & 0 & 0 \\
    0 & 0 & 0 \\
    0 & 0 & 0
  \end{array}\right) =
  \left(\begin{array}{ccc}
    T(s) & 0 & 0 \\
    0 & 0 & 0 \\
    \Ttilde^R(s) & 0 & 0
  \end{array}\right),
\end{align*}
where we have used the fact that
\begin{align*}
  \Ttilde^R(s) \id_{C_{[p]_0} C_{[p]_1}} =
  (\id_{C_q} \otimes \epsilon)(\id_{C_q} \otimes \id_{xe}) =
  \id_{C_q} \otimes \epsilondot = T(s),
\end{align*}
and the saddles $[s]_0$ and $[s]_1$ give the maps $Q_{00}(s)$ and
$\Ttilde_{11}^L(s)$ stated in the lemma.
\end{proof}

\begin{lemma}
\label{lemma:C3}
For a type $C3_\pm^\beta$ saddle $s:p(\ptb) \rightarrow q(\ptb)$ or
$s:p(\ptc) \rightarrow q(\ptc)$, the map $\partialbar^1(s)$ has
the form given in equation (\ref{eqn:dbar-1-s-form}), where
\begin{align*}
  &\partial_A(s) = T(s), &
  &\beta_1(s) = 0, &
  &\alpha_2(s) = 0.
\end{align*}
A type $C3_\pm^\beta$ saddle $s:p(\ptb) \rightarrow q(\ptb)$ gives the
following maps that contribute to $\partialbar^2$:
\begin{align*}
  \Ttilde^L_{00}(s) &=
  \left(\begin{array}{ccc}
    \Ttilde^L(s) & 0 & 0 \\
    0 & \Ttilde^L(s) & 0 \\
    0 & 0 & 0
  \end{array}\right), &
  \Ttilde^L_{11}(s) &=
  \left(\begin{array}{ccc}
    0 & 0 & 0 \\
    0 & 0 & 0 \\
    0 & 0 & \Ttilde^L(s)
  \end{array}\right).
\end{align*}
A type $C3_\pm^\beta$ saddle $s:p(\ptc) \rightarrow q(\ptc)$ gives
the following maps that contribute to $\partialbar^2$:
\begin{align*}
  \Ttilde^L_{00}(s) &=
  \left(\begin{array}{ccc}
    0 & 0 & 0 \\
    0 & \Ttilde^L(s) & 0 \\
    0 & 0 & 0
  \end{array}\right), &
  \Ttilde^L_{11}(s) &=
  \left(\begin{array}{ccc}
    \Ttilde^L(s) & 0 & 0 \\
    0 & 0 & 0 \\
    0 & 0 & \Ttilde^L
  \end{array}\right).
\end{align*}
\end{lemma}

\begin{proof}
For $s:p(\ptb) \rightarrow q(\ptb)$, we have the following square of
saddles, which is neither interleaved nor nested:
\begin{eqnarray*}
\begin{tikzcd}
  C_{[p]_0}
  \arrow{r}{C_\pm^\beta}[swap]{[s]_0}
  \arrow{d}{C_-^L}[swap]{n_p} &
  C_{[q]_0}
  \arrow{d}{C_-^L}[swap]{n_q} \\
  C_{[p]_1}
  \arrow{r}{C_\pm^\beta}[swap]{[s]_1} &
  C_{[q]_1}.
\end{tikzcd}
\end{eqnarray*}
A calculation shows:
\begin{align*}
  \partialbar^1(s) = T_{00}(s) + T_{11}(s) =
  \left(\begin{array}{ccc}
    T(s) & 0 & 0 \\
    0 & T(s) & 0 \\
    0 & 0 & 0
  \end{array}\right) +
  \left(\begin{array}{ccc}
    0 & 0 & 0 \\
    0 & 0 & 0 \\
    0 & 0 & T(s)
  \end{array}\right) =
  \left(\begin{array}{ccc}
    T(s) & 0 & 0 \\
    0 & T(s) & 0 \\
    0 & 0 & T(s)
  \end{array}\right),
\end{align*}
and the saddles $[s]_0$ and $[s]_1$ give the maps $\Ttilde_{00}^L(s)$
and $\Ttilde_{11}^L(s)$ stated in the lemma.
The argument for $s:p(\ptc) \rightarrow q(\ptc)$ is similar.
Note that if $\beta = R$ then $\Ttilde^L(s) = 0$.
\end{proof}

\begin{lemma}
For a type $C4_\pm^\beta$ saddle $s:p(\pto) \rightarrow q(\pto)$, the
map $\partialbar^1(s)$ has the form given in equation
(\ref{eqn:dbar-1-s-form}), where
\begin{align*}
  &\partial_A(s) = T(s), &
  &\beta_1(s) = 0, &
  &\alpha_2(s) = 0.
\end{align*}
The saddle $s$ gives the following maps that contribute to
$\partialbar^2$:
\begin{align*}
  \Ttilde^L_{00}(s) &=
  \left(\begin{array}{ccc}
    \Ttilde^L(s) \id_{C_{[p]_0}} & 0 & 0 \\
    0 & 0 & 0 \\
    0 & 0 & 0
  \end{array}\right), &
  \Ttilde^L_{11}(s) &=
  \left(\begin{array}{ccc}
    \Ttilde^L(s) \id_{C_{[p]_1}} & 0 & 0 \\
    0 & 0 & 0 \\
    0 & 0 & 0
  \end{array}\right).
\end{align*}
\end{lemma}

\begin{proof}
We have the following square of saddles, which is disjoint:
\begin{eqnarray*}
\begin{tikzcd}
  C_{[p]_0}
  \arrow{r}{C_\pm^\beta}[swap]{[s]_0}
  \arrow{d}{W_-}[swap]{n_p} &
  C_{[q]_0}
  \arrow{d}{W_-}[swap]{n_q} \\
  C_{[p]_1}
  \arrow{r}{C_\pm^\beta}[swap]{[s]_1} &
  C_{[q]_1}.
\end{tikzcd}
\end{eqnarray*}
A calculation shows that
\begin{align*}
  \partialbar^1(s) = T_{00}(s) + T_{11}(s) =
  \left(\begin{array}{ccc}
    T(s)\id_{C_{[p]_0}} & 0 & 0 \\
    0 & 0 & 0 \\
    0 & 0 & 0
  \end{array}\right) +
  \left(\begin{array}{ccc}
    T(s) \id_{C_{[p]_1}} & 0 & 0 \\
    0 & 0 & 0 \\
    0 & 0 & 0
  \end{array}\right) =
  \left(\begin{array}{ccc}
    T(s) & 0 & 0 \\
    0 & 0 & 0 \\
    0 & 0 & 0
  \end{array}\right),
\end{align*}
and the saddles $[s]_0$ and $[s]_1$ give the maps $\Ttilde_{00}^L(s)$
and $\Ttilde_{11}^L(s)$ stated in the lemma.
Note that if $\beta = R$ then $\Ttilde^L(s) = 0$.
\end{proof}

\begin{lemma}
For a type $C1_\pm^C$ saddle $s:p(\pto) \rightarrow q(\pto)$, the map
$\partialbar^1(s)$ has the form given in equation
(\ref{eqn:dbar-1-s-form}), where
\begin{align*}
  &\partial_A(s) = T(s), &
  &\beta_1(s) = 0, &
  &\alpha_2(s) = 0.
\end{align*}
The saddle $s$ gives the following maps that contribute to
$\partialbar^2$:
\begin{align*}
  \Ttilde^L_{00}(s) &=
  \left(\begin{array}{ccc}
    \Ttilde^{L_0}(s) \id_{C_{[p]_0}} & 0 & 0 \\
    0 & 0 & 0 \\
    0 & 0 & 0
  \end{array}\right), &
  \Ttilde^L_{11}(s) &=
  \left(\begin{array}{ccc}
    \Ttilde^{L_1}(s) \id_{C_{[p]_1}} & 0 & 0 \\
    0 & 0 & 0 \\
    0 & 0 & 0
  \end{array}\right),
\end{align*}
where the maps $\Ttilde^{L_0}(s)$ and $\Ttilde^{L_1}(s)$ are described
in the proof.
\end{lemma}

\begin{proof}
Define $\sigma = R$ if $s$ splits or merges a circle \emph{inside} the
circle component of $p$ containing $b$ and $c$, and $\sigma = L$
otherwise.
We have the following square of saddles, which is nested of type
$N_\pm^\sigmabar$:
\begin{eqnarray*}
\begin{tikzcd}
  C_{[p]_0}
  \arrow{r}{C_\pm^\sigma}[swap]{[s]_0}
  \arrow{d}{W_-}[swap]{n_p} &
  C_{[q]_0}
  \arrow{d}{W_-}[swap]{n_q} \\
  C_{[p]_1}
  \arrow{r}{C_\pm^\sigmabar}[swap]{[s]_1} &
  C_{[q]_1}.
\end{tikzcd}
\end{eqnarray*}
Consider a type $C1_+^C$ saddle $s:p(\pto) \rightarrow q(\pto)$.
The vector spaces and linear maps are
\begin{align*}
  & C_p = W_p \otimes A^{bc} &
  & C_q = W_p \otimes A \otimes A^{bc} &
  &
  T(s) = \id_{W_p} \otimes \Delta, \\
  & C_{[p]_0} = W_p \otimes e &
  & C_{[q]_0} = W_p \otimes A \otimes e &
  &
  T_{00}(s) = T^\bullet(s) \id_{C_{[p]_0}},\,
  \Ttilde^L_{00}(s) = \Ttilde^{L_0}(s) \id_{C_{[p]_0}}, \\
  & C_{[p]_1} = W_p \otimes x &
  & C_{[q]_1} = W_p \otimes A \otimes x &
  &
  T_{11}(s) = T^\bullet(s) \id_{C_{[p]_1}},\,
  \Ttilde^L_{11}(s) = \Ttilde^{L_1}(s) \id_{C_{[p]_1}},
\end{align*}
where $W_p = A^{\otimes (c(p) - 1)}$ and we have defined
\begin{align*}
  & T^\bullet(s) = \id_{W_p} \otimes \etadot \otimes \id_A, &
  & T^\circ(s) = \id_{W_p} \otimes \eta \otimes \id_A, \\
  & \Ttilde^{L_0}(s) = \delta_{\sigma,L} T^\circ(s), &
  & \Ttilde^{L_1}(s) = \delta_{\sigmabar,L} T^\circ(s).
\end{align*}
Using the decompositions of $C_p$ and $C_q$ given in equation
(\ref{eqn:decomp-2}), we find that
\begin{align*}
  \partialbar^1(s) =
  T_{00}(s) + T_{11}(s) +
  Q_{10}(n_q) \Ttilde^L_{00}(s) + \Ttilde^L_{11}(s) Q_{10}(n_p) =
  \left(\begin{array}{ccc}
    T(s) & 0 & 0 \\
    0 & 0 & 0 \\
    0 & 0 & 0
  \end{array}\right).
\end{align*}
where we have used the fact that
\begin{eqnarray*}
  \lefteqn{T^\bullet(s)\id_{C_{[p]_0}} +
  T^\bullet(s)\id_{C_{[p]_1}} +
  \id_{C_{[q]_1}C_{[q]_0}}\Ttilde^{L_0}(s)\id_{C_{[p]_0}} +
  \Ttilde^{L_1}(s)\id_{C_{[p]_1}}\id_{C_{[p]_1}C_{[p]_0}} =} \\
  & &
  \id_{W_p} \otimes \etadot \otimes \id_{ee} +
  \id_{W_p} \otimes \etadot \otimes \id_{xx} +
  (\delta_{\sigma,L} + \delta_{\sigmabar,L})
  \id_{W_p} \otimes \eta \otimes \id_{xe} =
  \id_{W_p} \otimes \Delta = T(s),
\end{eqnarray*}
and the saddles $[s]_0$ and $[s]_1$ give the maps $\Ttilde_{00}^L(s)$
and $\Ttilde_{11}^L(s)$ stated in the lemma.
The argument for a saddle of type $C1_-^C$ is similar.
\end{proof}

\begin{lemma}
For a type $C2_\pm^C$ saddle $s:p(\pto) \rightarrow q(\pto)$, the map
$\partialbar^1(s)$ has the form given in equation
(\ref{eqn:dbar-1-s-form}), where
\begin{align*}
  &\partial_A(s) = T(s), &
  &\beta_1(s) = 0, &
  &\alpha_2(s) = 0.
\end{align*}
The saddle $s$ does not give any maps that contribute to
$\partialbar^2$.
\end{lemma}

\begin{proof}
We have the following square of saddles, which is neither interleaved
nor nested:
\begin{eqnarray*}
\begin{tikzcd}
  C_{[p]_0}
  \arrow{r}{C_\pm^C}[swap]{[s]_0}
  \arrow{d}{W_-}[swap]{n_p} &
  C_{[q]_0}
  \arrow{d}{W_-}[swap]{n_q} \\
  C_{[p]_1}
  \arrow{r}{C_\pm^C}[swap]{[s]_1} &
  C_{[q]_1}.
\end{tikzcd}
\end{eqnarray*}
A calculation shows that
\begin{align*}
  \partialbar^1(s) = T_{00}(s) + T_{11}(s) =
  \left(\begin{array}{ccc}
    T(s) \id_{C_{[p]_0}} & 0 & 0 \\
    0 & 0 & 0 \\
    0 & 0 & 0
  \end{array}\right) +
  \left(\begin{array}{ccc}
    T(s) \id_{C_{[p]_1}} & 0 & 0 \\
    0 & 0 & 0 \\
    0 & 0 & 0
  \end{array}\right) =
  \left(\begin{array}{ccc}
    T(s) & 0 & 0 \\
    0 & 0 & 0 \\
    0 & 0 & 0
  \end{array}\right).
\end{align*}
\end{proof}

\begin{lemma}
For a type $C3_\pm^C$ saddle
$s:p(\ptb) \rightarrow q(\ptb)$ or $s:p(\ptc) \rightarrow q(\ptc)$, the map
$\partialbar^1(s)$ has the form given in equation
(\ref{eqn:dbar-1-s-form}), where
\begin{align*}
  &\partial_A(s) = T(s), &
  &\beta_1(s) = 0, &
  &\alpha_2(s) = 0.
\end{align*}
The saddle $s$ does not give any maps that contribute to
$\partialbar^2$.
\end{lemma}

\begin{proof}
For a saddle $s:p(\ptb) \rightarrow q(\ptb)$, we have the following
square of saddles, which is neither interleaved nor nested:
\begin{eqnarray*}
\begin{tikzcd}
  C_{[p]_0}
  \arrow{r}{C_\pm^C}[swap]{[s]_0}
  \arrow{d}{C_-^L}[swap]{n_p} &
  C_{[q]_0}
  \arrow{d}{C_-^L}[swap]{n_q} \\
  C_{[p]_1}
  \arrow{r}{C_\pm^C}[swap]{[s]_1} &
  C_{[q]_1}.
\end{tikzcd}
\end{eqnarray*}
A calculation shows that
\begin{align*}
  \partialbar^1(s) = T_{00}(s) + T_{11}(s) =
  \left(\begin{array}{ccc}
    T(s) & 0 & 0 \\
    0 & T(s) & 0 \\
    0 & 0 & 0
  \end{array}\right) +
  \left(\begin{array}{ccc}
    0 & 0 & 0 \\
    0 & 0 & 0 \\
    0 & 0 & T(s)
  \end{array}\right) =
  \left(\begin{array}{ccc}
    T(s) & 0 & 0 \\
    0 & T(s) & 0 \\
    0 & 0 & T(s)
  \end{array}\right).
\end{align*}
The proof for a saddle $s:p(\ptc) \rightarrow q(\ptc)$ is similar.
\end{proof}

\end{appendix}

\section*{Acknowledgments}

The author would like to thank Matthew Hedden, Ciprian Manolescu, and
Peter Ozsv\'{a}th for helpful discussions.

\bibliographystyle{abbrv}
\bibliography{bib-v2-archive-khovanov-tangles-annulus}

\begin{thebibliography}{10}

\bibitem{Asaeda-1}
M.~Asaeda, J.~H. Przytycki, and A.~S. Sikora.
\newblock Categorification of the {Kauffman} bracket skein module of
  {$I$}-bundles over surfaces.
\newblock {\em Algebr. Geom. Topol.}, 4(2):1177--1210, 2004.

\bibitem{Boozer-Fukaya}
D.~Boozer.
\newblock Khovanov homology and the {Fukaya} category of the traceless
  character variety for the twice-punctured torus.
\newblock {\em arXiv preprint arXiv:2210.16452}, 2022.

\bibitem{Gabrovsek}
B.~Gabrov\v{s}ek.
\newblock The categorification of the {Kauffman} bracket skein module of
  {$\RP^3$}.
\newblock {\em Bull. Austral. Math. Soc.}, 88(3):407--422, 2013.

\bibitem{Hedden-1}
M.~Hedden, C.~Herald, and P.~Kirk.
\newblock The pillowcase and perturbations of traceless representations of knot
  groups.
\newblock {\em Geom. Topol.}, 18(1):211--287, 2014.

\bibitem{Hedden-2}
M.~Hedden, C.~Herald, and P.~Kirk.
\newblock The pillowcase and traceless representations of knot groups {II}: a
  {Lagrangian}-{Floer} theory in the pillowcase.
\newblock {\em J. Symplect. Geom.}, 18(3):721--815, 2018.

\bibitem{Hedden-3}
M.~Hedden, C.~M. Herald, M.~Hogancamp, and P.~Kirk.
\newblock The {F}ukaya category of the pillowcase, traceless character
  varieties, and {K}hovanov cohomology.
\newblock {\em Trans. Amer. Math. Soc.}, 373:8391--8437, 2020.

\bibitem{Jones}
V.~Jones.
\newblock A polynomial invariant for knots via von {Neumann} algebras.
\newblock {\em Bull. Amer. Math. Soc.}, 12:103--111, 1985.

\bibitem{Khovanov}
M.~Khovanov.
\newblock A categorification of the {Jones} polynomial.
\newblock {\em Duke Math. J.}, 101(3):359--426, 2000.

\bibitem{Kronheimer-2}
P.~Kronheimer and T.~Mrowka.
\newblock {Khovanov} homology is an unknot-detector.
\newblock {\em Publ. Math. Inst. Hautes \'{E}tudes Sci.}, 113:97--208, 2011.

\bibitem{Kronheimer-1}
P.~Kronheimer and T.~Mrowka.
\newblock Knot homology groups from instantons.
\newblock {\em J. Topol.}, 4:835--918, 2011.

\bibitem{Kronheimer-3}
P.~Kronheimer and T.~Mrowka.
\newblock Filtrations on instanton homology.
\newblock {\em Quantum Topol.}, 5:61--97, 2014.

\bibitem{Rozansky}
L.~Rozansky.
\newblock A categorification of the stable {$SU(2)$}
  {Witten}--{Reshetikhin}--{Turaev} invariant of links in {$S^2 \times S^1$}.
\newblock {\em arXiv preprint arXiv:1011.1958}, 2010.

\bibitem{Willis}
M.~Willis.
\newblock {Khovanov} homology for links in {$\#^r(S^2 \times S^1)$}.
\newblock {\em Michigan Math. J.}, 70(4):675--748, 2021.

\end{thebibliography}

\end{document}